\documentclass[11pt,reqno]{amsart}
\usepackage[active]{srcltx}
\usepackage{amssymb,color}
\usepackage{float}
\restylefloat{table}

\usepackage{ragged2e}
\usepackage{array}
\usepackage[dvips]{epsfig}
\usepackage{graphicx}

\newtheorem{thm}{Theorem}[section]
\newtheorem{cor}[thm]{Corollary}
\newtheorem{lem}[thm]{Lemma}
\newtheorem{prop}[thm]{Proposition}
\newtheorem{Note}[thm]{Note}

\theoremstyle{remark}

\newtheorem{defn}[thm]{Definition}

\newcommand{\om}{\omega}
\renewcommand{\L}{\mathcal{L}}

\newcommand{\la}{\lambda}
\newcommand{\al}{\alpha}

\newcommand{\fy}{\varphi}

\newcommand{\I}{\infty}

\newcommand{\R}{\mathbb{R}}
\newcommand{\C}{\mathbb{C}}

\renewcommand{\bar}{\overline}
\renewcommand{\hat}{\widehat}

\newcommand {\zn}{{\noindent}}

\newcommand{\EQ}[1]{\begin{equation} \begin{split} #1 \end{split} \end{equation}}
\newcommand{\Del}[1]{}

\newcommand{\de}{\delta}

\newcommand{\Q}{\mathbb{Q}}

\newcommand{\HH}{\mathcal{H}}

\def\({\left(}
\def\){\right)}

\def\nn{\nonumber}
\def\eps{\varepsilon}

\def\const{\mathrm{const}}
\def\lam{\lambda}

\def\f{\frac}
\def\ds{\displaystyle}

\def \h{\boldsymbol{\chi}}

\def\Ei{\mathrm{Ei}}
\def\L{\left[}
\def\J{\right]}

\begin{document}

\title{On the spectral properties of $L_{\pm}$ in three dimensions}

\author{Ovidiu Costin}

\author{Min Huang}

\author{Wilhelm Schlag}

\begin{abstract}
This paper is part of the radial asymptotic stability analysis of the ground state soliton
for either the cubic nonlinear Schr\"odinger or Klein-Gordon equations in three
dimensions. We demonstrate by a rigorous method that the linearized scalar operators which arise in this setting, traditionally denoted
by $L_{\pm}$, satisfy the {\em gap property}, at least over the radial functions.
This means that the interval $(0,1]$ does not contain any eigenvalues of $L_{\pm}$
and that the threshold~$1$ is neither an eigenvalue nor a resonance. The gap property is required in  order to
prove {\em scattering to the ground states} for  solutions starting on the center-stable
manifold associated with these states. This paper therefore provides the final installment in the proof of this scattering
property for the cubic Klein-Gordon and Schr\"odinger equations  in the radial case, see the recent theory of Nakanishi and the third author, as well
as the earlier work of the third author and Beceanu on NLS.
The method developed here is quite general, and applicable to other spectral problems which arise in the theory of nonlinear equations.
\end{abstract}

\thanks{    The first author was partially  support by NSF DMS-0600369 and NSF DMS-0807266. The third author was partially supported by NSF DMS-0617854,
and he thanks the Mathematisches Forschungsinstitut at ETH Z\"urich for its hospitality during the Fall of~2010 where part of this work was conducted. 
}

\maketitle

\section{Introduction}

\subsection{The nonlinear context}

Consider the nonlinear Schr\"odinger equation
\begin{equation}\label{eq:NLS}
i\partial_{t}\psi - \Delta\psi= \pm |\psi|^{p-1}\psi \qquad (t,x)\in \R^{1+d}_{t,x}
\end{equation}
with powers $ 1<p< 2^{*}-1$ where $2^{*}=\frac{2d}{d-2}$ in dimensions $d\geqslant 3$ and $2^{*}=\infty$
in dimensions $d=1,2$. Assuming that $\psi(t,x)$ is a smooth solution of sufficient spatial decay,
one verifies by differentiating under the integral sign that  {\em mass and energy} are conserved:
\begin{align*}
M[\psi(t)] &:= \frac12 \|\psi(t)\|_{2}^{2}= M[\psi(0)] \\
E[\psi(t)] &:= \int_{\R^d} \Big(\frac12 |\nabla \psi(t)|^{2} \mp \frac{1}{p+1}|\psi(t)|^{p+1}\Big)\, dx = E[\psi(0)]
\end{align*}
The range $p<2^{*}-1$ is referred to as {\em energy subcritical} regime due to the fact that in the conserved
energy the nonlinear term $\|\psi\|_{p+1}^{p+1}$ is controlled by the $H^{1}$-norm of $\psi(t)$ via Sobolev
embedding.

The choice of sign in front of the nonlinearity of~\eqref{eq:NLS} is crucial: the $-$ sign
(known as the defocusing nonlinearity) leads to a positive definite conserved energy and one has
global existence and scattering to a free wave for any data in~$H^{1}(\R^{d})$, see \cite{Caz}, \cite{SS99}, and~\cite{Tao} for an account of these classical
results. {\em Scattering} here refers to the property
that there exists $\psi_{0}\in H^{1}$ so that with the associated free solution  $\psi_{0}(t):=e^{-it\Delta}\psi_{0}$
\[
\| \psi(t) - \psi_{0}(t)\|_{H^{1}}\to 0 \qquad t\to\infty
\]
On the other hand, the focusing sign $+|\psi|^{p-1}\psi$ on the nonlinearity renders the energy indefinite and finite-time blowup may occur,
for example for all data of negative energy and finite variance, see Glassey~\cite{Glassey}. {\em Blowup } here refers to the property that $\| \psi(t)\|_{H^{1}}\to\infty$ as $t\to T- <\infty$.
In addition, the focusing nonlinearity admits special stationary wave solutions of the form $e^{-it\alpha^{2}}\phi(x)$ with $\alpha\ne0$,  where
\begin{equation}\label{eq:semi linear}
-\alpha^{2}\phi + \Delta \phi = |\phi|^{p-1}\phi
\end{equation}
Existence of nontrivial decaying solutions to this equation has been known for a long time, see for example Strauss~\cite{Strauss}
and Berestycki, Lions~\cite{BerLions}. On the one-dimensional line, there are exactly two nonzero decaying solutions which are
given by
\[
Q(x) =  \pm \alpha \cosh^{-\frac{1}{\beta}}(\beta x),\qquad \al=\big(\frac{p+1}{2}\big)^{\frac{1}{p-1}}, \;\; \beta=\frac{p-1}{2}
\]
The existence and uniqueness  can be read off from the phase-portrait in the $(\phi,\phi')$-plane. In higher dimensions no explicit
formulas exist and one obtains existence via variational arguments. Moreover,  uniqueness in the strong sense as in one dimension fails,
as it is know that there are infinitely many solutions~\cite{BerLions}. However, there is exactly one positive, radial solution called the ground state.
In fact, any positive decaying solution of~\eqref{eq:semi linear} is necessarily radial about some point (by Gidas, Ni, Nirenberg~\cite{GNN}) as well
as exponentially decaying. This unique solution is called {\em ground state} and it is the one we consider in this paper.

The orbital stability analysis of this ground state standing wave was settled many years ago and depends on the power of the nonlinearity,
see~\cite{BerLions}, \cite{CazLions},  \cite{GSS}, \cite{Wein}, \cite{Wein2}:
for $p< p_{2}:=\frac{4}{d}+1$ (the latter being called  $L^{2}$-critical power)  the ground state is {\em stable}, whereas for $p_{2}\leqslant p< 2^{*}-1$
the ground state is {\em unstable} in the orbital sense. In fact, the instability is very strong: arbitrarily small perturbations of initial data $Q$ with respect to
the~$H^{1}$-topology can lead to finite-time blowup, see~\cite{BerCaz}, \cite{Caz}, \cite{SS99} (and for Klein-Gordon~\cite{PS}). The transition at the power $p_{2}$ can
be seen at the linearized level. More precisely, linearizing about the standing wave with $\alpha=1$ (which we may assume by scaling) and splitting into real
and imaginary parts yields the matrix operator $\HH:=\left[\begin{matrix} 0 & L_{-} \\ -L_{+} & 0\end{matrix}\right]$ where
\[
L_{-} = -\Delta+1 -  Q^{p-1},\qquad L_{+} = -\Delta +1 - p Q^{p-1}
\]
One then finds that for $p\leqslant p_{2}$ the spectrum of $\HH$ lies entirely on the real axis, whereas for $p>p_{2}$ there
exists a pair of imaginary simple eigenvalues, signifying exponential linear instability.

The more difficult {\em asymptotic stability problem} was considered in Buslaev, Perelman~\cite{BP1}, \cite{BP2},  Soffer, Weinstein~\cite{SofWei1}, Cuccagna~\cite{Cucc}, but it would take us too far
to review the literature on this topic. More relevant for our purposes is the {\em conditional asymptotic stability problem} which refers to the following question:  in the {\em orbitally unstable regime},
does the ground state remain asymptotically
stable in forward time  under a finite co-dimension condition on the perturbation?
 In fact, due to the structure of the spectrum in that case one might expect that a co-dimension~$1$ condition should suffice
in order to stabilize the ground state. This is indeed the case, as shown in the orbital stability sense by Bates and Jones~\cite{BJ} for the nonlinear Klein-Gordon equation and for the NLS equation in~\cite{GJLS} (following Bates and Jones).
\cite{BJ}~implemented the Hadamard (or graph transform) method in the infinite-dimensional setting given by nonlinear dispersive Hamiltonian PDEs such as NLS and Klein-Gordon.
The graph transform together with the Lyapunov-Perron fixed point approach constitute the only two known methods available for  the construction of invariant manifolds, and they were both intensely developed
in finite dimensions (in other words, for ODEs). See the introduction of~\cite{NakS} and the references cited there.

The asymptotic stability question for the cubic NLS in three dimensions was studied in~\cite{S} and~\cite{Bec2}, where the existence of a center-stable
manifold near the ground state was established on which the solutions remain {\em asymptotically stable and scatters to the ground state}.  See~\cite{KrS1}
for the one-dimensional case, and~\cite{NakS}, \cite{StefStan} for the Klein-Gordon equation.
Finally, in the monograph by Nakanishi and the third author~\cite{NakS} (see the references there for the original papers) it was shown that these center-stable manifolds divide a small ball around
the ground state into two halves which respectively give rise to  blowup in finite positive time on the one hand, and global existence in forward time and
scattering to zero, on the other hand.

The most delicate part of the conditional asymptotic stability analysis turns out to be the {\em scattering property} of solutions starting on the center-stable
manifold.    This refers to the fact that solutions starting on the manifold decompose into a ground state standing wave (with slightly different parameters - this is the phenomenon of {\em modulation})
plus a free wave plus a term which is $o(1)$ in the energy space as $t\to+\infty$.  By some dispersive PDE machinery  this is equivalent to the property that the perturbation of the modulated standing wave
satisfies {\em global dispersive estimates}, such as pointwise decay (as in~\cite{S}) or Strichartz estimates as in~\cite{Bec2}.

The reason that such global dispersive estimates can be considered  delicate lies with the fact that they appear to require detailed knowledge of the entire spectrum
of the linearized operator, including the behavior of the resolvent at the thresholds
of the essential spectrum.
In~\cite{S}, \cite{Bec2} one therefore needed to assume the gap property of $L_{\pm}$ for the cubic power nonlinearity (i.e., $p=3$) in~$\R^{3}$. As already mentioned before,
this refers to the fact that $L_{\pm}$ have no eigenvalues in $(0,1]$ and that $1$ is not a threshold resonance. And finally, the gap property which
we verify in this paper {\em implies} via the Lyapunov-Perron method that solutions on the center-stable manifold scatter to the ground state, see~\cite{NakS}.

Demanet and the third author~\cite{DS} implemented the {\em Birman-Schwinger} method numerically and showed that this assumption is indeed correct (in the general nonradial setting). Moreover,
they found that the gap property is even more delicate than expected: it fails if the power~$p$ on the nonlinearity is lowered slightly below~$p=3$.
This is somewhat surprising, as
 Krieger and the third author~\cite{KrS1}, based on Perelman~\cite{P}, had shown by analytical arguments  that the gap property holds in the entire $L^{2}$-supercritical regime in one dimension.
 This was facilitated by the explicit form of
the ground state in dimension~$1$ and one finds, moreover, that $L_\pm$ retain the gap property for all powers down to the completely integrable
cubic NLS, where a threshold resonance appears.

While \cite{DS} appears
  to be accurate on all empirical accounts, the numerical method implemented there
is not a proof since it seems very difficult --- if not impossible --- to give  rigorous  error bounds for all numerical approximations and calculations required
by the Birman-Schwinger method.  For example, an approximate soliton is computed numerically, but without any rigorous  bounds on the error introduced
by this approximation.

The purpose of this paper is  to offer {\em a completely rigorous}, albeit rather computational, proof
which confirms the gap property of $L_{\pm}$ in the cubic radial case in~$\R^{3}$.

\subsection{The main equations}

Over the radial functions, the equation for the ground state reduces to
\EQ{\label{mainODE}
  -y''(r)-\f{2}{r}y'(r)+y(r)-y^{3}(r)=0
}
By Coffman's theorem~\cite{Coff} there is a unique, positive decaying solution of \eqref{mainODE}
which is smooth on $[0,\I)$. It is denoted by $Q$ and called the {\em ground state}.
The eigenvalue problem over the radial subspace  now becomes
\begin{equation}
  \label{eqorL}
  L_+ u=\lambda u;\ \ \text{where}\ \  L_{+}=-\frac{d^{2}}{dr^{2}}-\frac{2}{r}\frac{d}{dr}+1-3Q^{2}
\end{equation}
and
\begin{equation}
  \label{eq:eqlm}
  L_- y=\lambda y;\ \ \text{where}\ \  L_{-}=-\frac{d^{2}}{dr^{2}}-\frac{2}{r}\frac{d}{dr}+1-Q^{2}
\end{equation}
We show that in the gap $[0,1]$ the operator~$L_+$ has no eigenvalues or resonance,
and the same is true for $L_-$ on $(0,1]$.

\subsection{Technical approach} \label{tech}
Eq.~\eqref{mainODE} is likely nonintegrable, and no useful
closed form representation  for $Q$ is known. However, in order to resolve the
aforementioned gap problem, an exact expression of
$Q$ or even the exact values of $Q(r)$ are clearly not required: a
sufficiently accurate approximation, which we denote by~$\tilde{Q}$, suffices.
Unfortunately, we have found that the required accuracy  for $\|Q/\tilde{Q}-1\|_\I$ in our problem is on the order of~$10^{-4}$.
This is a reflection of the phenomenon seen in~\cite{DS}, namely that the gap property is only {\em barely} correct.
More mathematically, it must mean that $L_{\pm}$ have complex resonances very close to the real axis that become eigenvalues
in the gap once the power is lowered slightly below~$p=3$. We remark that this phenomenon may also account for the failure of
``softer'' approaches based on bounds on the number of eigenvalues etc.

 We proceed as
follows. We find a suitable approximation~$\tilde{Q}$ in the form of a piecewise explicit function;
for $r\geqslant 5/2$ it is given by
\begin{equation}
  \label{defy3}
  \tilde{Q}(r)=  y_{3}(r;{\beta})=   r^{-1} \beta_{1} e^{-r} + \beta_{2}g(r);\ \
  g(r):=  r^{-1}(2e^{r}\Ei(-4r)-e^{-r}\Ei(-2r))
\end{equation}
for specific $\beta_{1},\beta_{2}$, see \S\ref{S21}. On $[0,5/2)$, the reciprocal  $ \frac{1}{\tilde{Q}}$ is a piecewise
polynomial. The coefficients of the polynomials arising in this construction are listed in the first two lines of the table on page~\pageref{table:1}, and
we refer the reader to Section~\ref{ApQ} for more details of the construction, see especially Definition~\ref{def:tildeQ}.

The representation of $\tilde{Q}$ described above is found in the
following manner. For $r\geqslant 5/2$ we iterate the Volterra
equation once to obtain $y_3$. On $[0,5/2)$, we construct by Taylor
series a solution which is well-behaved at zero and then we extend it
by matched Taylor series up to $5/2$. We take the value $\tilde{Q}(0)$ as a
parameter and determine it so that $\tilde{Q}(5/2^-)=y_3(5/2^+)$.  We
then optimize the polynomial representation by sampling points from the
collection of {\em reciprocals} of the aforementioned Taylor series and using least squares fitting with three polynomials\footnote{ \label{ftn}(i) This is essentially the discrete version of $L^2([a,b])$ orthogonal projection using Legendre polynomials. Projecting on Chebyshev polynomials would
  provide an even more economical representation, but we prefer the
  simplicity of least squares fitting; (ii) we look at
  the reciprocals since they are smoother in that the singularities in
  $\C$ are farther away from the real axis and  result in more efficient representations.}.
  We then rationalize the coefficients of
the polynomials by suitably accurate truncated continued fractions. This is the procedure by which
 the rational numbers listed in the tables at the end of the paper are obtained.

The next step is to show that $\tilde{Q}$ is close to $Q$. At this
stage the problem is already, in some sense, perturbative: $\delta
=Q-\tilde{Q}$ is very small. We thus proceed  in a natural way, by solving
a contractive equation for~$\delta$. A slight
hurdle arises at this point since the Green's function $G(r,r')$ in the integral equation for
$\delta$ is not explicit either: $G$ solves a linear second order ODE
with nontrivial coefficients (combinations of exponential integrals
and polynomials). We overcome this by finding a nearby {\em equation} with
explicit solutions, and contract out the difference between the two
equations.

Estimating the remainder as a result of replacing $Q$ by $\tilde{Q}$
in~\eqref{mainODE} reduces to bounding rational functions with
rational coefficients.  This is done rigorously, as the degrees of the
polynomials in the denominator and numerator are manageable. There are
many ways to estimate polynomials. Perhaps the most straightforward one
is to place absolute values on all coefficients, or at least on all coefficients of powers
higher than three, say: a polynomial with positive coefficients is easy to bound by
monotonicity; it is largest at the largest argument. However, inspection of the tables at the end reveals that
this method cannot be applied here, as the coefficients of higher powers are not within the small interval we need (which is $<10^{-4}$).
In order to overcome this problem, we
re-expand the polynomials at a number of intermediate points selected
so that the coefficients of the monomials of degree exceeding~$3$ are
small enough to be discarded modulo small errors. Polynomials of
degree~$3$ of course have explicit extrema. The re-expansion refers to nothing other than passing from $P(r)\in\Q[r]$ to the new
polynomial in $\Q[z]$ defined as $P(r_0+az)$, where we always keep $|z|\leqslant 1$. The values of~$r_0$ and~$a$ are always stated
explicitly in the text (we use the notation of  ``partitions" for this purpose). See Note~\ref{met1} for more on this issue.

After having obtained $Q$ up to explicitly controlled errors, we then analyze the
spectra of~$L_{\pm}$. This is done essentially in the same way, by
finding an accurate Jost quasi-solution for $r\geqslant 5/2$, and a
well-behaved one on $[0,5/2]$ (which simply refers to the requirement that
the solution remains bounded with a horizontal tangent at the origin). The quasi-solutions are
explicit combinations of exponential integrals and polynomials. The
way the quasi-solutions are obtained also mirrors the soliton
approach: iterating the Volterra equation for large $r$ and using
orthogonally projected Taylor series in the complement. This time
around, we need the solution $u_1$ only on $[0,5/2]$ and $u_2$ on
$[5/2,\infty]$. However, since the quasi-solutions depend on the
spectral parameter $\lambda$ as well, the calculations are more
involved.  We then check that $\inf_{\lambda\in [0,1]} |W(\lambda)|>0$ (for
$L^+$) ($\inf_{\lambda\in [0,1]} |W(\lambda)/\lambda|>0$ for $L^-,$ resp.)
where $W$ is the Wronskian of $u_1,u_2$ at $5/2$, and this concludes
the proof\footnote{A solution bounded near $0$ must be a multiple of
  $u_1$ since solutions linearly independent of $u_1$ are unbounded
  at zero; similarly any solution bounded as $r\to \infty$ is a
  multiple of $u_2$.}.

We emphasize that all coefficients involved are in $\mathbb{Q}$, the
calculations are exact, and the proof, tedious at places -- for
instance in having to repeatedly re-expand polynomials at numerous intermediate points --  is fully
rigorous.

In addition,  the integral operators upon which the contractive mappings are
based have small norm (see, e.g., \eqref{eqint}), allowing for the
calculation of the solutions rapidly and, in principle, with arbitrary
accuracy. Therefore,
this approach is useful numerically as well, to obtain rapidly convergent approximants.

While our approach can {\em in principle} be carried out by hand, it
is of course unrealistic to attempt this in praxis as the calculations in their current form
are too numerous as well as too long.     
While we organized Section~\ref{ApQ} in such a way that the calculations can be in practice
done by hand,  in later sections we preferred to use the computer algebra
packages {\tt Maple} and {\tt Mathematica} to perform basic operations
(such as multiplications of polynomials with rational coefficients and
solving quadratic equations). The later sections involve longer calculations, but of the
simple type mentioned above. The exponential integral, Ei,
is the only transcendental function needed; we estimate it  using
the asymptotic inequalities it obeys and/or  by integrating  inequalities
satisfied by its derivative, an elementary function. Once more, there are {\em no numerical}
calculations involved (such as numerical integration or numerical location of zeroes), and with substantially more optimization
effort, it is likely that every step could have been done by hand; we
felt there is little to gain from this, as human error has a
considerably higher chance to occur in such a setting.


We wish to emphasize again that all  calculations that were carried out by either {\tt Maple} or {\tt Mathematica} are
completely error free  as they only involve finitely many algebraic operations in the polynomial ring~$\Q[r]$.

Let us also emphasize that the concrete implementation of the method as it appears below is by no means the only
possible one. As mentioned in Footnote~\ref{ftn} on the previous page, one may substantially reduce the amount of computations required
(as well as the length of the tables in the appendix) by relying on Chebyschev polynomials instead of least squares fitting in order to
carry out the aforementioned projections of the matched Taylor series. We intend to present this simpler implementation, together
with the {\em nonradial gap property} in a future publication.

Finally, we would like to mention that the approach developed in this paper is by no means restricted to the specific
problem that we study here.  For example, in a forthcoming publication a similar approach will be used to settle
a long-standing spectral question arising in completely integrable equations (pertaining to a Painlev\'e equation).


\section{The approximate soliton}\label{ApQ}

In this section we  find an approximation of $Q$ by means of simple functions.
 Let  $$r_{1}:=3/10,\;\; r_{2}:=17/25, \;\; r_{3}:=9/10, \;\; r_{4}:=1, \;\; r_{5}:=3/2,\;\; r_{6}:=12/5,\;\; r_{7}:=5/2$$
 These increasing numbers define the partition points relative to which we will define the piecewise approximations.
  We denote the characteristic function $\boldsymbol{\chi}([r_i,r_j))$ by  $\boldsymbol{\chi}_{ij}, j\ne 7$ and  $\boldsymbol{\chi}_{i7}=\boldsymbol{\chi}([r_i,r_7])$; similarly  $\boldsymbol{\chi}([0,r_j))=: \boldsymbol{\chi}_{0j}$ and $\boldsymbol{\chi}([r_j,\infty))=: \boldsymbol{\chi}_{j\infty}$. We also denote $\boldsymbol{\chi}_{\overline{ij}}=\boldsymbol{\chi}([r_i,r_j])$.

\subsection{Solving \eqref{mainODE} from $\I$}\label{S21}

The following lemma describes all possible solutions of \eqref{mainODE} which decay
at~$\I$ (at least within a certain range of parameters chosen to suit our needs, and  up to some error).
 In what follows, we shall repeatedly encounter the {\em exponential integral }
\[
\Ei(x) := PP\int_{-\infty}^{x} \frac{e^{u}}{u}\, du; \ \ x\in\R
\]
where $PP$ denotes the Cauchy principal value of the integral. Define the nonlinear operator
\begin{equation}
  \label{eq:defn}
(\mathcal{N}(f))(r)=   \int_{r}^{\I} \sinh(r-s) s^{-2}f^{3}(s)\, ds
\end{equation}
With $g$ as in \eqref{defy3} one checks that
\begin{equation}
  \label{eq:formg}
  rg(r)=\mathcal{N}(e^{-r})
\end{equation}

The exponential integral admits the following asymptotic expansions.

\begin{lem}\label{N1}
(i) For each positive integer $N$ one has
\EQ{\label{Ei1 bd}
 e^{-x}\sum_{k=0}^{2N-1} k! \f{(-1)^k}{x^{k+1}}  < -\Ei(-x) <  e^{-x}\sum_{k=0}^{2N} k! \f{(-1)^k}{x^{k+1}} \quad \forall\; x>0
}
and
\begin{equation}
  \label{4p}
  \begin{aligned}
0>g(r) &=  \frac{e^{-3r}}{r^3}\left(-\frac{1}{8} +\frac{3}{16r}-\frac{21}{64r^2}+\frac{45}{64 r^3}\right)+\frac{15e^r}{64r}\int_r^\I(e^{-4s}-16e^{-2(s+r)})\,\frac{ ds}{s^6} \\
& >  \frac{e^{-3r}}{r^3}\left(-\frac{1}{8} +\frac{3}{16r}-\frac{21}{64r^2}+\frac{45}{64 r^3}  - \frac{465}{256 r^4} \right)
\end{aligned}
\end{equation}
 In particular $0>g(r)>-{e^{-3r}}/{8r^3}$ for $r\geqslant r_7$.

\zn (ii) The function  $h(r)=-re^rg(r)$ is positive and decreasing.

\noindent (iii) Define the norm $\|\psi\|=\sup_{r\geqslant r_{7}}|\psi(r)e^r|$ on continuous functions on $[r_7,\I)$. Then we have
\begin{equation}
  \label{est11}
 \|\mathcal{N}(e^{-r})\|\leqslant 1/8900
\end{equation}
\end{lem}
\begin{proof}
(i) Both \eqref{Ei1 bd} and \eqref{4p} follow by means of repeated integrations by parts. (ii) $h$
is manifestly positive, while
$$h'(r)=-e^{2r}\int_r^\I e^{-4s}s^{-2}\,ds<0$$ (iii) By (ii),  $\sup_{r\geqslant r_{7}}| e^r\mathcal{N}(e^{-r})|$ is reached at $r=r_{7}$
and \eqref{est11} is now immediate from \eqref{4p}.
\end{proof}

Henceforth we assume $0<{\beta}\leqslant 3$, which is sufficient for our purposes.
Also, recall the definition of $y_3$, see~\eqref{defy3}:
\[
y_{3}(r;{\beta})=   r^{-1}{\beta}e^{-r} + {\beta}^{3}g(r)
\]
By the lemma, this defines a positive function for all $0<\beta\leqslant 3$ and $r\geqslant  r_7$.

\begin{lem}
\label{lem:infty_asymptotics}
There exists a unique positive solution $y(r;{\beta})$ to~\eqref{mainODE}
with the property that $y(r;{\beta})\sim {\beta}r^{-1}e^{-r}$ as $r\to\I$. It satisfies
\EQ{
\label{yAasymp}
\Big| \frac{y(r;{\beta})}{y_{3}(r;{\beta})}-1\Big | &<  3.2\cdot  10^{-6}\qquad\forall\;r\geqslant r_{7}
 }
 uniformly in $0<\beta\leqslant 3$.
\end{lem}
\begin{proof}
Setting $z(r)=ry(r)$ (and suppressing the ${\beta}$-dependence for notational convenience) yields the ODE
\EQ{\label{ODEmod}
-z''(r)+z(r)=r^{-2} z^{3}(r)
}
Let $z_{0}(r)={\beta}e^{-r}$. For solutions $z$ with  $z(r)e^r$ bounded for large $r$,  \eqref{ODEmod} can be written as
\EQ{ \label{eqint}
z(r) = z_{0 }(r) +\mathcal{N}(z)(r)=:\mathcal{M}(z)(r)\quad r>0
}
Take the ball $B=\{h:\|h\|\leqslant \alpha {\beta}\}$ and choose  $\alpha$ so that
$\mathcal{M}B\subset B$; the latter condition gives $$1+{\beta}^2\alpha^3\|\mathcal{N}\|-\alpha\leqslant 0$$ which is satisfied for ${\beta}\in[0,3]$ if $\alpha = 1+1/985$, for example. Using \eqref{est11} again,  expanding out $(z+\delta)^3-z^3=\delta(3z^2+3z\delta+\delta^2)$ and estimating each term in the last parenthesis by its largest norm in $B$ (such as $\|\delta\|\leqslant 2\alpha\beta$) we get
\begin{equation}
  \label{NrmN}
  \left\|\mathcal{M}(z+\delta)-\mathcal{M}(z)\right\|\leqslant \frac{13\alpha^2 {\beta}^2\|\delta\|}{8900} <  \frac{\|\delta\|}{76}
\end{equation}
Thus $\mathcal{M}$ is contractive in $B$ and  \eqref{eqint} has a unique solution  $z_s$ there. First, since  $\mathcal{N}(e^{-r})=O(e^{-3r})$ for large $r$,  we have $y(r;{\beta})\sim {\beta}r^{-1}e^{-r}$, as claimed. It remains to show \eqref{yAasymp}. Note that $r y_3(r;{\beta})=\mathcal{M}(z_0)(r)$.  Since $\delta:=z_s-z_0=\mathcal{N}(z_s)$
we have
 $$\|\delta\|\leqslant \|\mathcal{N}(z_s)\|\leqslant {\beta}^3\alpha^3\|\mathcal{N}(e^{-r})\| < \f{\alpha\beta}{500}  $$
 Using this estimate to improve on~\eqref{NrmN} we conclude that
\begin{equation}
  \label{ratio1}
\Big| \frac{y(r;{\beta})}{y_{3}(r;{\beta})}-1\Big |=  \left|\frac{z_s-\mathcal{M}(z_0)}{\mathcal{M}(z_0)}\right|\leqslant
\frac{1}{327}\frac{{\beta}^3\alpha^3\|\mathcal{N}(e^{-r})\|}{{\beta}-{\beta}^3\|\mathcal{N}(e^{-r})\|}< 3.2\cdot  10^{-6}
\end{equation}
as claimed.
\end{proof}
\subsection{Solving \eqref{mainODE} on $[0,\I)$ up to a small error}

In this section, we study  an approximate solution of~\eqref{mainODE}. For the heuristics behind
this construction, we refer the reader to the introduction.
In particular, we specialize the value of ${\beta}$ in $y_{3}(r;{\beta})$ in Lemma~\ref{lem:infty_asymptotics}
so as to most closely approximate the ground state $Q(r)$.

\begin{defn} \label{def:tildeQ}
The approximate soliton $\tilde Q$ is defined as follows. First, set
 \begin{equation}
\label{p1p2}
\begin{aligned}
p_1(r) &:= q_1(r)    \\
p_2(r) &:= q_2(r)+p_1(r_{3})-q_2(r_{3})+(p_1'(r_{3})-q_2'(r_{3}))(r-r_{3}) \\
p_3(r) &:= \f{r}{ Ae^{-r} + Brg(r) }
\end{aligned}
\end{equation}
where $q_{1}, q_{2}$ are the explicit polynomials  $q_{\ell}(r)=\sum_{j=0}^{11} a_{j}^{\ell} \, r^{j}$ with coefficients as in Table~\ref{table:1},
and $A,B$ are chosen so that $p_2, p_3$ match up in a $C^1$ fashion at $r=r_{7}$; $p_3$ is very close to the solution in Lemma~\ref{lem:infty_asymptotics}, see  \eqref{A B error} below. Finally, set
 \begin{equation}
  \label{tildeQ}
\tilde Q:={\boldsymbol{\chi}_{03}}/{p_1}+{\boldsymbol{\chi}_{37}}/{p_2}+{\boldsymbol{\chi}_{7\infty}}/{p_3}
\end{equation}
\end{defn}

\bigskip
We remark that with $f(r):= \frac{e^{-r}}{r}$  the coefficients $A,B$ in~\eqref{p1p2} are
\EQ{\label{A B}
A &:=  \f{p_2(r_{7})g'(r_{7})+p_2'(r_{7})g(r_{7})}{p_2^2(r_{7})(f(r_{7})g'(r_{7})-f'(r_{7})g(r_{7}))} \\
B &:= - \f{p_2(r_{7})f'(r_{7})+p_2'(r_{7})f(r_{7})}{p_2^2(r_{7})(f(r_{7})g'(r_{7})-f'(r_{7})g(r_{7}))}
}
The specific form of $p_3$ of course originates with the exact solution from~\eqref{yAasymp}.  From \eqref{A B} one verifies that
\EQ{\label{A B error}
\frac{217}{80}<A<\frac{350}{129},\quad
|B-A^3| < 33\cdot 10^{-5};\ B <20
}
We now need to show that $\tilde Q$ from~\eqref{tildeQ}    is close to the actual unique ground state $Q$.
We begin by checking that $\tilde Q$ satisfies~\eqref{mainODE} up to a small error. Below we denote by $C^2_p$ the space of piecewise $C^2$ functions.
\begin{lem}
\label{lem:tildeQ1}
As defined above, $\tilde Q$ satisfies the following properties:
\begin{enumerate}
\item  It is decreasing for $r\in[0,r_{7}]$, and $0<\tilde Q(r)<{22}/{5}$.
\item  It belongs to  $C^{1}([0,\I))\cap C^2_p([0,\I))$,  and $\tilde Q'(0)=0$.
\item  It satisfies the bounds
\EQ{
\tilde Q(r) < 5(1+r)e^{-2r}\boldsymbol{\chi}_{02}(r)+\frac{11}{2}e^{-8r/5}\boldsymbol{\chi}_{27}(r)  \qquad \forall\, r>0
}
and
\begin{equation}
  \label{eq:bdQ2}
  \f{187}{69}\f{e^{-r}}{r}<\tilde Q(r)<\f{350}{129}\frac{e^{-r}}{r},\ \ \text{for $r\geqslant r_{7}$}
\end{equation}
\item \label{item4} In the complement of the three-point set  $\{r_{3},r_4, r_{7}\}$ the error
\[
R(r):=  -\tilde Q''(r)-\f{2}{r}\tilde Q'(r)+\tilde Q(r)-\tilde Q(r)^{3}
\]
satisfies the bound
\EQ{\label{diffeq error}
 |R(r) |< \rho_1({11}/{10}-r)\boldsymbol{\chi}_{04}(r)+\rho_2({13}/{5}-r)\boldsymbol{\chi}_{47}(r)+\frac{e^{-3r}}{25r^{3}}\boldsymbol{\chi}_{7\infty}(r)
 }
where $\rho_1:=15\cdot 10^{-6}$ and $\rho_2:=25\cdot 10^{-8}$.
\end{enumerate}
\end{lem}

\begin{Note} \label{met1}{\rm (i)
 To estimate a higher order polynomial $P$  on an interval, we partition the interval and Taylor-re-expand $P$ in each subinterval.
 We define the partitions so that $P$ equals the first four terms  plus a small error. There are of course other
ways to estimate polynomials, but this method leads to   straightforward calculations.  We represent the partition by a vector $\boldsymbol{\pi}$, whose components $\pi_i$
are precisely the partition points. Unless otherwise specified,  in each interval $[\pi_i,\pi_{i+1})$  we shall write $$\ell_i(z)=\frac{1}{2}(1-z)\pi_i+\frac{1}{2}(1+z)\pi_{i+1},\quad -1\leqslant z \leqslant 1$$ and re-expand
$P(\ell_i(z))$ around $z=0$.

(ii) We bound  a polynomial $P(z)$ from below on $|z|\leqslant 1$ by the minimum of the
cubic polynomial $P\mod O(z^4)$ minus the sum of the absolute value
of the remaining coefficients. Likewise, to obtain an upper bound, we
take the maximum of the cubic polynomial $P$ mod $O(z^4)$ plus the
sum of the absolute value of the remaining coefficients.}
  \end{Note}

\begin{proof}[Proof of Lemma~\ref{lem:tildeQ1}]
For property (i) we first note by inspection that $p_1'(r)$ has the following property: the coefficient of its first term ($r$ term) is bigger than 1 while the sum of absolute values of the remaining coefficients is less than 1, implying that $p_1'(r)>0$ for $r\in[0,r_{3}]$. Since obviously $p_1(0)>0$, we see that $1/p_1(r)$ is decreasing and positive for $r\in [0,r_{3}]$. All  coefficients of $p_2'(r_{3}+z)$ are positive. Thus $1/p_2(r)$ is decreasing for $r\geqslant r_{3}$. In particular $\tilde Q(r)\leqslant \tilde Q(0)<{22}/{5}$.

Property (ii) is immediate by construction.  Indeed, note that \eqref{p1p2} defines a $\tilde Q\in C^2_p$
in such a way that the values of the function and the value of its first derivative match up at $r_{3}, r_{7}$.
The vanishing $\tilde Q'(0)=0$ is a consequence of the fact that $p_{1}(r)$ has no linear component ($a_{1}^{1}=0$, see Table~\ref{table:1}).

For (iii) we start with the interval $[0,r_{2}]$, where we will show
$$(1+r)e^{-2r}p_{1}(r)>1/5$$
Instead of showing this directly, we first notice that
$$e^{-2r}\left(p_1(0)+\f{17r}{100}+r^2\right)>1/5$$
by explicitly finding the maximum via differentiation. Thus it is sufficient to show
$$(1+r)p_{1}(r)>p_1(0)+\f{17r}{100}+r^2$$
or equivalently $$p_{11}(r):=((1+r)p_1(r)-p_1(0))/r-\f{17}{100}-r>0$$

For this purpose we use the partition ${\boldsymbol{\pi}}=(0,r_{2})$ and  Note~\ref{met1} to re-expand
$p_{11}$ in $z$. The coefficients of $z^j$, $j> 3$, are all very small and
a direct calculation shows that $p_{11}>1/50>0$ on
$[0,r_{2}]$. Therefore $\tilde Q(r)<5(1+r)e^{-2r}$ for $r\in[0,r_{2})$.

Similarly, one can show $e^{-8r/5}p_{1}(r)$ is increasing on
$[r_{2},r_{3}]$ by using the partition
${\boldsymbol{\pi}}=(r_{2},r_{3})$ and estimating  $e^{8r/5}(e^{-8r/5}p_{1})'(r)$
using Note~\ref{met1}.

In the interval $[r_{3},r_{7}]$ we consider the polynomial $\tilde{p}_2(r):=e^{8r/5}(e^{-8r/5}p_2(r))''$. Using the partition ${\boldsymbol{\pi}}=(r_{3},r_{7})$ we obtain $\tilde{p}_2(r)<-1/5<0$ and thus $e^{-8r/5}p_2(r)$ is concave. This, combined with the fact that $e^{-8r_k/5}p_2(r_k)>2/11$ for $k=3,7$, shows that $\tilde Q(r)<\f{11}{2}e^{-8r/5}$ for $r\in[r_{3},r_{7}]$.

For $r\geqslant r_{7}$ we have using \eqref{Ei1 bd}
$$(re^{r}/p_3(r))'=Br^{-1}{e^{-2r}}-4Be^{2r}\Ei(4r)>0$$
Therefore $re^{r}/p_3(r)$ is increasing. Since obviously $\lim_{r\to
  \infty}re^{r}/p_3(r)=A<\f{350}{129}$ and
$r_{7}e^{r_{7}}/p_3(r_{7})>\f{187}{69}$, the estimate follows.

For (iv), we first introduce the notation $I_{1}:=[0,r_{3}]$ and $I_{2}:=[r_{3},r_{7}]$. For $j=1,2$ we let $y_{j}(r)=1/{p_{j}(r)}$ and define
\EQ{\nn
R_j(r) = - y_{j}''(r) -2r^{-1}y_{j}'(r) + y_{j}(r) - y_{j}^{3}(r),\quad r\in I_j
}
and  consider the polynomial $M_j(r)$ in $\Q[r]$  
given by
$$ M_j(r): = r^{j-1}p_j^3(r) R_j(r) $$
We introduce the partitions $\boldsymbol{\pi}_1=(0,\f{1}{10},\f{1}{5},\f{3}{10},\f{2}{5},\f{3}{5},\f{4}{5},\f{9}{10})$
and $\boldsymbol{\pi}_2=(\f{9}{10},\f{11}{10},\f{13}{10},\f{83}{50},\f{21}{10},\f52)$ and define
\EQ{   \label{M1k}
  M_{jk}(z) := M_j(\ell_{jk}(z));\ p_{jk}(z) := p_j(\ell_{jk}(z));\ j=1,2,\  |z|\leqslant 1. }
(where $\ell_{jk}$ is the $k$th component of $\ell_j$). We proceed to estimate $M_{jk}(z)$ 
on the unit disk as described in Note~\ref{met1}.

This yields, for each $1\leqslant k\leqslant 7$ and for any $-1\leqslant z\leqslant 1$,
\EQ{\label{M11M12}
 \left|\f{M_{1k}(z)}{p_{1k}(z)^3}\right| \leqslant \frac{\sup{|M_{1k}(z)|}}{p_{1k}(-1)^3}<\rho_1 \left(\f{11}{10}-\ell_{1k}(1)\right)\leqslant\rho_1 \left(\f{11}{10}-\ell_{1k}(z)\right) \qquad
}

On the interval $[r_3,r_4]$ we have
\begin{equation}
  \label{Rest0}
 \left|\f{M_{21}(z)}{\ell_{1k}(z)p_{21}(z)^3}\right| \leqslant \frac{10\sup{|M_{21}(z)|}}{9{p_{21}(-1)}^3}
 <\f{1}{10}\rho_1\leqslant \rho_1\left(\f{11}{10}-\ell_{21}(z)\right)
\end{equation}

On the interval $[r_4,r_7]$ we have for each $k=2,...,6$
\begin{equation}
  \label{Rest}
 \left|\f{M_{2k}(z)}{\ell_{2k}(z)p_{2k}(z)^3}\right| \leqslant \frac{\sup{|M_{2k}(z)|}}{{\ell_k(-1)p_{2k}(-1)}^3}
 <\rho_2\left(\f{13}{5}-\ell_{2k}(1)\right)\leqslant \rho_2\left(\f{13}{5}-\ell_{2k}(z)\right)
\end{equation}

These give the desired estimates of $|R(r)|$ for $r\in[0,r_{7}]$.

Finally, for $r\in I_3:=[r_{7},\I)$  we write, with $A$ and $g(r)$ as in \eqref{A B},
\[
y_3(r):=\f{1}{p_3(r)} =  y_3(r;A) +  (B-A^3) g(r) =\frac{A}{re^r}+A^3g(r) +
 (B-A^3) g(r)
\]
Then the error $R(r)$ has the form, with $\eps:=B-A^3$,
\[
R(r) = R_0(r) + \eps R_1(r) + \eps^2 R_2(r;\eps)
\]
where
\begin{align}\label{R0}
 & R_0(r)  :=  -  y_3''(r;A)-\f{2}{r}  y_3'(r;A)+ y_3(r;A)   -  y_3(r;A)^{3} \\
& R_1(r) := - \f{1}{r} (rg(r))''   + g(r) -3 y_3(r;A)^2 g(r)=\frac{e^{-3r}}{r^3}-3g(r)\left(A\frac{e^{-r}}{r}+A^3g(r)\right)^2\\
& R_2(r;\eps) := -3Ag^2(r)\left(\frac{e^{-r}}{r}+A^2g(r)+\frac{\eps g(r)}{3A}\right)
\end{align}
First, using Lemma~\ref{N1} we obtain $h(r_{7})\leqslant 9 \cdot 10^{-5}$ (with $h$ defined in that lemma) and thus
\EQ{
\label{R00}
|R_0(r)| &= \left|3\f{A^{5} e^{-3r}}{r^{3}}h(r)\left[1 -A^{2}h(r) + \f13 A^{4} h^{2}(r) \right] \right|\leqslant 3A^5h(\tfrac52)\frac{e^{-3r}}{r^3}
}
By \eqref{4p},  $0>g>-e^{-3r}/(8r^3)$ and thus, noting that $-r^{-1}(rg)''+g=e^{-3r}/r^3$, we get
\begin{equation}
  \label{R1}
|R_1(r)| < \frac{e^{-3r}}{r^3}\left(1+\frac{3A^2e^{-2r}}{8r^2}\right)<\f{11}{10}\frac{e^{-3r}}{r^3}\qquad\forall r\geqslant r_{7}
\end{equation}
Finally,
\EQ{\label{R2}
|R_2(r)|  < \frac{e^{-3r}}{r^3}\frac{3Ae^{-3r}}{64r^3}\left(\frac{e^{-r}}{r}+\frac{A^2e^{-3r}}{8 r^3}+\frac{\eps e^{-3r}}{8Ar^3}\right) < \f32\cdot 10^{-7}\frac{e^{-3r}}{r^3}\qquad \forall\; r\geqslant r_{7}
}
Combining \eqref{R0}, \eqref{R1} and \eqref{R2} with \eqref{A B} yields
\EQ{ \nn
|R(r)| &< \Big(3A^5h(5/2)+ 11 \eps/10 +  3\cdot 10^{-7}\eps^2/2\Big)\f{e^{-3r}}{r^3}  <   \f{1}{25} \; \f{e^{-3r}}{r^3}
}
for all $r\geqslant r_{7}$ as stated.
\end{proof}

\section{The exact soliton}\label{ExQ}

\subsection{Finding the exact ground state $Q$ near the approximate one $\tilde Q$}

We now need to show that $\big|1- \tilde Q/{Q} \big|$ is small. Eq.\ \eqref{mainODE} implies \EQ{
\label{de eq}
-\de''(r)  - 2\frac{\de'(r)}{r}  +(1-3\tilde Q^{2}(r)) \de(r) = -R(r) + 3\tilde Q(r)\de^{2}(r) + \de^{3}(r) \ \ (\delta:=Q-\tilde Q)
}
The boundary conditions are $\de'(0)=0$, $\de(\I)=0$.
We shall describe, again via polynomials, how to find an approximate fundamental system for~\eqref{de eq}.
The challenge here is of course that we cannot hope to find an exact fundamental system for this equation, but require
something close to it in order to set up a contraction for~$\delta$. We find it technically convenient to find an {\em exact}
fundamental system for a homogeneous equation which is slightly different from the one in~\eqref{de eq}.

\subsubsection{An approximate Green's function}
\begin{defn}
\label{Green Poly}
We define $\fy_{1}(r), \fy_{2}(r)$ as follows. Set  $J_{1}:=[0,r_{1})$, $J_{2}:=[r_{1}, r_{4})$,
$J_{3}:=[r_{4}, r_{7})$, $J_{4}:=[r_{7},\I)$ and let
\begin{equation}
  \fy_{j}(r) :=   q_{j+2}(3r)\h_{01}+q_{j+2}(r-1/2)\h_{14}+ q_{j+2}(r-2)\h_{47}+a_je^{\sigma_j r} \h_{7\infty}
 \end{equation}
 for $j=1,2$, where $a_1=1,a_2=1/2, \sigma_1=-1,\sigma_2=1$, $q_{3}, q_{4}$ are of the form
 $q_{j}(r) = \sum_{k=0}^{13} b_{k}^{j}\, r^{k}$ where the $ b_{k}^{j}$ are given in Table~\ref{table:1}.
 The factor $\f12$ in front of $e^{r}$ in the definition of $\fy_{2}$ is chosen so as to normalize a Wronskian to~$1$.
Finally, set $ g_{j}=\fy_{j}$ on $J_{4}$ and
\EQ{\label{gj def}
g_{j}(r) := \fy_{j}(r) + g_{j}(r_{\ell}'+) - \fy_{j}(r_{\ell}'-) + (g_{j}'(r_{\ell}'+) - \fy_{j}'(r_{\ell}'-))(r-r_{\ell}')  \quad \forall\; r\in J_{\ell-1}
}
for all $\ell=2,3,4$, where $J_{\ell-1}=[r_{\ell-1}',r_\ell')$ (we use this notation only for~\eqref{gj def}).
\end{defn}

We remark that the jumps appearing in \eqref{gj def} are very small;  more precisely,
\[
\max(|\fy_{j}(r_{\ell}+) - \fy_{j}(r_{\ell}-)|, | \fy_{j}'(r_{\ell}+) - \fy_{j}'(r_{\ell}-)|) < 3\cdot 10^{-4} \quad\forall\; \ell= 2,3,4,\; j=1,2
\]
The   functions $g_{1}, g_{2}$ from the previous definition satisfy an ODE which is a perturbation of our main Sturm-Liouville equation.

\begin{lem}
\label{lem:fund sys}
The functions $g_{1}, g_{2}$ are in $C^{1}([0,\I))\cap C^2_p$, and solve the ODE
\EQ{\label{ U V ODE}
-y''(r) + U(r) y'(r) + V(r) y(r) =0
}
where $U$ and $V$ are piecewise rational functions which obey the estimates
\EQ{\label{UV est}
\| U\|_{\I} < U^e:=\f1{165}, \quad \| V -1 + 3 \tilde Q^{2}\|_{\I} < V^e:=\f1{36}
}
as well as $U=0$, $V=1$ on $J_{4}$.
\end{lem}
\begin{proof}
That $g_{j}$ are in $C^{1}([0,\I))\cap C^2_p$ is clear by construction.
One has for $r\in J_{\ell}$, $1\leqslant \ell\leqslant 4$,
\EQ{\nn
U(r) = \f{ g_{1}''(r) g_{2}(r) - g_{2}''(r) g_{1}(r) }{g_{1}'(r) g_{2}(r) - g_{2}'(r) g_{1}(r)}, \quad
V(r) = -\f{    g_{1}''(r)    g_{2}'(r) -   g_{2}''(r)  g_{1}'(r) }{ g_{1}'(r)  g_{2}(r) -  g_{2}'(r)  g_{1}(r)}
}
which are manifestly rational functions. Moreover, $U=0$, $V=1$ on $J_{4}$.
To obtain the bounds of~\eqref{UV est} we rely on Note~\ref{met1}. Let
${\boldsymbol{\pi}}=\(0,\f{4}{25},\f{13}{50},\f{3}{10},\f{2}{5},\f{3}{5},\f{4}{5},1,\f{3}{2},2,\f52\)$.

To estimate $U$, one writes within each subinterval $U(r)={M(z)}/{P(z)}$ with polynomials $M,P\in \Q[e^{ r_{7}},e^{- r_{7}}][z]$ and obtain
\[
|U(r)| \leqslant \f{\bar M(1)}{|P(0)|-\bar P(1)}
\]
where $\bar M(z)$ is obtained by placing absolute values on all coefficients of $M$, whereas $\bar P$ is the result of the same
procedure applied to $P(z)-P(0)$.  For $V$, one proceeds similarly, by first writing $p_j(r)^2(V-1+3\tilde Q^2)$ as a rational function on $(0,r_{7})$
with $j=1,2$ depending on whether $r<r_{3}$, or $r>r_{3}$, respectively. Substituting the same affine change of variable as for $U$
into the numerator and denominator of this rational function now establishes the desired bound.
\end{proof}

We shall need to modify the system $g_{j}$ to accommodate the $\frac{2}{r}\f{d}{dr}$ term
in the three-dimensional radial Laplacian. In the following lemma, note that $g_{0}$ is regular at $r=0$,
but grows exponentially, whereas $g_{\I}$ decays exponentially, but is singular at $r=0$.
\medskip

\begin{lem}\label{lem:fund sys 2}
The functions $$ g_{0}(r):=r^{-1}\left(g_{2}(r)-\frac{g_{2}(0)}{g_{1}(0)} g_{1}(r)\right), \quad   g_{\I}(r):= r^{-1}g_{1}(r)$$
form a fundamental system for  the equation
\begin{equation}
  \label{ U V ODE 2}
  -y''(r)+\left(U(r)-\f{2}{r}\right)y'(r)+\left(V(r)+\f{U(r)}{r}\right)y(r)=0,
\end{equation}
and their Wronskian satisfies the estimate
\EQ{
\label{W bd}
\quad W(r)=r^{-2} \text{\ \ if\ \ }r\geqslant r_{7}\\
\f{9}{10} < r^{2} |W(r)|\leqslant 1 \qquad \forall\ 0<r<r_{7}
}
One has the following pointwise bounds: $$|g_{\I}(r)|\leqslant (\tfrac{1}{4r} + \tfrac{3}{20})\h_{04}(r)+\tfrac{e^{-r}}{r}\h_{4\infty}(r)\text{\ \  and\ \  }|g_{\I}'(r)|\leqslant
\tfrac{1}{2r^{2}}\h_{04}(r) + \tfrac{1+r}{r^2}{e^{-r}} \h_{4\infty}(r)$$
as well as
$$|g_{0}(r)|< \tfrac{13}{2}\h_{[0, \tfrac12)}(r)+ \tfrac{7}{5} \tfrac{e^{r}}{1+r}\h_{[\tfrac12,\infty)}(r)\text{\ \  and \ \ } |g_{0}'(r)|< 18\h_{04}(r)+\tfrac{e^{r}}{2r}\h_{4\infty}(r)$$
\end{lem}
\begin{proof}
The first claim follows immediately from \eqref{ U V ODE}.
Let $\tilde g_{j}(r):=r^{-1}g_{j}(r)$ for $j=1,2$. Then
the Wronskian $W(r)=g_{0}(r)g_{\I}'(r)-g_{0}'(r)g_{\I}(r)$ of $g_{\I}$ and $g_{0}$ satisfies
\EQ{ \label{Wronskian}
W(r):=\tilde g_1(r) \tilde g_2'(r)- \tilde g_1'(r) \tilde g_2(r)=r^{-2}\, e^{\int_{r_{7}}^r U(s)\, ds}
}
This follows from the fact that $W$ is continuous by Lemma~\ref{lem:fund sys}, and
\[
W'(r)= (U(r)-2/r) W(r), \quad W(r)=r^{-2} \text{\ \ if\ }r\geqslant r_{7}
\]
by \eqref{ U V ODE 2}. From \eqref{UV est} one now obtains \eqref{W bd}.

For the
estimates of $g_0$ and $g_{\infty}$ we use re-expansions and the following simple observation.

\begin{Note}\label{N2.4}{\rm  Estimates of  functions $f$ for which $f'$ is a quadratic polynomial multiplied by a monomial of any degree or by an exponential are elementary.  We can use this in our estimates as follows:
if $f=\sum_{j=1}^m f_j$ and  $J\subset\mathbb{R}$ then, clearly,
\begin{equation}
  \label{eq:eqh3}
\inf_Jf\geqslant  \sum_{j=1}^m\inf_J f_j;\ \   \sup_Jf\leqslant \sum_{j=1}^m\sup_J f_j;
\end{equation}
If $f$ is a polynomial we write $f$ as a sum of  subpolynomials $f_j$, the first one containing the monomials
of degree $\leqslant 3$ and the others consist of the monomials
of degrees $\in [3l+1,3l+3]$ for all $l$ with $3l+1\leqslant$deg$f$.  If $f(r)=e^{\pm r} P(r)$ , the same applies
to a decomposition $e^{\pm r}P_j$ in which $P_1$  has degree two and the  $P_j,\; j\geqslant 2$ consist of the monomials of degrees $2j+1$ and $2j+2$.}
\end{Note}
To establish $|g_1(r)|=|rg_\I (r)|<\f1{4}+\f{3r}{20}$ on $(0,1)$ let
$m_1(r)=\tfrac14+\tfrac{3r}{20}$. The result follows using Note~\ref{N2.4} for the polynomials $g_1\pm m_1$ taking
$r=\tfrac3{20}+\tfrac{3s}{20}$ and re-expanding (as a polynomial in $s$) $s\in
[-1,1]$ and then $r=\tfrac{3}{5}+s$, $s\in [-\tfrac3{10},\tfrac25].$

To establish $|g_\I' (r)|<\f{1}{2r^2}$ on $(0,r_4)$, one proves the
equivalent \EQ{\label{g1'g1} |rg_1'(r)-g_1(r)|< \f12 \qquad 0<r<r_4 }
The result follows once more from
Note~\ref{N2.4}, re-expanding $rg_1'(r)-g_1(r)$ via $r=\tfrac3{20}+\tfrac{3s}{20}$
for $|s|<1$, $r=\tfrac12+s$,  $s\in [-\tfrac15,\tfrac15]$,  and $r=\f7{10}+s$,   $s\in [0, \f{3}{10}]$.

Next, we turn to the estimate $|g_\I(r)|< \f{e^{-r}}{r}$ on $r\geqslant r_4$. On $r\geqslant r_{7}$ we
have an exact equality to $\f{e^{-r}}{r}$  so it suffices to deal with $r_4\leqslant r< r_{7}$. We note that the inequality is sharp, and we need a different method:
we let  $\varphi =e^rg_1$ and look at $\varphi''$. Here we use Note~\ref{N2.4} and expansions at $r_{6}+s$,
$|s|\leqslant \f1{10}$, and $r=2+s,s\in [-\f15, r_{1}]$ and $r=r_{5}+s,s\in [-\f12, r_{1}]$, to see that $\varphi''<0$. Since $\varphi'(r_4)>0$ and $\varphi'( r_{7})=0$,
we see that $\varphi'\geqslant 0$. Since $\varphi( r_{7})=1$, the property follows.

Finally, one has the bound $|g_\I'(r)|\leqslant (1+r){e^{-r}}/{r^2}$ for $r>r_4$. In view
of the exact expression one has for $r\geqslant r_{7}$ it suffices to deal with $r_4\leqslant r< r_{7}$.  Thus, we
need to verify that
$${e^{r}}| rg_1'(r)- g_1(r)|  \leqslant 1+r , \quad  r_{4}\leqslant r< r_{7}
$$
We let $\varphi_1(r):={e^{r}} (rg_1'(r)- g_1(r)$). Using the partition $\boldsymbol{\pi}=(r_{4},\tfrac{3}{2},\tfrac{9}{5},\tfrac{21}{10},r_{7})$, explicitly
minimizing the leading cubic polynomial and taking the $\ell^1$
norm of the rest, we see that $\varphi_1''(r)>0$. Thus, by monotonicity,   $\varphi_1'(r)\leqslant \varphi_1'(r_7)<-1 $.
This implies that $\varphi_1(r)< 0$ on $[r_4,r_7]$. Moreover,
$ \varphi_1(r)+1+r\geqslant 0$ on $[r_4, r_{7}]$ since $\varphi_1(r)+r$ is decreasing and $\varphi(r_7)+r_7+1\geqslant 0$.

For $g_0$ one proceeds in a similar fashion.
We begin with the bound $|g_0(r)|<{13}/{2}$ on $(0,\f12)$. On $[0,r_1]$
we apply Note~\ref{N2.4} for $f$ with $r=\f{3}{20}+\f{3}{20}s$ while on $[r_1,r_4]$ we look at
the polynomials $rg_0(r)\pm \frac{13r}{2}$ with $r=\frac{13}{20}+s$ and $s\in[-\f{7}{20},\f{7}{20}]$.

Next, one verifies that $|g_0(r)|<\f75 \f{e^r}{1+r}$ on $r\geqslant\f12$. On the interval $r \geqslant r_{7}$ one checks
the explicit expression $g_0(r) = \f{e^r}{2r} + \f{k}{r e^r}$ where $k=-\frac{g_2(0)}{g_1(0)}\in (0,4)$ and thus
\[ g_0(r) = \f{e^r}{2r} + \f{k}{r e^r}=\f{e^r}{1+r}\left[\f{1+r}{r}\left(\frac{1}{2}+\frac{k}{e^{2r}}\right)\right]\leqslant \f{7}{5}\left(\frac{1}{2}+\frac{k}{100}\right)\f{e^r}{1+r}\leqslant \f{7}{5}\f{e^r}{1+r}
\]
On the interval $[\f12, r_{7}]$ one can check via the partition $\boldsymbol{\pi}=(\f12,1,\f{9}{5},\f52)$ that $g_0>0$.
Furthermore, we apply the same partition to the expression
$$
5(g_2(r)+k g_1(r))(1+r)-7re^r
$$
which is of the form admitted by Note~\ref{N2.4}. One then sees that it is negative.

For the bound $|g_0'(r)|<18 $ on $(0,r_1)$, we multiply through by~$r^2$.  We then use the substitution~$r=\f75+s$. On $[ r_{1},r_{4}]$
one uses $r=\f12+s$.

Finally, we verify $|g_0'(r)| < \f{e^r}{2r}$  for $r\geqslant r_4$.  On $r\geqslant r_{7}$ one checks that
\[
r e^{-r} g_0'(r) = \f12 - \alpha e^{-2r} -\f{1}{2r} - \f{\beta}{r} e^{-2r}
\]
with $\alpha, \beta>0$. The right-hand side is clearly increasing in $r$
and $<\frac12$. Since    $g_0'( r_{7})>0$, the claim holds for $r\geqslant r_{7}$.  On $r_4\leqslant r< r_{7}$ we re-expand  $r g_0'(r)e^{-r}$ via $r=\f74+s$ with $|s|\leqslant\f34$.
\end{proof}

Now we come to the main result of this section, which is the estimate of the relative error between $\tilde Q$ and $Q$.

\begin{prop}
\label{prop:Q error}
Let $Q$ be the exact ground state of \eqref{mainODE} and $\tilde Q$ be the approximate one  given in
Definition~\ref{def:tildeQ}. Then one has the error bound
\EQ{\label{dQ}
| \tilde Q(r) - Q(r) | \leqslant  \eps_{0}\f{e^{-r}}{1+r} \quad\forall\; r\geqslant0
;\ \ \eps_{0}:=7\cdot 10^{-5}}
\end{prop}
\begin{proof}
Rewrite \eqref{de eq} in the form
\EQ{\label{de eq 2}
   -\delta''(r)+\Big(U(r)-\f{2}{r}\Big)\delta'(r)+\Big(V(r)+\f{U(r)}{r}\Big)\delta(r) &= h_{1}(\delta,r)
}
where
\EQ{\nn
h_{1}(\de,r) &:=-R(r)+U(r)\delta'(r)+\Big[V(r)-1+3\tilde Q^2(r)+\f{U(r)}{r}\Big]\delta(r)+3\tilde Q(r)\delta^{2}(r)+\delta^{3}(r)
}
We seek a solution to \eqref{de eq 2} which obeys the boundary conditions $$\de(0+) \in \R, \quad \de(\I)=0$$
In fact, this solution is unique and is of the form
$\delta=H(\delta)$
where
\EQ{  \label{H}
H(\delta)(r)  &= g_{\I}(r)\int_{0}^{r}\frac{g_{0}(s)h_{1}(\delta,s)}{W(s)} \, ds + g_{0}(r)\int^{\infty}_{r}\frac{g_{\I}(s)h_{1}(\delta,s)}{W(s)}\, ds \\
&=: H_{1}(\delta)(r) + H_{2}(\delta)(r)
}
in terms of the fundamental system from Lemma~\ref{lem:fund sys 2}.
We also have
\begin{multline}\label{Hp}
  [H(\delta)(r)]'=:H'=g_{\I}'(r)\int_{0}^{r}\frac{g_{0}(s)h_{1}(\delta,s)}{W(s)} \, ds + g_{0}'(r)\int^{\infty}_{r}\frac{g_{\I}(s)h_{1}(\delta,s)}{W(s)}\, ds \\
=: H'_{1}(\delta)(r) + H'_{2}(\delta)(r)
\end{multline}
\begin{lem}\label{LQ}
  $H$ is a contraction in the ball
\begin{equation}
  \label{eq:eps0}
  X:= \big\{ f\in C^{1}((0,\I))\mid \|f\|_{X}\leqslant \eps_{0}\}
\end{equation}
with norm
\[ \| f \|_{X}:=\sup_{r\geqslant0}\, (r+1)e^{r}\left(|f(r)|+\tfrac15|f'(r)|\right) \]
Thus there  is a unique fixed point $\de_{0}\in X$.
\end{lem}
\begin{cor}
   Therefore,  $y:=\tilde Q+\de_{0}>0$ solves~\eqref{mainODE}
on $(0,\I)$\footnote{Since $y\in C^1$ by construction, it is a weak solution of~\eqref{mainODE} and by standard Sturm-Liouville theory therefore also a smooth one.}\label{f2}, remains bounded as $r\to0+$, and decays as $r\to\I$. By Coffman's theorem~\cite{Coff},
this uniquely characterizes $Q$ whence $Q-\tilde Q=\de_{0}$.
\end{cor}
\subsection{Proof of Lemma~\ref {LQ}}
 For any $r\in[0,r_{7}]$, denoting $\omega(r):={e^{-r}}/({1+r})$, any $\de\in X$ satisfies
\begin{multline}\label{h1 bd 1}
|h_{1}(\de,r)| \leqslant |R(r)| + \|U\|_{\I} |\de'(r)| + (\|V-1+3\tilde Q^2\|_{\I} + r^{-1}\| U\|_{\I}) |\de(r)|\\
\qquad  + (3\|\tilde Q\|\|\de\|_{\I}  + \|\de\|_{\I}^{2} )|\de(r)|\\
 \leqslant   |R(r)| +\left[ \f{\eps_{0}}{33}  + \(\f{1}{36}  + r^{-1} \f{1}{165}\) \eps_{0} + 14 \eps_{0}^{2}+ \eps_{0}^{3}\right] \om(r)
\leqslant |R(r)|+\left[\f{3\eps_{0}}{50} + \f{\eps_{0}}{165r}\right] \om(r)
\end{multline}
For $r>r_{7}$ we state a different bound, see the case $r>r_{7}$
in  Lemmas~\ref{lem:fund sys} and~\ref{lem:tildeQ1}: with $\rho_{3}:=\f{9}{200}$,
\begin{multline}\label{h1 bd 2}
|h_{1}(\de,r)| \leqslant \f{e^{-3r}}{25r^3} + 3\tilde Q^2(r)|\delta(r)|+3\tilde Q(r)\delta^{2}(r)+|\delta^{3}(r)|
\leqslant  (\rho_{3}+ 24\eps_{0}) \f{e^{-3r}}{r^{3}}  \ \ \forall\; r>r_{7}
\end{multline}

We now show that $H$ takes $X$ to itself, which is based on the estimates of the following lemma.
This is the most technical part of the contraction argument for $\delta$, and it is proved by inserting \eqref{h1 bd 1}
and~\eqref{h1 bd 2} into~\eqref{H} and using the estimates from Lemma~\ref{lem:fund sys 2}.
We suppress the argument $r$ for the most part in the formulation of the following lemma.

\begin{lem}\label{hh12}
Let $\delta\in X$ be fixed and consider \eqref{H}.
For all $r>0$ one has the estimates
\begin{multline}
  | H_1(\de)| <  \(\tfrac{11}{10}\cdot 10^{-5}r^2+\tfrac{\eps_0}{50}\)\boldsymbol{\chi}_{[0,\tfrac12)}+
\left(\tfrac{9}{10}-\tfrac12 r\right)\left(8\cdot 10^{-6}+\tfrac{\eps_0}{25}\right) \boldsymbol{\chi}_{[\tfrac12,1)}\\+\tfrac{1}{r}\left(\tfrac32\cdot 10^{-6}\left(3 - r\right)+\tfrac{13\eps_0}{1000}\right)\boldsymbol{\chi}_{47}+\tfrac{e^{-r}}{r}\left(10^{-4}\left(\tfrac{33}{100}-36e^{-2r}\right)+\tfrac{7\eps_0}{50}-\tfrac{48\eps_0}{25}e^{-2r}\right) \boldsymbol{\chi}_{7\infty}
\end{multline}
\begin{multline}
  | H'_1(\de) | < \left[10^{-5}r\left(2-\tfrac{6r}{5}\right)+\tfrac{\eps_0}{20} \right]\boldsymbol{\chi}_{[0,\tfrac12)}+\left(\tfrac72-3r\right) \left(8\cdot10^{-6}+\tfrac{\eps_0}{25}\right) \boldsymbol{\chi}_{[\tfrac12,1)}\\+\tfrac{2}{r}\left({\tfrac32\cdot 10^{-6}}\left(3 - r\right)+\tfrac{13}{1000}\eps_0\right) \boldsymbol{\chi}_{47}\\+\tfrac{\left(r+1\right)e^{-r}}{r^2}\left(10^{-4}\left(\tfrac{33}{100}-36e^{-2r}\right)+\left(\tfrac{7}{50}-\tfrac{48}{25}e^{-2r}\right)\eps_0\right)\boldsymbol{\chi}_{7\infty}
\end{multline}
\begin{multline}
  |H_2(\de) | <  \left[10^{-5}\left(\tfrac{23}{25}-\tfrac{6}{5}r^2\right)+\tfrac{\eps_0}{15}\right]\boldsymbol{\chi}_{[0,\tfrac12)}+\tfrac{7e}{10}\left[ 10^{-6}\left(\tfrac{11}{50}+2\left(1-r\right)\right)+\tfrac{1}{100}\eps_0\right]\boldsymbol{\chi}_{[\tfrac12,1)}\\+ \tfrac{7}{5(r+1)}\left(\tfrac{11}{10}\cdot10^{-6}+10^{-2}(6-\tfrac{11}{5} r)\eps_0\right) \boldsymbol{\chi}_{47}+\tfrac{21e^{-3r}\left(3+1600\eps_0\right)}{4000r^2(r+1)}\boldsymbol{\chi}_{7\infty}
\end{multline}
\begin{multline}
  |H'_2(\de) | <  \tfrac{36}{13}\left[10^{-5}\left(\tfrac{23}{25}-\tfrac{6}{5}r^2\right)+\tfrac{\eps_0}{15}\right]\boldsymbol{\chi}_{[0,\tfrac12)} + 18\left(10^{-6}\left(\tfrac{11}{50}+2\left(1-r\right)\right)+\tfrac{1}{100}\eps_0\right) \boldsymbol{\chi}_{[\tfrac12,1)}\\+
 \tfrac{1}{r+1}\left(\tfrac{11}{10}\cdot10^{-6}+10^{-2}(6-\tfrac{11}{5} r)\eps_0\right)\boldsymbol{\chi}_{47}+\tfrac{3e^{-3r}\left(3+1600\eps_0\right)}{1600r^3}\boldsymbol{\chi}_{7\infty}
\end{multline}

\end{lem}

\begin{proof}  To avoid working with Ei$(x)$, we write  $\om\left(s\right)\leqslant {e^s}/({a+1})$ inside every definite integral from $a$ to $b$, and ${e^{-3s}}/{s^3}\leqslant {4e^{-3s}/(25s)}$ for $s\geqslant r_{7}$. Also, estimating  sums or products  is of course elementary: for
instance, $e^{x}+ax^2+bx+c<d$ is equivalent to
$e^{-x}\left(d-\left(ax^2+bx+c\right)\right)-1>0$ -- checked by examining the derivative. Using this, the result  follows by straightforward
calculations of the integrals using Lemma~\ref{lem:fund sys 2},
\eqref{h1 bd 1}, and \eqref{h1 bd 2}. See Section~\ref{sec:app1} for details.
\end{proof}
Now we can show that  $H\left(\de\right)$ takes
the $\eps_{0}$-ball of $X$ to itself
\begin{cor}\label{c2.7}
Let  $(H_0(\delta))(r)=|H\left(\delta\right)\left(r\right)|+\tfrac{1}{5}|H'\left(\delta\right)\left(r\right)|$. We have
  \begin{equation}
    \label{eq:estH0*}
    (r+1)e^r (H_0(\delta))(r)\leqslant \eps_{0}\ \ \forall r\in\mathbb{R}^+
    \end{equation}
\end{cor}
\begin{proof} This is a straightforward consequence of Lemma~\ref{hh12};
the details are given in~\S\ref{Pc29}.
\end{proof}

\subsection{Contractivity of the map}
Let $\de_0(r)=\de_1(r)-\de_2(r)$. Clearly, $\delta_{1,2}$ only occur in the integrands in $H_0(\de_1(r)) -H_0(\de_2(r))$  (through $h_{1}(\de_1,r)-h_{1}(\de_2,r)$).  Note now that
\begin{multline}\label{m54}
  |h_{1}(\de_1,r)-h_{1}(\de_2,r)| \leqslant (\|V-1+3\tilde Q^2\|_{\I} + r^{-1}\| U\|_{\I}) |\de_0(r)|\\
 + (6\|\tilde Q\|\|\de_0\|_{\I}  +3 \|\de_0\|_{\I}^{2} )|\de_0(r)|+ \|U\|_{\I} |\de_0'(r)|\leqslant \(\f{3}{50}+\f{1}{165r}\)\om(r)\|\de_0\|
\end{multline}
If we replace $\eps_0$ by $\|\delta_0\|$ and set $R=0$, the last term  in \eqref{h1 bd 1}
is identical to the last term in \eqref{m54}.  Hence, with the replacements above, all contractivity calculations
shadow those for the bounds on $h_1$.  Thus, with virtually the same proof as that of Corollary \ref{c2.7}, see Section~\ref{Pc29},  one derives the estimate
\begin{multline}
(r+1)e^r(|H(\de_1)(r)-H(\de_2)(r)|+|H'(\de_1)(r)-H'(\de_2)(r)|/5)/\|\de_0\|\\\leqslant
  \f{4e^r(r+1)}{25}\boldsymbol{\chi}_{[0,\tfrac12)}+\left(\f{13}{100}-\f{11r}{250}\right)(r+1)e^r\boldsymbol{\chi}_{[\tfrac12,1)}+\f{10^{-2}e^r}{r}\left(\f{37}{20}+9r-\f{12r^2}{5}\right)\boldsymbol{\chi}_{47}\\+\left[\frac{21}{125}+\frac{7}{250 r^2}+\frac{49}{250 r}+e^{-2r}\left(-\frac{23}{10}+\frac{3}{5 r^3}+\frac{431}{50 r^2}-\frac{67}{25 r}\right)\right]\boldsymbol{\chi}_{7\infty}<1/2
\end{multline}
where the final bounds follow by differentiating in $r$ for the first three and in $1/r$ for the last. For the last one we note that $-\frac{23}{10}+\frac{3}{5 r^3}+\frac{431}{50 r^2}-\frac{67}{25 r}$ is negative and decreasing.
\end{proof}
\subsection{Further estimates}In the study of $L_{\pm}$ we need a sharper estimate of $Q-\tilde{Q}$.
\begin{lem}\label{210}
For $r\geqslant \f{5}{2}=r_7$, we have
$$\de(r)=Q(r)-\tilde{Q}(r)= b_1e^{-r}/r+b_2(r)$$
where $|b_1|<5\cdot 10^{-5}$ and $|b_2(r)|<\f{3}{50}\frac{e^{-3r}}{r^3}$.
\end{lem}
\begin{proof}
It follows directly from \eqref{H} that
\begin{equation}
\begin{aligned} \label{Hnew}
\delta(r) & = g_{\I}(r)\int_{0}^{\infty}\frac{g_{0}(s)h_{1}(\delta,s)}{W(s)} \, ds +
g_{\I}(r)\int_{\infty}^{r}\frac{g_{0}(s)h_{1}(\delta,s)}{W(s)} \, ds \\
&\qquad +g_{0}(r)\int^{\infty}_{r}\frac{g_{\I}(s)h_{1}(\delta,s)}{W(s)}\, ds
\end{aligned}
\end{equation}
It then follows from \eqref{h1 bd 2}  that
\begin{equation}
  \label{eq:*1}
 \left|\int_{0}^{\infty}\frac{g_{0}(s)h_{1}(\delta,s)}{W(s)} \, ds\right|\leqslant 5\cdot10^{-5};\ \ \left|g_{\I}(r)\int_{\infty}^{r}\frac{g_{0}(s)h_{1}(\delta,s)}{W(s)} \, ds\right |\leqslant
 \f{1}{25} \f{e^{-3r} }{r^3}
\end{equation}
and
\EQ{
\left|g_{0}(r)\int^{\infty}_{r}\frac{g_{\I}(s)h_{1}(\delta,s)}{W(s)}\, ds\right |\leqslant \f{1}{50} \f{e^{-3r}}{r^3}
}
which concludes the proof.
\end{proof}

\section{The operator $L_+$}

In this section we prove the first of our main results, namely the gap property of $L_+$.

\begin{thm}\label{P1}
The operator $L_{+}$   has no ($L^2(\R^+)$) eigenvalue or resonance in  $[0,1]$.
\end{thm}
\zn Standard ODE analysis shows that there are
two solutions $u_1(r;\lambda)$ and $u_2(r;\lambda)$ of \eqref{eqorL} with the properties $u_1(0;\lambda)=1$ and $u_2(r;\lambda)=r^{-1}e^{-r\sqrt{1-\lambda}}(1+o(1)), \,\,r\to\infty$. These are, up to constants, the only ones
acceptable at zero and infinity respectively, see  \S\ref{S51}.

Let $W=u_1u_2'-u_2u_1'$ be the Wronskian of the two special
solutions $u_{1,2}$. As mentioned in Section \ref{tech}, the existence of an eigenvalue or resonance of
$L_+$ is equivalent to $W=0$ for some $\lambda$.  Note that $W$ is not constant due to the first order
derivative in $L_\pm$. However, $ W'(r)=-\frac{2}{r}W(r)$ whence $W(r_0)=0$ at one point
implies that $W(r)=0$ everywhere.
Theorem~\ref{P1} is a corollary of the following result.

\begin{prop}\label{P3.2}
 We have the following estimate
 \begin{equation}
   \label{eq:eqW}
   \inf_{\lambda\in[0,1]}|W(r_7;\lambda)|\geqslant 43\cdot 10^{-4}
 \end{equation}
\end{prop}

\subsection{Proofs}
As mentioned in Section \ref{tech}, we construct the quasi-solution
$w^{+}_1$, a piecewise polynomial of two variables $r$ and $\la$, on
the interval $[0,r_7]$, and $w^{+}_2$ whose expression involves
exponential integrals on the interval $[r_7,\infty)$. To show that
$w^{+}_1$ is close to $u_1$ and $w^{+}_2$ is close to $u_2$, we use
the same contractive mapping strategy we used for $Q-\tilde{Q}$
outlined in Section \ref{tech}: the only difference is that the
equations and solutions depend on the parameter $\la$. The method for
obtaining the estimates is explained in Note~\ref{met2}
below. Finally, to estimate the Wronskian, we first approximate
$w^{+}_{2}(r_7)$ and $\partial_r w^{+}_{2}(r_7)$ by polynomials of degree~$7$ in
$\la$, and the calculation then reduces to estimating a polynomial
using Note~\ref{met1}.
\subsubsection{A polynomial quasi-solution for $r\leqslant r_{7}$}
To build the quasi-solution, we first define
  \begin{equation}
    \label{eq:qp1}
  \tilde{w}_1(r)=\sum_{(j,k,l)\in S}\boldsymbol{\chi}_jc_{kl;j}\lambda^{k}z^l    \end{equation}
where $S=\{j,k,l\mid 1\leqslant j\leqslant 3,\ 0\leqslant k\leqslant M_j,\ 0\leqslant l  \leqslant M_j\}$,
$c_{kl;j}$ are given in the appendix, $\boldsymbol{\chi}_j$, $j=1,2,3$,
are the characteristic functions of   $[0,r_{2})$, $[r_{2},r_5)$, and $[r_5,r_{7}]$, respectively, $M_j\leqslant 15$ and
$z$ depends on  $r$ as specified in the top rows of the tables in the appendix. For the most part, we suppress the $\lambda$-dependence in
our notations.
To ensure that $w^{+}_1$ is $C^1$ we next let
\begin{multline}\label{w1pdef}
   w^{+}_1(r)= \tilde{w}_1(r)\boldsymbol{\chi}_{{57}}+\left[\tilde{w}_1(r)+\tilde{w}_1(r_{5})-\tilde{w}_1(r_{5}{-})+(\tilde{w}'_1(r_{5})-\tilde{w}'_1(r_{5}{-}))(r-r_{5})\right]\boldsymbol{\chi}_{25}\\+
\left[\tilde{w}_1(r)+w^{+}_1(r_2)-\tilde{w}_1(r_2{-})+\tfrac{1}{2r_2}(\partial_{r} w_{1}^{+}(r_2)-\tilde{w}'_1(r_2{-}))(r^{2}-r_2^{2})\right]\boldsymbol{\boldsymbol{\chi}}_{02}
\end{multline}

\begin{Note}\label{met2}{\rm

In this section, we divide $[a,b]\times [0,1]$ ``vertically'',  using a partition $\boldsymbol{\pi}$ of $[a,b]$ in $r$ and the trivial partition $(0,1)$ in $\lambda$. Clearly,
such partitions are determined  by $\boldsymbol{\pi}$.

(i) Let  $P(y,z)$ be a polynomial  with $(y,z)\in [-1,1]^2$ and
 $P_1$ defined to be the sub-polynomial consisting of all terms of $P$ of the form $c_{k,0}z^k,0\leqslant k\leqslant 3$ and $c_{0,j}y^j,1\leqslant j\leqslant 3$. We bound  $P$ above (below)  by the maximum (minimum, resp.) of $P_1$ plus (minus, resp.) the sum of the absolute value of the coefficients of $P-P_1$. Since $P_1$ is the sum of a cubic polynomial in $z$ and a cubic polynomial in $y$, the extrema calculations reduce to solving one variable quadratic equations. We proceed in this way for simplicity, given that the mixed terms typically do not contribute too significantly
to the final estimates of our polynomials.
(ii) If  $f(r)=e^{ar}Q(r)$ (equivalently, $f(r)=e^{ar}/Q(r)$) where $Q$ is a quadratic polynomial, then $f'=0$
iff a quadratic is zero; thus $e^{ar}\pm Q(r)>c$ where $c$ is any constant can be verified by elementary means.
}
\end{Note}
\begin{lem}\label{w1b}
  We have $ |w^{+}_1(r)| \leqslant M(r):=\tfrac{11(1-r)}{10}\h_{02}+\tfrac{9}{20}\h_{{27}}$
\end{lem}
\begin{proof}

We introduce the  partition of $[0,r_{7}]\times[0,1]$ induced by ${\boldsymbol{\pi}}=(0,3/10,17/25,3/2,5/2)$.
The lemma now follows by applying the method in Note~\ref{met2} to $ w^{+}_1(\ell_{k}(z),\tfrac12 (1+y))\pm M(\ell_{k}(z))$ for $1\leqslant k\leqslant 4$.
\end{proof}

Now we invoke the estimates for $|Q-\tilde{Q}|$ (Proposition~\ref{prop:Q error}) and $\tilde{Q}$ (Lemma~\ref{lem:tildeQ1}) to obtain
  \begin{multline}\label{3q2}
    |3Q^{2}(r)-3\tilde{Q}^{2}(r)|\leqslant 3(2\tilde{Q}+|Q-\tilde{Q}|)|Q-\tilde{Q}|(r)\\ \leqslant
    \tfrac{1}{10000}\left[ 21 e^{-3r}\h_{02}+ \f{116}{5} (1+r)^{-1}e^{-13r/5}\h_{{27}}\right]+\f32\cdot10^{-8}e^{-2r}
    \\ \leqslant \tfrac{1}{5000}\left[ 11 e^{-3r}\h_{02}(r)+ 12 (1+r)^{-1}e^{-13r/5}\h_{{27}}(r)\right]
\end{multline}
We shall use this estimate repeatedly, for example in the following  bound on the remainder
\begin{equation}
  \label{eq:R1*}
  R_1:=-w^{+}_1{''}-2w^{+}_1{'}/r+(1-\lambda-3Q^{2})w^{+}_1
\end{equation}

\begin{lem}\label{lpr1}
  We have
  \begin{equation}\label{estr1}
 |R_{1}(r)|\leqslant 10^{-4}\left\{ \tfrac{33}{10}(11-35r+34r^2)\boldsymbol{\chi}_{02}+
 \tfrac{11-6r}5\boldsymbol{\chi}_{25}+\tfrac{39-9r}{100}\boldsymbol{\chi}_{{57}}\right\}
  \end{equation}
\end{lem}
\begin{proof}
We first consider $R_0=-w^{+}_1{''}-2w^{+}_1{'}/r+(1-\lambda-3\tilde{Q}^{2})w^{+}_1$ and the piecewise  polynomial
\begin{equation}
  \label{eq:eqR0p}
  \tilde{R}_0(r)=[ R_0(r)/\tilde{Q}^2(r)]\h_{02}+[rR_0(r)/\tilde{Q}^2(r)]\h_{{27}}
  \end{equation}
  On $[0,r_{2}]$
we use the partition $
 \boldsymbol{\pi}= \({0}, \tfrac{7}{100}, \tfrac{21}{100}, \tfrac{7}{20}, \tfrac{12}{25}, \tfrac{59}{100}, \tfrac{33}{50}, r_2\)$. Since $\tilde{Q}$ is positive and decreasing, in each subinterval we only need to show that $|\tilde{R}_0(\ell_k(z))|\tilde{Q}^2(\ell_k(-1))$ is bounded by the right-hand side of \eqref{estr1}. It is easy to see that this can be reduced
to applying the method in Note~\ref{met2} for $\tilde{R}_0(\ell_k(z))\tilde{Q}^2(\ell_k(-1))\pm 10^{-4}(11-43r+54r^2)$. In this fashion, we obtain the estimate
$$|R_0(\ell_k(z))|\leqslant|\tilde{R}_0(\ell_k(z))|\tilde{Q}^2(\ell_k(-1))<10^{-4}(11-43r+54r^2)$$
for all $1\leqslant k\leqslant 7$.

Next we introduce the partition   of  $[r_{2},r_{5}]$ provided  by
$\boldsymbol{\pi}=\(r_{2}, \f {37}{50}, r_{3}, \f{57}{50}, \f{69}{50}, r_{5}\)$
to estimate $|\tilde{R}_0(\ell_k(z))|\tilde{Q}^2(\ell_k(-1))\pm 2\cdot 10^{-5}\ell_k(z)$.  This yields the estimate
$$\ell_k(z)|R_0(\ell_k(z))|\leqslant|\tilde{R}_0(\ell_k(z))|\tilde{Q}^2(\ell_k(-1))<2\cdot 10^{-5}\ell_k(z)\ \  \forall\; 1\leqslant k\leqslant 5 $$
For the interval $[r_{5}, r_{7}]$ we use the partition  induced by
$\boldsymbol{\pi}=\(r_{5}, \tfrac{17}{10}, \tfrac{19}{10}, \tfrac{52}{25}, \tfrac{56}{25}, \tfrac{59}{25}, \tfrac{61}{25}, r_{7}\)$
and concludes that $$\ell_k(z)|R_0(\ell_k(z))|\leqslant|\tilde{R}_0(\ell_k(z))|\tilde{Q}^2(\ell_k(-1))<15\cdot 10^{-6}\ell_k(z)$$
We have $|R_1-R_0|=3|Q^2-\tilde{Q}^2|\, |w_1^+|$.
Combining these results with \eqref{3q2} and Lemma~\ref{w1b} to estimate
$|3Q^2-3\tilde{Q}^2\|w^{+}_1|$, we obtain
\begin{multline}\label{tildeQ1}
  | R_1(r)| \leqslant 10^{-4}\left\{\left[11-43r+54r^2+\tfrac{121(1 - r)}{5}e^{-3r}\right]\h_{02}+\left(1+\tfrac{{54}e^{-13r/5}}{1+r}\right)\frac{\h_{25}}{5}\right.\\\left.+\left(\tfrac{3}{20}+\tfrac{54}{5}\tfrac{e^{-13 r/5}}{1+r}\right)\h_{{57}}\right\}
\end{multline}
The right-hand side
of \eqref{estr1} minus the right-hand side of \eqref{tildeQ1}
is positive; this follows using Note~\ref{met2} (ii), after
the substitution of  ${5e^{-3r}}\h_{02}$ by the bound
 $(5 - 11 r + 7 r^2)\h_{02}$.
\end{proof}

\subsubsection{A quasi-fundamental system of solutions on $[0,r_{7}]$}\label{S312}
To show there is an actual solution $u_1$ of \eqref{eqorL} on $[0,r_7]$ with $\delta:=u_1-w_1^+$ small, we construct two functions $g^{+}_{1,2}$, approximating two linearly independent solutions of \eqref{eqorL}; define first
\begin{equation}
   \tilde{g}^{+}_{1}(r):=\sum_{(j,k,l)\in S}d_{kl;j}^+\boldsymbol{\chi}_j\lambda^{k}z^l;\ \ \tilde{g}^{+}_{2}(r):=\sum_{(j,k,l)\in S}e_{kl;j}^+\boldsymbol{\chi}_j\lambda^{k}z^l
\end{equation}
where $S=\{(j,k,l)\mid 1\leqslant j\leqslant 3, \ 0\leqslant k\leqslant M_j,\ 0 \leqslant l\leqslant 15\}$ and now $\boldsymbol{\chi}_j$, $j=1,2,3$, are the characteristic functions of
$[0,r_{1})$, $[r_{1},r_{4})$, and $[r_{4},r_{7}]$,
respectively. The coefficients, the intervals corresponding to
$j=1,2,3$, and the expressions  $z=z(r)$ are given in the appendix.

To ensure $C^1$ behavior
we use the same method as in Definition \ref{Green Poly} to set $ \hat{g}^{+}_{j}=\tilde{g}^{+}_{j}$ on $J_{1}$ and
 \EQ{
\hat{g}^{+}_{j}(r) := \tilde{g}^{+}_{j}(r) + \hat{g}^{+}_{j}(r_{\ell}-) - \tilde{g}^{+}_{j}(r_{\ell}+) + (\hat{g}^{+}_{j}{'}(r_{\ell}-) - \tilde{g}^{+}_{j}{'}(r_{\ell}+))(r-r_{\ell})  \quad \forall\; r\in J_{\ell}
}
for all $\ell=2,3,4$
 and\footnote{ We note that $g^{+}_1$ is approximately $w^{+}_1$ redefined on the same interval as $g^{+}_2$ for convenience.}
 $g^{+}_j(r)=\hat{g}^{+}_{j}(r)/r$.  In contrast to Definition~\ref{Green Poly} where the $C^{1}$-matching is done from right to left, we find it necessary to carry out the matching by going from left
 to right.

We construct a second order equation satisfied
by $g^{+}_1$, $g^{+}_2$ in the form
\begin{equation}
  \label{eq:unkn}
 -g''+(A(r)-2/r)g'+(B(r)+A(r)/r)g=0
\end{equation}
where
$$A(r)=\frac{\hat{g}^{+}_1(r) \hat{g}^{+}_2{''}(r)-\hat{g}^{+}_2(r) \hat{g}^{+}_1{''}(r)}{\hat{g}^{+}_2(r) \hat{g}^{+}_1{'}(r)-\hat{g}^{+}_1(r) \hat{g}^{+}_2{'}(r)} \text{ \rm and } B(r)=\frac{\hat{g}^{+}_1{'}(r) \hat{g}^{+}_2{''}(r)-\hat{g}^{+}_2{'}(r) \hat{g}^{+}_1{''}(r)}{\hat{g}^{+}_2(r) \hat{g}^{+}_1{'}(r)-\hat{g}^{+}_1(r) \hat{g}^{+}_2{'}(r)}$$

\begin{lem}\label{L4.4}
   We have $|A(r)|\leqslant \frac{1}{1000}(\tfrac{1}{2}\h_{06}+4\h_{{67}})$ and $|B(r)-1+\lambda+3\tilde{Q}^{2}(r)|<\frac{3}{500}$.
  \end{lem}
\begin{proof}
We use a partition $\boldsymbol{\pi}=\({0}, \tfrac{7}{50}, \tfrac{7}{25}, r_1, \tfrac{9}{25}, \tfrac{1}{2}, \tfrac{18}{25}, \tfrac{46}{50}, r_{4},\tfrac{7}{5}, \tfrac{19}{10}, \tfrac{11}{5}, r_{6}, r_{7}\)$ and estimate above/below, by re-expansion,  the absolute value of the numerator (denominator, resp.). \end{proof}
 We rewrite the equation for $\delta$ (see beginning of \S\ref{S312}) in the form
\begin{multline*}
  -\delta''+(A(r)-2/r)\delta'+(B(r)+A(r)/r)\delta=R_1(r)+A(r)\delta'(r)\\
  +[B(r)-1+\lambda+3\tilde{Q}(r)^{2}+A(r)/r+3(Q(r)^{2}-\tilde{Q}(r)^{2})]\delta
  =:h_2
\end{multline*}

\zn {\bf Note.} In the following we write  $\|f\|$ for  $\sup_{[0,r_{7}]} |f|$.
\begin{lem}
  We have
  \begin{multline}
  |h_{2}(\delta,r)|\leqslant 10^{-4}\left[{33}(11 - 35 r + 34 r^2)/10+5\|\delta'\bigr\|+(82+5/r)\|\delta\|\right]\h_{02}\\
+10^{-4}\L(11-6r)/5+70\|\delta\|+5\|\delta'\| \J\h_{25}+10^{-4}\L(39-9r)/100+64 \|\delta\|+5\|\delta'\|\J\h_{56}\\+10^{-4}\L (39-9r)/{100}+78\|\delta\| +40\|\delta'\|\J\h_{{67}}
  \end{multline}
  \end{lem}
\begin{proof}
This is obtained by combining  \eqref{estr1}, Lemma~\ref{L4.4} and \eqref{3q2}
and using the monotonicity of the coefficients containing exponentials. \end{proof}

 \begin{lem}
   We have
$|g^{+}_{1}|\leq \tfrac{11}{10}\h_{01}+\tfrac{1}{2}\h_{{17}}
     \ \  \text{\rm and }     |g^{+}_{2}|\leqslant \tfrac{11}{10r}+\tfrac{17}{10} $
\end{lem}
 \begin{proof}
We use the partition induced by $\boldsymbol{\pi}=(0,r_1,\f35,r_4,2,r_7)$ and apply the method in Note~\ref{met2}
 to $\h_{01}(\tfrac{11}{10}\pm g^{+}_{1})$, $\h_{17}(\tfrac{r}{2}\pm rg^{+}_{1})$ and $\tfrac{11}{10}+\tfrac{17r}{10}\pm rg^{+}_{2}$.
 \end{proof} \zn In addition, one has the following.
 \begin{lem}
   There are the bounds $| g^{+}_{1}{'}(r)|\leqslant 3\h_{02}+\h_{25}+\tfrac3{10}\h_{{57}};\ \ |g^{+}_{2}{'}(r)|\leqslant \frac{27}{10r^2}\h_{02}+\f{4}{r^2}\h_{25}+\frac{17}{10}\h_{{57}}$.
\end{lem}
 \begin{proof}
Similar, using
$\boldsymbol{\pi}=({0}, r_{1}, r_2, r_{4}, r_{5}, \tfrac{17}{10}, \tfrac{21}{10}, r_{7}) $.
 \end{proof}

\subsubsection{The actual smooth solution on $[0,r_{7}]$}\label{up0}
Let
\begin{equation}
  \label{eq:eqH0}
  H_{0}(\delta)=g^{+}_{2}(r)\int_{0}^{r}\frac{g^{+}_{1}(s)h_{2}(\delta,s)}{W(s)}\, ds-g^{+}_{1}(r)\int_{0}^{r}\frac{g^{+}_{2}(s)h_{2}(\delta,s)}{W(s)}\, ds
\end{equation}
Clearly, we have $\delta=H_0(\delta)$ and  $\delta'=H'_0(\delta)$ where
\begin{equation} \label{eq:eqH0p}
   H'_{0}(\delta)=g^{+}_{2}{'}(r)\int_{0}^{r}\frac{g^{+}_{1}(s)h_{2}(\delta,s)}{W(s)}\, ds-g^{+}_{1}{'}(r)\int_{0}^{r}\frac{g^{+}_{2}(s)h_{2}(\delta,s)}{W(s)}\, ds
\end{equation}

\begin{lem}
\label{lpcon}
  There is the bound
  \begin{equation}
    \label{eq:estH0}
    |H_{0}(\delta)(r)|+|H'_{0}(\delta)(r)|/5\leqslant 1/1200+\|\delta\|/10+\|\delta'\|/80
  \end{equation}
 Thus $H$ is a contraction in the ball
 \begin{equation}
   \label{eq:sizeball}
   X:= \big\{ f\in C^{1}((0,r_{7}))\mid \|f\|_{X}\leqslant 1/1080\}
 \end{equation}
where
$\| f \|_{X}:=\sup_{r\in[0,r_{7})}\, \left(|f(r)|+\tfrac15|f'(r)|\right)$ (cf.\ also footnote~\ref{f2} on p.~\pageref{f2}).
\end{lem}
\begin{proof}
   We crudely
 estimate the quantities $H_0$ and $H_0'$ by placing absolute values on all terms, and by using the bounds already calculated for $g^{+}_{1,2}$ etc; $|\delta|$
and $|\delta'|$ are estimated by their supremum norms, written as $\|\cdot\|$.
   Since $W(r):=g^{+}_1(r)g^{+}_2{'}(r)-g^{+}_1{'}(r)g^{+}_2(r)={r^{-2}}\exp\Big(\int_{0}^rA(s)ds\Big)$ $\geqslant r^{-2}\exp\Big(-\int_0^{r_{7}}|A(s)|ds\Big)$ (we note that $W(r) r^2\to 1$ as $r\to 0$) we have
$$1/|W(r)|<\frac{626}{625}r^2<\frac{51r^2}{50}$$
On the first interval, $[0,r_1)$ a direct calculation shows that $\tilde{H}_0(\delta)(r):=|H_{0}(\delta)|+|H_{0}(\delta)'(r)|/5$ is majorized by
    \begin{equation}
 \max\(  | rP_4(r)|+|rP_2(r)|\|\delta'\| +|P_3(r)|\|\delta\|\)\leqslant 2\cdot 10^{-6}({200+1300\|\delta\|+57\|\delta'\|})
    \end{equation}
where $P_j$ are polynomials of degree $j$,
 easily maximized since they are  increasing on this interval (all have positive coefficients except
$rP_4$; $(rP_4)'$ has positive coefficients after
we replace $r^2$ by $rr_1$ and $r^3$ by $r r_1^2$); thus $rP_4$ is increasing.

 Calculating $\tilde{H}_0(\delta)(r)$  for $r\in [r_1,r_{2})$,
we obtain a rational function; we first replace $1/r,1/r^2$  by
$1/r_1,1/r_1^2$ respectively (their coefficients are positive)  and we get
an expression of the form $P_5(r)+P_{3}(r)\|\delta\|+\tilde{P}_3(r)\|\delta'\|$,
with the same convention as above for the polynomials (different from
the $Ps$ on the previous interval). Once more, all polynomials except
$P_5$ have positive coefficients.  In $P_5$
we first replace $r^5$ (whose coefficient is positive)
by $r^4 r_2$; the derivative of the new polynomial has an explicit
positive minimum. Thus the maximum of $P_5$
is reached at $r=r_2$. Thus, by taking $r=r_2$, we get,
for $r\in [r_1,r_2)$,
\begin{equation}
\tilde{H}_0(\delta)(r)\leqslant \frac{19}{25000}+\frac{\|\delta\|}{125}+\frac{\|\delta'\|}{2500}
\end{equation}
On the interval $[r_{2},r_{5})$ we proceed in the same way, replacing $1/r,1/r^2$
by $1/r_2,1/r_2^2$ resp. This results in an expression of the
form $P_4+P_3\|\delta\|+\tilde{P}_3\|\delta'\|$ with the same
properties and conventions as above. Now the derivative of $P_4$ can
be minimized explicitly: it is positive and thus $P_4$ is maximal at the right-hand endpoint. We get
   the following majorization of $| \tilde{H}_0(\delta)(r)|$
\begin{equation}
\tilde{H}_0(\delta)(r)\leqslant
\frac{1}{1200}+\frac{\|\delta\|}{34}+\frac{\|\delta'\|}{515}\ \forall \ r\in [r_2,r_5)
\end{equation}
On the interval $[r_{5},r_{6})$ we replace $1/r$ by $1/r_5$ and
obtain an expression very similar to the one on $[r_{2},r_{5})$. It is dealt with
in the same way, whence
\begin{equation}
\tilde{H}_0(\delta)(r)\leqslant {1}/{1250}+{\|\delta\|}/{11}+{\|\delta'\|}/{164}
\ \forall \ r\in [r_5,r_6)
\end{equation}
Finally, on $[r_6,r_7)$, after replacing $1/r$ by $1/r_6$ in
the positive terms and by $1/r_7$ in the negative ones,  we reduce $\tilde{H}_0(\delta)(r)$ to the form $P_4+P_3\|\delta\|+\tilde{P}_3\|\delta'\|$. $P_3$ and $\tilde{P}_3$
are manifestly increasing and $\min_{[r_6,r_7)} P_4'>0$. Thus the maximum
is reached at $r_7$ and we get
\begin{equation}\label{lasteqH}
  \tilde{H}_0(\delta)(r)\leqslant  1/1200+{\|\delta\|}/{10}+{\|\delta'\|}/{80}
  \quad  \forall \ r\in [r_6,r_7)
\end{equation} Of all estimates, the worst bounds are in \eqref{lasteqH};
contractivity as well as preservation of the ball thus follow from  \eqref{lasteqH}.
\end{proof}

\begin{cor}\label{d1d}
The function $w^{+}_1$ differs from an actual solution $u_0$ of $L_+u=\lambda u$ by at most $1/1080$ in $\|\cdot\|_X$. Furthermore,
$$ |\delta'(r_{7})|\leqslant1/1800 $$
\end{cor}
\begin{proof}
The first part is just \eqref{eq:sizeball}. The second part comes from
 direct substitution in $|H_0'|$ followed by
 Lemma~\ref{lpcon}:
$$|{H}_0(\delta)'(r_{7})|\leqslant 1/2080+\|\delta\|_X/17\leqslant  1/2080 +1/17\cdot 1/1080$$
as claimed.
\end{proof}

\begin{lem}\label{w1d}
  We have $
|w^{+}_1(r_{7})|<19/47,|w^{+}_1{'}(r_{7})|<4/19$.
\end{lem}
\begin{proof}
Viewed as quintic polynomials in $\lambda$,  $w^{+}_1(r_{7})$ and $w^{+}_1{'}(r_{7})$ are simply estimated as in Note~\ref{met2} (reduced here to one
variable $y$ where $\lambda=\frac12(1+y)$).
\end{proof}

\subsubsection{The quasi-solution bounded on $[r_{7},\infty)$}\label{upi}
 Let $\sigma=\sqrt{1-\lambda}$ and $\sigma_1=1+\sigma$.
We let $a_0={1413}/{64}$. In this region, we look for $u_{\infty}$ in the form $w^{+}_2+\delta_0$  where
\begin{equation}\label{eqw22}
  w^{+}_2(r)=\frac{e^{- \sigma  r}}{  r}\Big(1 +a_0f_1(r)\Big);\ \ \ \sigma f_1(r):=-\mathrm{Ei}(-2r)+e^{2 \sigma  r}\sigma_1\mathrm{Ei}(-2\sigma_1r)
\end{equation}
is close to an exponentially  decaying
solution of~\eqref{eqorL}.\footnote{  To obtain this approximation,
we rewrite  \eqref{eqorL} as $(-\frac{d^{2}}{dr^{2}}-\frac{2}{r}\frac{d}{dr}+1-\lambda)u=3Q^2u$, replace $Q$ by the leading term of $\tilde{Q}$, and iterate the associated integral equation.}
Then, $w^{+}_2$  satisfies
\begin{equation}
  \label{eq:eq48}
-w^{+}_2{''}-2w^{+}_2{'}/r+\(1-\lambda-a_0\frac{e^{-2r}}{r^{2}}\)w^{+}_2=- a_0^2\frac{e^{-2r-\sigma r}}{  r^{3}}f_1(r)=:R_2
\end{equation}

\begin{lem}\label{l313}  We have (i)
$|w^{+}_2(r)|<e^{-\sigma r}/r$ and (ii)
$|R_2(r)|<\f{2e^{-2r}}{25 r^3}$
\end{lem}
\begin{proof}
By straightforward algebra, using the bounds in \eqref{Ei1 bd} one obtains
\begin{equation}
  \label{eq:eqdiff}
  f'_1(r)=2\sigma_1 e^{2\sigma r}\mathrm{Ei}(-2\sigma_1r)+ \frac{ e^{-2 r}}{r}>0
\end{equation}
Since $\lim_{r\to\infty}f_1(r)=0$, we clearly have $f_1(r_7)\leqslant f_1(r)<0$.
Now $$f_1(r_7;\sigma)=\int_{-\infty}^{-5}\f{e^{u}(u+5)}{u(u-5\sigma)}\, du\geqslant   f_1(r_7;0)=6\mathrm{Ei}(-5)+e^{-5}>-\f{1}{6500}$$
and the estimates follow.
\end{proof}

\subsubsection{The equation for $\delta_0$} The difference $u_{\infty}-w_2^+=\delta_0$ satisfies the equation
\begin{equation}
  \label{eq:eq50}
  -\delta_0''-2\delta_0'/r=R_2+(\lam-1+a_0\frac{e^{-2r}}{r^{2}})\delta_0+\left(3Q^{2}-a_0\frac{e^{-2r}}{r^{2}}\right)(w^{+}_2+\delta_0)=:h_{3}(r)
\end{equation}

\begin{Note}\label{NdQ}
  One has
(i)  \begin{equation*}
    |Q^{2}(r)-\tilde{Q}^{2}(r)|\leq(2\tilde{Q}(r)+|Q(r)-\tilde{Q}(r)|)|Q(r)-\tilde{Q}(r)|\leqslant 4\cdot10^{-4}\frac{e^{-2r}}{r^{2}}
  \end{equation*}
(ii)
\begin{equation*}
  \label{eq:eq51}
|3Q^{2}(r)-a_0\frac{e^{-2r}}{r^{2}}|\leqslant|3Q^{2}(r)-3\tilde{Q}^{2}(r)|+\left|a_0\frac{e^{-2r}}{r^{2}}-3\tilde{Q}^{2}(r)\right|<\frac{1}{20}\frac{e^{-2r}}{r^{2}}
\end{equation*}
\end{Note} Indeed,
(i) follows from \eqref{eq:bdQ2} and \eqref{dQ}. (ii) uses \eqref{eq:bdQ2} and \eqref{dQ} and (i).

\medskip

\noindent
Now using \eqref{eq:eq50}, Lemma~\ref{l313}, and Note~\ref{NdQ} we have
$$|h_3(r)|\leqslant |R_2|+\left|3Q^{2}-a_0\frac{e^{-2r}}{r^{2}}\right||w^{+}_2(r)|+|\lam-1+3Q^{2}||\delta_0(r)| <\f{13e^{-2r}}{100r^3}+|\de_0(r)|$$ Looking  for exponentially decreasing solutions,  we write \eqref{eq:eq50}
 in the integral form
\begin{equation}
  \label{eq:eqd2}
  \delta_0 =H_1 (\delta_0):=-\int_r^{\infty}\f{dt}{t^2}\int_t^\infty s^2 h_3(s)\, ds
\end{equation}

\subsubsection{The actual solution   on $[r_{7},\infty)$}
\begin{lem}\label{l315}
$H_1$ is contractive in the ball $\{f\mid \|f\|\leqslant 13/300\}$ in the Banach space $\{f\mid \|f\|=\sup_{r>r_{7}}r^3e^{2r}|f(r)|<\infty\}$.
\end{lem}
\begin{proof}
  Noting that $(s/t)>1$ and $t>r$ we write
  \begin{equation}
       |H_1 (\delta_0)|(r) \leqslant \f{1}{r^3}
\int_r^{\infty}dt\int_t^\infty s^3h_3(s)ds \leqslant \f{e^{-2r}}{r^3}\(\frac{13}{400}+\frac{\|\de_0\|}{4}\)
  \end{equation}
  whence the claim.
\end{proof}

\noindent Hence $\|H_1(\delta_0)\|\leqslant 13/400+\|\delta_0\|/{4}$ and the claim follows.
\begin{cor}\label{crbd}
  We have
  \begin{equation}
  \begin{split}
    \label{eq:esfin}
    |u_2(r_{7})-w^{+}_2(r_{7})| &=  |\delta_2(r_{7})|\leqslant \frac{13}{300}\frac{2^3}{5^3} {e^{-5}}< 2\cdot 10^{-5} \\
 |u'_2(r_{7})-w'^{+}_2(r_{7})| &=  |\delta'_2(r_{7})|<4\cdot10^{-5}
\end{split}  \end{equation}
\end{cor}
\begin{proof}
  Only $\delta'$ needs to be estimated; this is immediate:
$$|\delta_0'(r)|=|(H_{1}\delta_0)'(r)|\leqslant \f{1}{r^3}\int_r^\infty s^3h_3(s)\, ds\leqslant \f{e^{-2r}}{r^3}
\(\frac{13}{200}+\frac{\|\de_0\|}{2}\)  <4\cdot10^{-5}$$
and we are done.
\end{proof}

\subsubsection{The Wronskian of the well-behaved quasi-solutions}
Eq.~\eqref{eqw22} implies
\begin{multline}
  w^{+}_{2}(r_{7})=\frac{2e^{-r_{7} \sigma  }}{5 \sigma  }( \sigma  -a_0\mathrm{Ei}(-5)+a_0e^{5 \sigma  }\sigma_1\mathrm{Ei}(-5\sigma_1));\ \
w^{+}_{2}{'}(r_{7})=\\\frac{2e^{-r_{7} \sigma  }}{25 \sigma  }(\sigma (2a_0e^{-5}-2-5\sigma)  +\mathrm{Ei}(-5)a_0(2+5  \sigma)+a_0e^{5 \sigma  }(5\sigma ^2+3 \sigma -2)\mathrm{Ei}(-5\sigma_1))
\end{multline}
\begin{lem}\label{L14}
(i)   Let $z=2\sigma -1$. The functions  $w^{+}_{1}$ and $w^{+}_{1}{'}$ satisfy the estimates
{\small \begin{multline}\label{tsz}
\left|\frac{79}{691}-\frac{737 }{2580}\frac{z}{2}+\frac{147 }{412}\frac{z^2}{2^2}-\frac{49 }{165}\frac{z^3}{2^3}+\frac{103 }{556}\frac{z^4}{2^4}-\frac{131 }{1419}\frac{z^5}{2^5}+\frac{13 }{340}\frac{z^6}{2^6}-\frac{6 }{445}\frac{z^7}{2^7}-w^{+}_{2}(r_{7})\right|<3\cdot10^{-5}\\
\left|-\frac{52}{509}+\frac{26 }{185}\frac{z}{2}-\frac{10 }{313}\frac{z^2}{2^2}-\frac{80 }{857}\frac{z^3}{2^3}+\frac{41 }{307}\frac{z^4}{2^4}-\frac{41 }{392}\frac{z^5}{2^5}+\frac{66 }{1109}\frac{z^6}{2^6}-\frac{13 }{480}\frac{z^7}{2^7}-w^{+}_{2}{'}(r_{7})\right|<6\cdot10^{-5}
\end{multline}}
(ii) We have $|w^{+}_2(r_{7})|\leq 21/50$ and $|w^{+}_2{'}(r_{7})|<19/100$, which will be used in estimating the Wronskian of the two possible eigenfunctions.
\end{lem}
\begin{proof}
The polynomials in \eqref{tsz} are simply truncates
of the Taylor series of the functions involved. These, and the estimates, are obtained as follows. Denoting $w^{+}_3(\sigma)=\sigma w^{+}_2(r_{7};\sigma)$ we have $w^{+}_2(r_{7};\sigma)=\int_{0}^{1}w^{+}_3{'}(\sigma s)\,ds$. We then approximate $w^{+}_3{'}(\sigma s)$ using a Taylor polynomial around $\sigma =\tfrac{1}{2}$ with rigorous bounds for the remainder, using Cauchy's formula. This
can be obtained by expanding the exponential functions and expanding
the integrands and then integrating the series term by term. For
example, $\f{e^{-5-u}}{-5-u}$ differs from the sum of
the first 12 terms of its Taylor series expansion in $u$ at $u=\f52$
by no more than $2\cdot 10^{-7}$ (this follows from by Cauchy's integral
formula).  Hence $\mathrm{Ei}(-5-5\sigma s)=\mathrm{Ei}(-10)+\int_{5\sigma s}^{5}\f{e^{-5-u}}{-5-u}\, du$
differs from the integral of the Taylor polynomial by no more than
$10^{-6}$. After obtaining a polynomial approximation of $w^{+}_2(r_7)$ in
this way, we re-expand it in $z=2\sigma-1$, and it so turns out
that the coefficients of $z^k$ for $k>7$ are manifestly
small. Discarding them and using rational approximations of the remaining
coefficients we obtain the polynomial in Lemma~\ref{L14}. The result
for $w^{+}_2{'}(r_7)$ follows in a similar way. The proof of (ii) follows from (i) using Note~\ref{met1} with the partition ${\boldsymbol{\pi}}=(0,1/2,1)$.
\end{proof}

\begin{cor}\label{C15} Let  $W[w^{+}_1,w^{+}_2](r_{7})=w^{+}_{1}(r_{7})w^{+}_{2}{'}(r_{7})-w^{+}_{2}(r_{7})w^{+}_{1}{'}(r_{7})$. We have
  \begin{equation}
    \label{eq:firstW}
    \sup_{\lambda\in[0,1]}\Big|W[w^{+}_1,w^{+}_2]\Big| \geqslant 48\cdot  10^{-4}
  \end{equation}
\end{cor}
\begin{proof}
  We substitute $\lambda=1+\sigma^2$ in~\eqref{w1pdef}, and use the
  polynomials in \eqref{tsz} to calculate the Wronskian within an
  accuracy of~$\pm 4\cdot10^{-5}$ (obtained by crudely bounding
  away the effects of the errors on the right-hand side of~\eqref{tsz}).  For
  estimating the resulting polynomials, we write $\sigma
  =\tfrac12+\tfrac12z$, re-expand and use Note~\ref{met1}.
\end{proof}

\subsubsection{End of the proof of Proposition~\ref{P3.2}: the
  Wronskian of the actual solutions} \begin{lem} \label{417} The
  Wronskian $W[u_1,u_2](r_{7})$ of the actual solutions satisfies
\begin{equation}
  \label{eq:eq54*}
  |W[u_1,u_2](r_{7})-W[w^{+}_1,w^{+}_2](r_{7})|\leqslant  5\cdot 10^{-5}
\end{equation}
\end{lem}
\begin{proof}
  This follows by estimating $W[u_1,u_2]$ via
   \eqref{eq:firstW}.  This is straightforward and uses Corollary
  \ref{d1d}, Lemma~\ref{w1d}, Corollary \ref{crbd}  and Lemma
  \ref{L14} (ii) to bound $|u_{1,2}-w^{+}_{1,2}|$,
  $|u'_{1,2}-\partial_r w^{+}_{1,2}|$ as well as $|u_{1,2}|$, $|u'_{1,2}|$,
  $|w^{+}_{1,2}|$ and $|\partial_r w^{+}_{1,2}|$.
\end{proof}

\section{The Operator $L_-$}\label{sec:lmm}

The second main result of this paper is the following one, which establishes the gap property for $L_-$.

\begin{thm}\label{P1m}
The operator $L_-$ has no eigenvalue or resonance for $\lambda$ in the interval $(0,1]$ in $L_2(\mathbb{R}^+)$.
\end{thm}
\zn As in the  case of $L_+$, there are
two solutions $y_1(r;\lambda)$ and $y_2(r;\lambda)$ of the equation \eqref{eq:eqlm}
with the properties $y_1(0;\lambda)=1$ and $y_2(r;\lambda)=r^{-1}e^{-r\sqrt{1-\lambda}}(1+o(1)), \,\,r\to\infty$.
Let $W[y_1,y_2](r;\lambda)=y_1y_2'-y_2y_1'$ be the Wronskian of these two special
solutions.  Theorem~\ref{P1m}
is a corollary to the following result.

\begin{prop}\label{p42}
One has the lower bound
 \begin{equation}
   \label{eq:eqWm}
\sup_{\lambda\in [0,1]} \lambda^{-1} |W[y_1,y_2](r_{7};\lambda)|\geqslant {17}/{1000}
 \end{equation} Therefore $y_1$ and $y_2$ are linearly independent  for all $\lam\in(0,1]$.
\end{prop}
\subsection{Proofs}
Let
\begin{equation}
  \label{eq:eqy-u}
   \tilde{u}_1(r)=\lam^{-1}\left[y_1(r)-\f{Q(r)}{Q(0)}\right]; \ \   \tilde{u}_2(r)=\f{1}{\lambda}\left[y_2(r)-\f{1}{A_1}{Q(r)}\right]\  \text{where }\ A_1=A+b_1
\end{equation}
(see Lemma~\ref{210}). Then  $\tilde{u}_1$ satisfies
\begin{equation}\label{lmm}
-u''-2u'/r+(1-\lambda-Q^{2})u=Q/Q(0)
\end{equation}
We construct a pair of functions that agree with $y_1$ and $y_2$
within relatively small errors. To this effect,
we first define ${g}^{-}_{1,2}$ in the following way. Consider the piecewise polynomials
 \begin{equation}
   \tilde{g}^{-}_{1}(r)=\sum_{(j,k,l)\in S}d_{kl;j}^-\boldsymbol{\chi}_j\lambda^{k}z^l;\ \ \tilde{g}^{-}_{2}(r)=\sum_{(j,k,l)\in S}e_{kl;j}^-\boldsymbol{\chi}_j\lambda^{k}z^l
\end{equation}
where $S=\{(j,k,l)\mid 1\leqslant j\leqslant 3, \ 0\leqslant k\leqslant M_j,\ 0 \leqslant l\leqslant 15\}$,  $d_{kl;j},e_{kl;j}$ are given in the appendix, $\boldsymbol{\chi}_j$, $j=1,2,3$,
are the characteristic functions of   $[0,r_{5})$, $[r_{5},r_1)$ and $[r_1,r_{7}]$ respectively, $M_j\leqslant 15$ and  $z$ depends on  $r$ as specified in the top rows of the tables in the appendix.
To ensure that ${g}^{-}_{1,2}$ are $C^1$ we next let
\begin{multline}
  \hat{g}^{-}_j(r)=  \tilde{g}^{-}_j(r)\h_{{47}}+\L\tilde{g}^{-}_j(r)+\tilde{g}^{-}_j(1)-\tilde{g}^{-}_j(1{-})+(\tilde{w}'_1(1)-\tilde{w}'_1(1{-}))(r-1) \J\h_{14}\\
+\L \tilde{g}^{-}_j(r)+\hat{g}^{-}_j(r_{1})-\tilde{g}^{-}_j(r_{1}{-})+5(\hat{g}'_j(r_{1})-\tilde{g}'_j(r_{1}{-}))(r^{2}-r_{1}^{2})/3\J\h_{01}
  \end{multline}
where $j=1,2$.
We let ${g}^{-}_j(r)=\hat{g}^{-}_j(r)/r^{j-1}$ and $w^-_1(r)=({g}^{-}_{1}(r)-{g}^{-}_{1}(r)|_{\lam=0})/\lambda$.\footnote{Note that for $L_+$ we constructed $w^+_1(r)$ using a different piecewise representation due to the high accuracy required. Here, however, it is sufficient to use ${g}^{-}_{1}(r)$ to define $w^-_1(r)$.  }

\begin{lem}
\label{bdm}
The following bounds hold:
\begin{equation}
  |w^{-}_1(r)| \leqslant {1}/{100}+{3r}/{25} \quad \forall\; 0\leqslant r\leqslant r_{7}
\end{equation}
 \begin{equation}\label{eq:bdgm}
   r|{g}^{-}_1(r)| \leqslant \tfrac{11r}{10}\h_{01}+\tfrac{r+3}{10}\h_{{17}};\ \ r|{g}^{-}_2(r)|\leqslant
\tfrac{11}{10}\h_{01} +\tfrac{23r}{10}\h_{14}+ \tfrac{22r}{5}\h_{{47}}
  \end{equation}
   \begin{equation}
     \label{eq:bdp}
      |{g}^{-}_1{'}(r)| \leqslant \tfrac{13}{10}\h_{04}+\tfrac12\h_{{47}};\ \  r^2|{g}^{-}_2{'}(r)| \leqslant  \tfrac85\h_{04}+ 3r^2\h_{{47}}
   \end{equation}
  \end{lem} \begin{proof}
We check this using Note~\ref{met2} and the partition $\boldsymbol{\pi}=\(0,r_{1},\f35,r_{4},r_{5},2,r_{7}\)$. For  ${g}^{-}_1$ in $[0,r_1]$ we first divide
\eqref{eq:bdgm} by $r$. \end{proof}
\zn Define
\begin{equation}
  \label{eq:eqR1}
  R_1=-w^{-}_{1}{''}-2w^{-}_{1}{'}/r+(1-\lambda-Q^{2})w^{-}_{1}-Q/Q(0)
\end{equation}
and
\begin{equation}
  \label{eq:eqR00}
  \tilde{R}_0=-w^{-}_{1}{''}-2w^{-}_{1}{'}/r+(1-\lambda-\tilde{Q}^{2})w^{-}_{1}-\tilde{Q}/\tilde{Q}(0)
\end{equation}
\begin{lem}\label{64} We have $|\tilde{R}_0|\leqslant 10^{-5}(3\h_{01}+2\h_{{17}})$
and $|R_1|\leqslant 10^{-5}(6\h_{04}+3\h_{{47}}) $.
\end{lem}
\begin{proof} Note that
\[
|R_1|\leqslant |\tilde{R}_0|+|(Q^{2}-\tilde{Q}^{2})w^{-}_{1}|+|Q-\tilde{Q}|/Q(0)\]
We start with $\tilde{R}_0$, for which we use the same idea as in Lemma~\ref{lpr1} (we kept the same notations although the functions are
different; this should cause no confusion). Due to the  the monotonicity of $\tilde{Q}$ and the fact that $\tilde{R}_0(r)/\tilde{Q}(r)^2$ is a polynomial we have
 $$|\tilde{R}_0(\ell_k(z))|\leqslant |\tilde{R}_0(\ell_k(z))\tilde{Q}^2(\ell_k(-1))/\tilde{Q}(\ell_k(z))^2|\leqslant 3\cdot10^{-5}$$
using the method in Note~\ref{met2}  and  the partition given by
$\boldsymbol{\pi}=\(0,\f{3}{25},\f{6}{25},\f{3}{10}\)$.

Similarly on $[r_1, r_7]$ we use the partition $ \boldsymbol{\pi}=\(r_{1},\f{2}{5},\f{3}{5},\f{4}{5},\f{9}{10},1,\f{3}{2},\f{19}{10},
\f{11}{5},r_{6},r_{7}\)$ and obtain
$$|\tilde{R}_0(\ell_k(z))|\leqslant |\ell_k(z)/\ell_k(-1) \tilde{Q}^2(\ell_k(-1))/\tilde{Q}(\ell_k(z))^2\tilde{R}_0(\ell_k(z))|\leqslant 2\cdot 10^{-5}$$ The rest of the proof is relatively
straightforward, and the details are given in Section~\ref{d64}.
\end{proof}

 The functions ${g}^{-}_1$, ${g}^{-}_2$ solve a second order equation
\begin{equation}
  \label{eq:unkn*}
 -g''+(A(r)-2/r)g'+(B(r)+A(r)/r)g=0
\end{equation}
where
$$A(r)=\frac{\hat{g}^{-}_2(r) \hat{g}^{-}_1{''}(r)-\hat{g}^{-}_1(r) \hat{g}^{-}_2{''}(r)}{\hat{g}^{-}_2(r) \hat{g}^{-}_1{'}(r)-\hat{g}^{-}_1(r) \hat{g}^{-}_2{'}(r)} \text{ \rm and } B(r)=\frac{\hat{g}^{-}_1{'}(r) \hat{g}^{-}_2{''}(r)-\hat{g}^{-}_2{'}(r) \hat{g}^{-}_1{''}(r)}{\hat{g}^{-}_2(r) \hat{g}^{-}_1{'}(r)-\hat{g}^{-}_1(r) \hat{g}^{-}_2{'}(r)}$$
\begin{lem}\label{65}
   We have the following bounds on $[0,r_7]$
   \begin{equation}\label{BABr}
     |A(r)|\leqslant 2\cdot 10^{-4}\big(\h_{[0,\frac{11}{5})}(r)+3\h_{[\f{11}{5},r_7]}(r)\big);\ \ |B(r)-1+\lambda+\tilde{Q}^{2}(r)|<{3}/{2500}
   \end{equation}
   \end{lem}
\begin{proof}
We use the
partition
$\boldsymbol{\pi}=\(0,\f1{10},\f15,r_1,\f{8}{25},\f25,\f12,\f35,\f45,\f9{10},\f{23}{25},1,\f{13}{10},
\f{9}{5},\f{21}{10},\f{11}{5},r_{6},r_{7}\)$ and as usual maximize the numerators and minimize the denominators. \end{proof}
\begin{Note}
 We now look for  a nearby actual solution $\tilde{u}_1=w^{-}_1-\delta$, $
\delta(0)=0$. In the following we write  $\|f\|$ for  $\sup_{[0,r_{7}]} |f|$.
\end{Note}
 \noindent  By \eqref{lmm} and \eqref{eq:eqR1} $\delta$ satisfies
$-\delta''-\f{2}{r}\delta'+(1-Q^2-\lambda) \delta+R_1=0$, or
\begin{equation}
  -\delta''+(A(r)-2/r)\delta'+(B(r)+A(r)/r)\delta=h_2(r)
\end{equation}
where $h_2(r)=R_1(r)+A(r)\delta'(r)
  +\Big[B(r)-1+\lam+A(r)/r+Q(r)^{2}\Big]\delta(r)$.
\begin{lem}
  We have
  \begin{multline*}
     |h_{2}(\delta,r)|\leqslant 10^{-4}\left\{\L\f{3}{5}+2\|\delta'\bigr\|+(20+\f{2}{r})\|\delta\|\J\h_{01}+\left( \f{3}{5}+2\|\delta'\|+22\|\delta\|\right)\h_{14}\right.\\\left.\left( \f{3}{10}+2\|\delta'\|+15\|\delta\|\right)\h_{[r_4,\tfrac{11}5)}+\left( \f{3}{10}+6\|\delta'\|+15\|\delta\|\right)\h_{[\tfrac{11}5,r_7]}\right\}
  \end{multline*}
 \end{lem}
\begin{proof}
We write \begin{multline}
h_2(r)\leqslant |R_1(r)|+|A(r)||\de'(r)|
+\L|B(r)-1+\lam+\tilde{Q}(r)^{2}|\right.\\+\left.|A(r)|/r+|Q(r)^{2}-\tilde{Q}(r)^{2}|\J|\de(r)|$$
 \end{multline}
The  estimates now follow from the corresponding bounds satisfied by $Q^2-\tilde{Q}^2$, $R_1, A$ and  $B(r)-1+\lambda+\tilde{Q}^{2}(r)$, which are established in \eqref{3q2}, Lemma~\ref{64} and Lemma~\ref{65}. On every subinterval we use the monotonicity of the exponential and replace it by the
value at the left endpoint. On the intervals $[r_1,r_4] ,[r_4,11/5],[11/5,r_7]$ we  simply replace $1/r$ by its
left endpoint value. The resulting rational functions are of low degree and are maximized explicitly.
\end{proof}
\zn As in the $L_+$ case $\delta$ satisfies
\begin{equation}
  \label{eq:syst2}
  \delta =H_0(\delta),\ \ \delta'=H'_0(\delta)
\end{equation}
 where
$H_0,H_0'$ are of the form given in \eqref{eq:eqH0} and \eqref{eq:eqH0p} with $g^{+}_{j}$ replaced by $g^{-}_{j}$.
\begin{lem}
\label{lpconm}
(i)   We have
  \begin{equation}
    \label{eq:estN0}
 N_0:=     |H_{0}(\delta)(r)|+|H_{0}(\delta)'(r)|/5 \leqslant  5\cdot10^{-4}+\f{23\|\de\|_{\infty}}{1000}+5\cdot10^{-3}\|\de'\|_{\infty}
  \end{equation}
Consequently,  $(H_0,H_0')$ is a contraction in the ball
\[
X:= \big\{ f\in C^{1}((0,r_{7}))\mid \|f\|_{X}\leqslant 6\cdot 10^{-4}\}
\]
where
$\| f \|_{X}:=\sup_{r\in[0,r_{7})}\, \left(|f(r)|+\tfrac15|f'(r)|\right)$. Thus
there  in an actual solution $\hat{u}_1$ within $6\cdot 10^{-4}$ in $\|\cdot\|_X$ of $w^{-}_1$.

(ii)  $ |w^{-}_1{'}(r_{7})-\hat{u}_1{'}( r_{7})|=|\delta'(r_7)|<6\cdot10^{-4}$ (cf.~again footnote~\ref{f2} on p.~\pageref{f2}).
\end{lem}
\begin{proof}
     Since $1/W(r):=({g}^{-}_1(r){g}^{-}_2{'}(r)-{g}^{-}_1{'}(r){g}^{-}_2(r))^{-1}=r^{2}\exp\Big(\int_{r}^{5/2}
A(s)ds\Big)$, we have
$$1/|W(r)|<{51r^2}/{50}$$
The proof follows from a piecewise analysis of $N_0$. This is done by
straightforward integration of the polynomials involved in $N_0$ and
then maximization of the rational functions (whose numerators are of degree at most 5)
multiplying $\|\delta\|$, $\|\delta'\|$ and of the free term in
the result of the integration.  Taking the derivative of the rational
functions and re-expanding the numerators at the left endpoint of each
interval, we see that all these re-expanded
polynomials have positive coefficients. Thus the maximum is obtained
by evaluating $N_0$ at the right endpoint of each interval. We provide
the intermediate bounds; the largest one, from which
the result follows,  is on the last interval.
\begin{multline*}
   N_0\leqslant 10^{-4}\Big \{ ( 1/10+5\|\de\|+3/10\|\de'\|)\h_{01}+(7/10+23\|\de\|+3\|\de'\|)\h_{14}\\+ (4+170\|\de\|+22\|\de'\|)\h_{[r_4,11/5)} +(5+230\|\de\|+47\|\de'\|)\h_{[11/5,r_7]} \Big\}
\end{multline*}
  (ii) This follows in the same way as \eqref{eq:estN0}, by estimating $H'_0$ on $[11/5,r_{7}]$.
\end{proof}
\begin{Note}\label{N4.9}{\rm
We have  $\hat{u}_1(0)=w^{-}_1(0)=P_5(\lambda)$, a quintic polynomial.
Substituting $\lambda=1/2(1+z)$ and using Note  \ref {met1} to estimate
this quintic, we get
\begin{equation}
  \label{eq:bd4}
 |\hat{u}_1(0)|<2\cdot 10^{-6}
\end{equation}
Thus $Q/Q(0)+\lam \hat{u}_1$ is close to but not exactly equal to the function $y_1$ defined immediately after Theorem~\ref{P1m}. We will  modify it multiplicatively to make up for the discrepancy.}
\end{Note}
\begin{lem}
 The function
\end{lem}
\begin{equation}
  \label{eq:equu}
\tilde{u}_1(r)=\hat{u}_1(r)-\hat{u}_1(0)\frac{Q(r)/Q(0)+\lambda \hat{u}_1(r)}{1+\lambda \hat{u}_1(0)}
\end{equation}
is the solution of  \eqref{lmm} with $\tilde{u}_1(0)=0$ (and thus $Q/Q(0)+\lam \hat{u}_1=y_1$).
\begin{proof}
  Clearly, $\tilde{u}_1(0)=0$.  Using \eqref{mainODE}, it is straightforward  to check that $\tilde{u}_1$ solves \eqref{lmm}.
\end{proof}

\begin{lem}\label{w1dm}
One has the estimates $
|w^{-}_1(r_{7})|<\f25$ and  $|w^{-}_1{'}(r_{7})|<\tfrac{1}{4}$.
\end{lem}

\begin{proof}
The bound for $w^{-}_1(r_{7})$ follows from Lemma~\ref{bdm}. We estimate
the quintic polynomial $w^{-}_1{'}(r_{7})$ in $\lambda$ by taking $\lambda=1/2(1+z)$ using again Note  \ref {met1}.
\end{proof}

\begin{prop}\label{acts}
The solution $\tilde{u}_1$ satisfies
$$ |w^{-}_1(r_{7})-\tilde{u}_1( r_{7})|<7\cdot10^{-4};
\ |w^{-}_1{'}(r_{7})-\tilde{u}_1{'}( r_{7})|<7\cdot10^{-4}$$
\end{prop}
\begin{proof}
Note that Lemma~\ref{LQ} and the definition of $\tilde{Q}$ imply $|Q(r_7)/Q(0)|<\f{3}{100}$ and $|Q'(r_7)/Q(0)|<\f{3}{100}$, and \eqref{eq:bd4} and \eqref{eq:equu} imply the (crude) bounds $|\tilde{u}_1(r_{7})-\hat{u}_1( r_{7})|<3\cdot10^{-6}(\f{3}{100}+|\hat{u}_1( r_{7})|)$ and $|\tilde{u}_1'(r_{7})-\hat{u}_1'( r_{7})|<3\cdot10^{-6}(\f{3}{100}+|\hat{u}_1'( r_{7})|)$. The rest is straightforward from
Lemma~\ref{lpconm} and Lemma~\ref{w1dm}.  \end{proof}

\subsubsection{The region $r>r_{7}$: the quasi-solution bounded for large $r$}\label{S7}
Let $\sigma =\sqrt{1-\lambda}$,  and $a_{1}=({6839}/{2521})^2$, $\sigma_1=1+\sigma$, $\sigma_2=1-\sigma$.
 We consider
\begin{equation}\label{g3}
  {g}^{-}_3(r)=(\sigma r)^{-1}{e^{- \sigma  r}}( \sigma  -a_{1}\mathrm{Ei}(-2r)\\
  +a_{1}e^{2 \sigma  r}\sigma_1 \mathrm{Ei}(-2\sigma_1r));\ w^{-}_2(r)=\lambda^{-1}({{g}^{-}_{3}(r)-{g}^{-}_{30}(r)})
\end{equation}
where
\begin{equation}
  \label{eq:g30}
  {g}^{-}_{30}(r)={g}^{-}_{3}(r)|_{\lam=0}=e^{-r}/r+a_{1}g(r)
\end{equation}
Since $\lambda=1-\sigma ^2$ we have
\begin{equation}\label{eqw2}
  R_2(r):=  -w^{-}_2{''}-2w^{-}_2{'}/r+(1-\lambda-a_{1}r^{-2}{e^{-2r}})w^{-}_2-{g}^{-}_{30}=\frac{a_{1}^2e^{-4r}}{ \sigma_1r^{3}}\tilde{R}_2
\end{equation}
where
\begin{equation}
  \label{eq:eq48*}
\tilde{R}_2(r)
=(\sigma  \sigma_2)^{-1}{e^{-\sigma r}}(2\sigma e^{(3+\sigma) r}\mathrm{Ei}(-4r)+(e^{2r}-\sigma e^{\sigma_1 r
})\mathrm{Ei}(-2r)-e^{2\sigma_1  r}\sigma_1\mathrm{Ei}(-2\sigma_1r))
\end{equation}

\begin{lem}\label{r2wm}
We have $|R_2(r)|<\f{3}{250}e^{-2r}/{r^3}$ and
$|w^{-}_2(r)|<\f{51}{50}e^{-\sigma r}$.
\end{lem}
\begin{proof}
Rewriting $\Ei(-r)$ as $e^{-r}\int_{-\infty}^0e^s/(s-r)\, ds$ we obtain
\begin{equation*}
\tilde{R}_2(r) = e^{-r} \int_{-\infty}^0e^s\frac{\left(1-e^{\sigma_2  r}\right) s^2+2 r s \left(2(e^{\sigma_2  r}-1)+ \sigma_2 \right) }{\sigma_2 (s-4 r) (s-2 r) (s-2 \sigma_1  r)}\, ds
\end{equation*}
We use the following inequalities
to bound $\tilde{R}_2(r)$:  For $\sigma\in (0,1), r>0$, we have  $(e^{\sigma_2 r}-1)\leqslant \sigma_2 e^r$; furthermore, $s<0$,  and thus  $|s-2 \sigma_1  r|\geqslant |s-2r|$ and $|r/(s-4r)|<1/4$;
we are in the range $r\geqslant  r_7$, thus $e^{-r}\leqslant e^{-r_7}$,
  and with $m>0$, $|s-mr|\geqslant |s-mr_7|$. Therefore,
\begin{multline}
  \label{eq:Estr2}
|\tilde{R}_2(r)|\leqslant   \int_{-\infty}^{0}\left|\frac{s^2e^s }{(s-10)(s-5)^2 }\right|+\left|\frac{2(2+e^{-r_{7}})se^s}{4 (s-5)^2}\right|\, ds=-4e^{10} \mathrm{Ei}(-10)+2e^5\mathrm{Ei}(-5)\\
-3\mathrm{Ei}(-5)e^{r_7}-\tfrac{\ds e^{-r_7}}2
\ \ \ \ \Rightarrow \ |R_2|<\f{3a_{1}^2e^{-5}}{100}\frac{e^{-2r}}{r^3}<\f{3}{250}  \frac{e^{-2r}}{r^3}
\end{multline}
 Using \eqref{g3} and \eqref{eq:eq48*}) we get
$
re^r\sigma_1  w^{-}_2(r)=({e^{\sigma_2 r}-1})/{\sigma_2  }-a_{1}e^{-r}\tilde{R}_2
$. Therefore,  since $\forall \sigma_2 r\geqslant 0$ we have ${e^{\sigma_2 r}-1}\leqslant {\sigma_2  } r e^{\sigma_2 r}$ we obtain, using the estimate for $\tilde{R}_2$ in \eqref{eq:Estr2},
\begin{equation}
  \label{eq:w2b}
  |w^{-}_2(r)|\leqslant \frac{e^{-\sigma r}+e^{-2r}a_{1}|\tilde{R_2}|}{r\sigma_1 }< |w^{-}_2(r)|\leqslant\frac{re^{-\sigma r}+e^{-\sigma r-5/2}a_{1}|\tilde{R_2}|}{r\sigma_1}\leqslant
\frac{51e^{-\sigma r}}{50}
\end{equation}
and we are done.
\end{proof}

\noindent Since $Q$ satisfies \eqref{mainODE} and $y_2$ is a solution of \eqref{eq:eqlm},
the definition of $\tilde{u}_2$ (see \eqref{eq:eqy-u} and \eqref{A B}) implies that $\tilde{u}_2$ satisfies the equation
\begin{equation}
  \label{eq:equ2}
  -\tilde{u}_2{''}-2\tilde{u}_2{'}/r+(1-\lambda-Q^{2})\tilde{u}_2=Q/A_1
\end{equation}
Writing $\tilde{u}_2=w^{-}_2+\delta_0$, we obtain from \eqref{eqw2} and \eqref{eq:equ2} after regrouping the terms,
\begin{equation}
  \label{eq:eq50m}
  -\delta_0''-2\delta_0'/r=-R_2+\(\lambda-1+Q^2\)\delta_0+\left(Q^{2}-a_{1}\frac{e^{-2r}}{r^{2}}\right)w^{-}_2+\f{Q}{A_1}-{g}^{-}_{30}=:h_{3}
\end{equation}
Looking  for exponentially decreasing solutions,  we write \eqref{eq:eq50m}
 in the integral form
\begin{equation}
  \label{eq:eqd2m}
  \delta_0 =H_1 (\delta_0)=T_1 (h_3)\ \text{where}\  [T_1(f)](r)=-\int_r^{\infty}\f{dt}{t^2}\int_t^\infty s^2 f(s)\, ds
\end{equation}
Using \eqref{eq:bdQ2} and Note~\ref{NdQ} we get
\begin{equation}\label{estdq}
 \left|Q^{2}(r)-a_{1}\frac{e^{-2r}}{r^{2}}\right|\leqslant \left|Q(r)^{2}-\tilde{Q}(r)^{2}\right|+\left|a_{1}\frac{e^{-2r}}{r^{2}}-\tilde{Q}(r)^{2}\right|<\f{3}{200}\frac{e^{-2r}}{r^{2}}
\end{equation}
Using Lemma~\ref{210} to estimate $Q$ in terms
of $\tilde{Q}$ and to bound $b_{1,2}$,   Definition~\ref{def:tildeQ} and \eqref{tildeQ} to estimate $\tilde{Q}$, \eqref{4p} to estimate $g$, \eqref{A B error} for  $A$ and $B$, \eqref{eq:eqy-u} for $A_1$  and \eqref{eq:g30} we get, for $r\geqslant 5/2$,
\begin{equation}\label{estlt}
\Big|{Q}(r)/{A_1}-{g}^{-}_{30}(r)\Big|=\Big|\f{Bg(r)+b_2(r)}{A_1}-a_1g(r)\Big|\leqslant \Big|\f{b_2(r)}{A_1}\Big|+\Big|\(\f{B}{A_1}-a_1\)g(r)\Big|\leqslant \f{1}{20}\f{e^{-3r}}{r^3}
\end{equation}
In the same way it is shown that for $r\geqslant  r_7$ we have $Q^2<\f1{125}$ which implies $|1-\lambda-Q^{2}|\leqslant 1$.

Therefore denoting $\|f\|=\sup_{r\geqslant r_{7}}r^2e^{2r}|f(r)|$ we use Lemma~\ref{r2wm}, \eqref{eq:eq50m}, \eqref{estdq} and \eqref{estlt} to obtain
\begin{equation}
  \label{eq:esth3}
 |h_3(r)|<\(\f{11}{500}+\|\de_0\|\)\f{e^{-2r}}{r^2}
\end{equation}

\begin{lem}\label{w2dm}
$H_1$ is contractive in the ball $\{f\mid \|f\|\leqslant 1/125\}$.
Furthermore, $|\tilde{u}_2(r_7)-w^{-}_2(r_7)|<10^{-5}$ and $|\tilde{u}_2'(r_7)-w^{-}_2{'}(r_7)|<2\cdot10^{-5}$.
\end{lem}
\begin{proof}
We use the crude estimate $ |T_1 (\exp(-2r)/r^2)|< r^{-2}
\int_r^{\infty}dt\int_t^\infty\exp(-2s)\,ds=\tfrac14$  and obtain from \eqref{eq:esth3}
  \begin{equation}\label{Hpb1}
       |H_1 (\delta_0)|(r)  \leqslant \f{e^{-2r}}{r^2}\(\frac{11}{2000}+\frac{1}{4}\|\de_0\|\)
  \end{equation}
\end{proof}

\zn
We simply bound  $|H_1'(\delta)|=\int_r^\infty(s/r)^2 h_3(s)\,ds$  using $|H_1'(\exp(-2s)/s^2)|=\f12r^{-2}e^{-2r}$ and get
\begin{equation}\label{Hpb2}
|H'_{1}(\delta_0)(r)|\leqslant \(\f{11}{1000}+\f12\|\de_0\|\){e^{-2r}}/{r^2}
\end{equation}
It follows from \eqref{Hpb1} and \eqref{Hpb2} that
$|\delta_0(r_{7})|=|H_{1}(\delta_0)(r_{7})|<10^{-5},
|\delta_0'(r_{7})|=|H'_{1}(\delta_0)(r_{7})|<2\cdot10^{-5}$.

\subsubsection{The Wronskian} The formulas for $w^{-}_2$ and $w^{-}_2{'}=\partial_{r} w^{-}_2$ are
\begin{multline}
  w^{-}_2(r_7)=\f{2e^{-r_7}}{5\lambda}(e^{r_7\sigma_2 }-1)-\f{4a_1}{5\lambda}\mathrm{Ei}(-10)e^{r_7}\\-\f{2a_1e^{-r_7}}{5\lambda\sigma}(e^{\sigma_2 r_7}-\sigma)\mathrm{Ei}(-5)+\f{2a_1\sigma_1}{5\lambda\sigma}e^{r_7\sigma}\mathrm{Ei}(-5\sigma_1)
\end{multline}
\begin{multline}
w^{-}_2{'}(r_7)=-\f{2}{25\lambda}(2e^{-r_7\sigma }-7e^{-r_7}+2a_1e^{-3r_7}+5\sigma e^{-r_7\sigma }-2a_1e^{-5-r_7\sigma })-\f{12a_1}{25\lambda}\mathrm{Ei}(-10)e^{r_7}\\+\f{2a_1\mathrm{Ei}(-5)}{25\lambda\sigma}(-7\sigma e^{-r_7}+2e^{-r_7\sigma}+5\sigma e^{-r_7\sigma})+\f{2a_1\sigma_1}{25\lambda\sigma}e^{r_7\sigma}(5\sigma-2)\mathrm{Ei}(-5\sigma_1)
\end{multline}
 It is useful to estimate these in terms of polynomials.
\begin{lem}\label{L14m} (i) With $z=2\sigma -1$ we have
{\small \begin{multline}
\left|\frac{61}{560}-\frac{139}{588}z+\frac{97}{316}\f{z^{2}}{2^2}-\frac{124}{409}\f{z^{3}}{2^3}+\frac{199}{786}\f{z^{4}}{2^4}
-\frac{73}{383}\f{z^{5}}{2^5}+\frac{71}{526}\f{z^{6}}{2^6}-\frac{16}{173}\f{z^{7}}{2^7}+\frac{24}{385}\f{z^{8}}{2^8}-w^{-}_{2}(r_{7})\right|<3\cdot10^{-4}
\\
\left|-\frac{23}{303}+\frac{52}{587}z-\frac{3}{103}\f{z^{2}}{2^2}-\frac{13}{311}\f{z^{3}}{2^3}+\frac{19}{235}\f{z^{4}}{2^4}
-\frac{22}{257}\f{z^{5}}{2^5}+\frac{21}{292}\f{z^{6}}{2^6}-\frac{13}{242}\f{z^{7}}{2^7}+\frac{59}{1563}\f{z^{8}}{2^8}-w^{-}_{2}{'}(r_{7})\right|<2\cdot10^{-4}
\end{multline}}
(ii) We also have $|w^{-}_2(r)|<1/2$ and $|w^{-}_2{'}(r)|<1/5$.
\end{lem}
\begin{proof}
We consider
 $$ w^{-}_3(\sigma_2)=\f{2e^{-r_7}}{5}(e^{r_7\sigma_2 }-1)-\f{4a_1}{5}\mathrm{Ei}(-10)e^{r_7}-\f{2a_1e^{-r_7}}{5}(e^{\sigma_2 r_7}-\sigma)\mathrm{Ei}(-5)+\f{4a_1}{5}e^{r_7\sigma}\mathrm{Ei}(-5\sigma_1)
 $$
 and
$$w^{-}_4(\sigma)=-\f{2a_1e^{-r_7}}{5\sigma_1}(e^{\sigma_2 r_7}-\sigma)\mathrm{Ei}(-5)+\f{2a_1}{5\sigma_1}e^{r_7\sigma}\mathrm{Ei}(-5\sigma_1)$$
A direct calculation shows that $w^{-}_2(r_7;\sigma)=\lambda^{-1}w^{-}_3(\sigma_2)+\sigma^{-1} w^{-}_4(\sigma)
=\sigma_1^{-1}\int_{0}^{1}w^{-}_3{'}(\sigma_2 s)\, ds+\int_{0}^{1}w^{-}_4{'}(\sigma s)\, ds$.
The rest of the proof  is the same as that of Lemma~\ref{L14}. The calculations for $w^{-}_2{'}$ are similar.
\end{proof}

\begin{Note}{\rm
To estimate  the Wronskian $W[y_1,y_2](r_7;\lam)$, we express
it in terms of $\tilde{u}_{1,2}$ using \eqref{eq:eqy-u},
  \begin{multline}\label{Wy12}
\lambda^{-1}W[y_1,y_2](r_7;\lam)=   \f{\tilde{u}_{1}(r_{7}) {Q}'(r_{7})-\tilde{u}_{1}  {'}(r_{7}){Q}(r_{7})}{A_1}+\f{{Q}(r_{7})\tilde{u}_{2}{'}(r_{7})-{Q}'(r_{7})\tilde{u}_{2}(r_{7})}{{Q}(0)}\\
+\lam\(
\tilde{u}_{1}(r_{7})\tilde{u}_{2}{'}(r_{7})-\tilde{u}_{2}(r_{7})\tilde{u}_{1}{'}(r_{7})\)
  \end{multline}
 At the first stage, we look for  a nearby quantity solely containing {\em polynomials} with rational coefficients; we find such an approximation by
replacing in \eqref{Wy12} $\tilde{u}_{1,2}$ by $w^{-}_{1,2}$, $Q$ by $\tilde{Q}$, and $A_1$ by $A$ (cf. \eqref{eq:eqy-u}).
Let thus
 \begin{multline}\label{w0}
 \lambda^{-1}\tilde{W}(\lambda):=  \f{w^{-}_{1}(r_{7}) \tilde{Q}'(r_{7})-w^{-}_{1} (r_{7}) {'}\tilde{Q}(r_{7})}{A}+\f{\tilde{Q}(r_{7})w^{-}_{2}{'}(r_{7})-\tilde{Q}'(r_{7})w^{-}_{2}(r_{7})}{\tilde{Q}(0)}\\
+\lam\(
w^{-}_{1}(r_{7})w^{-}_{2}{'}(r_{7})-w^{-}_{2}(r_{7})w^{-}_{1}{'}(r_{7})\)
  \end{multline}}
\end{Note}
\begin{cor}\label{C15m} The following lower bound holds:
  \begin{equation}
    \label{eq:bdbW}
   \sup_{\lambda\in[0,1]}\Big|\lambda^{-1}\tilde{W}(\lambda)\Big| \geqslant \f{1}{55}
  \end{equation}
 \end{cor}
\begin{proof}
By \eqref{A B error}, we have $1/A=\f{296}{803}+A_2$ where $|A_2|<10^{-4}$. Using the polynomials in Lemma~\ref{L14m}, we can write $\tilde{W}$ as a polynomial with explicit rational coefficients plus a remainder whose absolute value is smaller than $2\cdot10^{-4}$. Then we re-expand the resulting  polynomial approximation of $\tilde{W}$ at $\sigma =\tfrac{1}{2}$ using Note~\ref{met1}.
\end{proof}

\subsubsection{The
  Wronskian of the actual solutions}
\begin{lem}
  The Wronskian  of the actual solutions $y_1,y_2$, $W[y_1,y_2](r_{7};\lam)$,  satisfies
\begin{equation}
  \label{eq:eq54}
  |W[y_1,y_2](r_{7};\lam)-\tilde{W}(\lambda)|<\f{\lam}{1000}\ \forall \lambda\in [0,1]
\end{equation}
\end{lem}
\begin{proof}
The difference between \eqref{Wy12}  and  \eqref{w0} (we omit
its long and clearly straightforward expression) is  bounded simply using the
triangle inequality. The terms that need to be estimated in
the difference are $\tilde{u}_{1,2}-w^{-}_{1,2}$ and $\tilde{u}'_{1,2}-\partial_{r} w^{-}_{1,2}$ for which we use  Lemma~\ref{lpconm} (ii), Proposition~\ref{acts}, and Lemma~\ref{w2dm}, $Q/A_1-\tilde{Q}/A$ and $Q/Q(0)-\tilde{Q}/\tilde{Q}(0)$ which we estimate using
\eqref{A B error}, Lemma~\ref{LQ} and Lemma~\ref{210}, $|w^{-}_{1,2}|$
and $|\partial_{r} w^{-}_{1,2}|$ which are bounded
in Lemma~\ref{w1dm} and Proposition~\ref{acts}.
\end{proof}

\subsubsection{End of the proof of Proposition~\ref{p42}}
Proposition~\ref{p42} follows from Corollary~\ref{C15m} and \eqref{eq:eq54}. Therefore $W[y_1,y_2](r_{7};\lam)\neq 0$ for $\lam\in(0,1]$,
 implying Proposition~\ref{p42}.

\section{Appendix 1: Further details of proofs}\label{sec:app1}

\subsection{The general solutions of  \eqref{eqorL} and \eqref{eq:eqlm}}\label{S51}
 We have shown existence of two solutions $v_0$ and $v_\infty$ of
\eqref{eqorL} (and separately of \eqref{eq:eqlm}) which belong to~$L^2$ near
$r=0$ and $r=\infty$, respectively.  Here we  address  the question
of uniqueness.
\begin{lem}\label{lpv}
Let  $u_{0}$ and $u_\infty$ be the solutions of \eqref{eqorL} described in Sections \ref{up0} and \ref{upi}. Let $v$ be any solution of \eqref{eqorL}. Then either
the Wronskian of $v$ with respect to $u_0$ ($u_\infty$) is zero
or else $v$ is not in $L^2$ near zero (infinity, resp.).
\end{lem}

\begin{proof}
The functions  $U_{0,\infty}:=ru_{0,\infty}(r)$ satisfy the linear equation
\begin{equation}
  \label{eq:eqU3}
  -U''+(1-\la-3Q^2)U=0
\end{equation}
The Wronskian of any two solutions of \eqref{eq:eqU3} is constant
(since the coefficient of $U'$ is zero). If the Wronskian of $v$ with respect to $u_0$ ($u_\infty$) is not zero, then obviously the Wronskian $W_0$ ($W_\infty$) of $V_1=rv$ with respect to $U_0$ ($U_\infty$) is a nonzero constant. Now it follows from the expression for $w^+_{1,2}$, Corollary \ref{d1d}, \eqref{eq:eqd2}, and Lemma~\ref{l315} that $U_{0}(r)\to 0,U_{0}'(r)\to u_{1}(0)\ne 0$ as $r\to 0$, and $U_{\infty}(r)\to \const,U_{\infty}'(r)\to 0$ ($U_{\infty}'$ decays exponentially) as $r\to \infty$. Thus if $W_0\ne0$ then either $V_{1}(r)\to \const\ne 0$ or $V_{1}'(r)\sim \const/r$ as $r\to 0$, and if $W_\infty\ne0$ then $V_{1}(r)$ must increase to $\infty$ either exponentially or like $\const.r$ as $r\to \infty$. The conclusion then follows.
\end{proof}

\begin{lem}
Let  $y_1$ and $y_2$ be the solutions of \eqref{eq:eqlm} described in the beginning of Section~\ref{sec:lmm}. Let $v$ be any solution of \eqref{eq:eqlm}. Then either
the Wronskian of $v$ with respect to $u_0$ ($u_\infty$) is zero
or else $v$ is not in $L^2$ near zero (infinity, resp.).
\end{lem}

\begin{proof}
Essentially the same as the proof of Lemma~\ref{lpv}, using \eqref{eq:eqy-u}, Lemma~\ref{lpconm}, \eqref{eq:eqd2m}, and Lemma~\ref{w2dm}.
\end{proof}

\subsection{Detailed proof of Lemma~\ref{hh12}}
In the sequel, we use the bounds in  Lemma~\ref {lem:tildeQ1} (\ref{item4}), \eqref{H},  \eqref{h1 bd 1},  \eqref{h1 bd 2} and  Lemma~\ref{lem:fund sys 2}.

\begin{Note}\label{Typest}
{\rm We use various bounds to find   estimates for  the integrals in terms of exponentials
and/or polynomials of low degree, avoiding special functions: (i) We replace $\om\left(s\right)$, $s^{-2}$  and $1/(s+1)$ by their left endpoint values
 and  $\f{e^r}{r+1}$  by their values at the right end,
 $\f{1}{4r}+\f{3}{20}$ by
  $\f{9}{10}-\f12 r$ on $[1/2,1]$.
(ii) We use the  partition $\boldsymbol{\pi}=(\f12,\f34,1)$ (chosen so that
the coefficient of the $4$th power is small enough not to alter the desired estimate) for  some quartic polynomials  (iii) A  function $f$ of the form {\em quadratic plus an exponential} is estimated
by explicitly finding the zeros of the second derivative and inferring
the monotonicity properties of $f$. (iv) We also note that $H_{1,2}$
are different from $H'_{1,2}$ only by the prefactor function. The estimates
of  $H'_{1,2}$ are thus obtained easily from those for $H_{1,2}$. (v) In some
integrals where $\omega(s)$ would introduce exponentials in the final result,
we simply bound $\omega(s)$ by  the value at the left endpoint. (vi) Polynomials of the form $x^nq(x)$ with
$q$ quadratic can be clearly maximized explicitly.}

\end{Note}

\subsubsection{$r\in [0,1/2)$}
We start with $H_2$ and $H_2'$. We break the
integral at $\tfrac12,1$ and $\tfrac52$ ($\tfrac12$ is introduced for better
bounds; it is not otherwise a special point; in the second
integral we bound $\omega(s)$ by $2e^{-s}/3$). Using Note~\ref{Typest} (on $[\tfrac12,1]$ we use (v)) to bound the integrands we get
\begin{multline}\label{o119}
|H_{2}\left(\de\right)(r)| \leqslant \tfrac{65}{9} \int_{r}^{\f12}
\left(1/4 +
{3s}/{20}\right)\left[\rho_1\left(11/10-s\right)s+\eps_{0}\left(3s/50+
1/165\right)  \right]\, ds \\
+\tfrac{65}{9} \int_{\f12}^{1} \left(1/4 +
{3s}/{20}\right)\left[\rho_1\left(11/10-s\right)s+\tfrac{2}{3}\eps_{0}e^{-s}\left(3s/50+
1/165\right)  \right]\, ds \\
\quad + \tfrac{65}{9} \int_{1}^{\f{5}{2}}
e^{-s}\left[\rho_2\left(\tfrac{13}{5}-s\right)s+\eps_{0}\left(\tfrac{3s}{100}
+ \tfrac{1}{330}\right)  e^{-s}\right]\, ds+
\tfrac{13}{2} \int_{\f{5}{2}}^{\I}
e^{-s}\left(\rho_3+24\eps_0\right)\tfrac{4e^{-3s}}{25}\, ds
\end{multline}
The right-hand side of \eqref{o119} is a quartic polynomial bounded above and below  by  cubic polynomials, obtained by noting that, in the present interval, $r^4\in [0, \f12 r^3)$.
We obtain
\begin{equation}\label{estH2a}
|H_{2}\left(\de\right)\left(r\right)| < 10^{-5}\left({23}/{25}-{6r^2}/{5}\right)+{\eps_0}/{15}
\end{equation}
by explicit extremization of the cubic polynomials. Using Note~\ref{Typest} (iv) we get
\begin{equation}\label{estH2ad}
|H_{2}'\left(\de\right)\left(r\right)|\leqslant \frac{36}{13}\left[10^{-5}\(\frac{23}{25}-\frac{6r^2}{5}\)+\eps_0/15\right]
\end{equation}
To estimate $H_1$ we replace $\omega(s)=e^{-s}(s+1)^{-1}$ by the upper bound $1$  so that the integral evaluates
to a  polynomial; we get
\begin{multline}\label{estH1a}
|H_{1}\left(\de\right)\left(r\right) | < \left(\f{1}{4r} + \f{3}{20}\right) \f{5}{9} \int_{0}^{r}  13s \left(\rho_1\left(\f{11}{10}-s\right)s+\eps_{0}\left(\f{3s}{50} + \f{1}{165}\right) \right)\, ds \\
 <  \f{11}{10}\cdot 10^{-5}r^2+\f{1}{50}\eps_0
\end{multline}
After evaluating the integral, the part without $\eps_0$ is maximized using
Note~\ref{Typest} (vi). For the derivative we use Note~\ref{Typest} (iv):
\EQ{\nn \label{estH1ad}
|H_{1}'(\delta)(r)|<  10^{-5}r\left(2-\f{6r}{5}\right)+\f{\eps_0}{20}
}
\subsubsection{$r\in [\f12,1)$}
Using Note~\ref{Typest} we get
\begin{multline} \label{estH2b}
|H_{2}\left(\de\right)\left(r\right)| \leqslant  \f{7e}{10}
\Bigg( \frac{10}{9} \int_{r}^{1} \left(\f14 + \f{3}{20}s\right)
\left(\rho_1\left(\f{11}{10}-s\right)s+\eps_{0}\left(\f{3s}{50} + \f{1}{165}\right)\f{2e^{-1/2}}{3}\right) \, ds \\
\quad + \frac{10}{9} \int_{1}^{\f{5}{2}} e^{-s}\left(\rho_2\left(\f{13}{5}-s\right)s+\eps_{0}\left(\f{3s}{50} + \f{1}{165}\right) \f{e^{-s}}{2}\right)\, ds+ \int_{\f{5}{2}}^{\I} e^{-s}\left(\rho_3+24\eps_0\right)\f{4e^{-3s}}{25}\, ds\Bigg)\\
<\f{7e}{10}\left(10^{-6}\left(\f{11}{50}+2\left(1-r\right)\right)+\f{1}{100}\eps_0\right)
\end{multline}
and
\begin{equation}
  \label{estH2bd}
  |H_{2}'\left(\de\right)\left(r\right)| \leqslant 18 \left(10^{-6}\left(\f{11}{50}+2\left(1-r\right)\right)+\f{1}{100}\eps_0\right)
\end{equation}
For $H_1$ we distribute $e^s$ inside the integral and then use Note~\ref{Typest}. We get
\begin{multline} \label{estH1b}
|H_{1}\left(\de\right)\left(r\right)| \leqslant \left(\f{1}{4r}+\f{3}{20}\right) \Bigg( \f{65}{9}\int_{0}^{\f12} \left(\rho_1\left(\f{11}{10}-s\right)s+\eps_{0}\left(\f{3s}{50} + \f{1}{165}\right)e^{-s}\right)s\, ds \\
\qquad + \f{14}{9}\int_{\f12}^{r}  \f{2s}{3}\left(e^1\rho_1\left(\f{11}{10}-s\right)s+\f{2\eps_{0}}3\left(\f{3s}{50} + \f{1}{165}\right) \right)\, ds \Bigg)\\
\leqslant \left(\f{9}{10}-\f{r}2 \right)\left(8\cdot 10^{-6}+\f{\eps_0}{25}\right)
\end{multline}
For $H_{1}'$ we get
\begin{equation}\label{estH1bd}
|H_{1}'\left(\de\right)\left(r\right)| \leqslant \f{1}{2r^2}\left(8\cdot10^{-6}+\f{1}{25}\eps_0\right)\leqslant \left(\f72-3r\right) \left(8\cdot10^{-6}+\f{1}{25}\eps_0\right)
\end{equation}
\subsubsection{$r\in [r_4,r_7)$}\label{sbsbbsb} Here we use Note~\ref{Typest} except for the
re-expansions; other estimates are explained below. Let $B_1=\tfrac{11}{10}\cdot10^{-6}+10^{-2}(6-\tfrac{11}{5} r)\eps_0$. We get
\begin{multline}\label{estH2c}
 \f57(r+1) |H_{2}(\de)(r)| \leqslant  f_1(r):=\f{10e^r}{9}\int_{r}^{\f{5}{2}} e^{-s}\Bigg(\rho_2(13/5-s)s+\eps_{0}\Bigg(\f{3s}{50} + \f{1}{165}\Bigg)  \f{e^{-1}}{2}\Bigg)\, ds\\
+e^r \int_{\f{5}{2}}^{\I} e^{-s}\big(\rho_3+24\eps_0\big)\f{4e^{-3s}}{25}\, ds<B_1\ \ \end{multline}
The coefficients of  $\eps_0^n$, $n=0,1$  of $f_1-B_1$ have the form in Note~\ref{Typest} (iii) and are estimated as explained there.
 A nearly identical calculation yields
\begin{equation}\label{estH2cd}
|H_{2}'(\de)(r)|<\f{B_1}{2r}\leqslant \f{B_1}{r+1}
\end{equation}
 In evaluating $H_1$, we break the interval of integration
as follows: $[0,\tfrac12],[\tfrac12,1]$ and $[1,r]$.
To further simplify the result, in  the integral on $[\tfrac12,1]$ we first replace $s/(s+1)$ by its (alternating) Taylor series at $s=\f34$ mod $O(s-\f12)^4$:
\[
\frac{s}{s+1}\leqslant \frac37+\frac{16}{49}(s-3/4)- \frac{64}{343}(s-3/4)^2+ \frac{256}{2401}(s-3/4)^3
\]
On the same interval, we also bound $\omega(s)\leqslant \frac23 e^{-s}$.
In the
integral on $[1,r]$, we first bound $s^2/(s+1)$ by $5s/7$ (which holds for $s\leqslant 5/2$)
and then apply the bounds  in Note~\ref{Typest}; for example, we use that $\omega(s)\leqslant \frac12 e^{-s}$. After evaluation of the integrals
the result is of the form  described in Note~\ref{Typest}, (iii) and we get
\begin{equation}\label{estH1c}
r|H_{1}\left(\de\right)\left(r\right)|
\leqslant \f32\cdot 10^{-6}\left(3 -
r\right)+\f{13}{1000}\eps_0
\end{equation}
 Using Note~\ref{Typest} (iv) we get
\begin{equation}\label{estH1cd}
|H_{1}'\left(\de\right)\left(r\right)|<
\f2r\left(\f32\cdot 10^{-6}\left(3 -
    r\right)+\f{13}{1000}\eps_0\right)
    \end{equation}
\subsubsection{$r\geqslant  r_{7}$} The bounds for $H_2$ and $H_2'$ are straightforward:
We get
\begin{align}\label{estH2d}
& |H_{2}\left(\de\right)\left(r\right)| \leqslant \f{7e^r}{5r^2\left(r+1\right)}\int_{r}^{\I} e^{-s}\left(\rho_3+24\eps_0\right)e^{-3s}\, ds
\leqslant \f{21e^{-3r}\left(3+1600\eps_0\right)}{4000r^2\left(r+1\right)} \\
&|H_{2}'\left(\de\right)\left(r\right)| \leqslant \f{3e^{-3r}\left(3+1600\eps_0\right)}{1600r^3}
\end{align}
In the estimate for $H_1$ we proceed as in \S\ref{sbsbbsb} (iii) above,
except for the interval $[1,r_7]$ where
instead of the Taylor series of $s/(s+1)$ we use the simple
inequality $\f{1}{s+1}\leqslant \f{16}{25}-\f{7s}{50}$.
After integration, the estimates are elementary and we get
\begin{align}\label{estH1d}
|H_{1}\left(\de\right)\left(r\right)| \leqslant
\f{e^{-r}}{r}\left(10^{-4}\left(\f{33}{100}-36e^{-2r}\right)+\left(\f{7}{50}-\frac{48}{25}e^{-2r}\right)\eps_0\right)\\
|H_{1}'\left(\de\right)\left(r\right)|\leqslant \f{\left(r+1\right)e^{-r}}{r^2}\left(10^{-4}\left(\f{33}{100}-36e^{-2r}\right)+\left(\f{7}{50}-\frac{48}{25}e^{-2r}\right)\eps_0\right)
\end{align}
\subsection{Proof of Corollary \ref{c2.7}}\label{Pc29}
We first note that, by definition, $H_0\geqslant  0$. Throughout this proof,
we use the bounds \eqref{estH2a}--\eqref{estH1ad};  sometimes we replace them by
nearby polynomials with simpler coefficients, or majorize them by close enough
lower order polynomials, easier to maximize.
\subsubsection{$r\in J_1:=[0, \f12)$} Here we write
\EQ{\label{H est 1}
H_0\left(\delta\right)\left(r\right)<B_1(r):=10^{-5}(\tfrac32+\tfrac25 r-r^2)+\tfrac{4}{25}\eps_0
 \ \ \ \forall\; 0\leqslant r<\f12
}
Since  $(r+1)B_1<35\cdot 10^{-6}$ in $J_1$, we have
 \begin{equation}\label{ball1}
|\left(r+1\right)e^{r}H_0\left(\delta\right)\left(r\right)|\leqslant 35 e^{1/2}\cdot 10^{-6}<7\cdot 10^{-5}
\end{equation}
\subsubsection{$r\in J_2:=[\f12,1)$} Here
\EQ{\label{ball2}
H_0\left(\delta\right)\left(r\right)< B_2(r):= 10^{-5}\left(\f{27}{10}-\f{21}{10}r\right)+\left(\f{13}{100}-\f{11}{250}r\right)\eps_0 \quad\forall\; \f12\leqslant r<1
}
and thus $\left(r+1\right)e^{r}H_0\leqslant \max_{J_2}\left(r+1\right)e^rB_2(r)<\eps_0$. Note that the extrema of $\left(r+1\right)e^rB_2(r)$ can be found explicitly.
\subsubsection{$r\in J_3:=[r_4,r_7)$}
Here we get
\begin{multline}\label{ball3}
\left(r+1\right)e^r H_0\left(\delta\right)\left(r\right)<f(r) \ \text{ where $f(r)$ is given by}\\
\f{10^{-3}e^r}{r}\left(10^{-3}\(\f{63}{10}+6r-\f{21r^2}{10}\)+\left(\f{91}{5}+\f{571}{5}r-\f{176r^2}{5}+\f32(r-1)^2\right)\eps_0\right)
\leqslant \eps_0
\end{multline}
We artificially added the  term $\f32(r-1)^2$ to $f(r)$ so that $f'(r)$ is of the form $e^r r^{-2}P_3(r)$ with $P_3$ a
positive cubic polynomial on $J_3$, as it can be checked by studying
its derivative. Thus the maximum in \eqref{ball3}, which is  $<\eps_0$, is attained at
$r=r_7$.
\subsubsection{$r\in J=[r_7,\infty)$}
In this region,  $\left(r+1\right)e^r H_0\left(\delta\right)\left(r\right)$ is bounded by
\begin{multline}\label{ball4}10^{-5}\left(\frac{99}{25}+\frac{231}{50
r}+\frac{33}{50
r^2}\right)+10^{-3}e^{-2r}\left(-\frac{108}{25}-\frac{126}{25
r}+\frac{404}{25 r^2}+\frac{113}{100
r^3}\right)\\
+\eps_0\left(\frac{21}{125}+\frac{49}{250
r}+\frac{7}{250 r^2}+e^{-2r}\left(-\frac{23}{10}-\frac{67}{25
r}+\frac{431}{50
r^2}+\frac{3}{5 r^3}\right)\right)\\<
  10^{-5}\left(\frac{257}{50}+\frac{6}{r}+\frac{107}{125 r_7
r}\right)+10^{-3}e^{-2r}\left(-\frac{112}{25}-\frac{26}{5 r}+\frac{84}{5
r_7 r}+\frac{59}{50 r_7^2 r}\right)=:t_1+t_2<\eps_0
\end{multline}
where the first inequality above follows  after
replacing $\eps_0$ by $7\cdot 10^{-5}$,  simple term-by-term comparison, and then
using the fact that
$r\geqslant  r_7$ in this region. For the very last bound we first check that $t_2<0$
and then show that $(t_1-\eps_0)/t_2>-1$ by finding
the maximum ($>-3/10$) of this ratio by an elementary computation.
\subsection{Details of the proof of Lemma~\ref{64}}\label{d64}
Combining the result for $R_0$ with Lemma~\ref{lem:tildeQ1}, Lemma~\ref{bdm} and Proposition~\ref{prop:Q error}, we get
\begin{equation}
  |R_1(r)|\leqslant 10^{-5}
  \begin{cases}
3+73(1/100+3r/25)e^{-3r}+ (7/4) {e^{-r}}/({1+r}),\ \ r\in[0,r_{2})\\
2+80(1/100 +3r/25){e^{-13r/5}}/({1+r})+(7/4){ e^{-r}}/({1+r}),\ \ r\in[r_{2},r_{7})
  \end{cases}
\end{equation}
On the first interval we use the estimate $e^{-jr}\leqslant 1/(1+r)^j,~j=1,3$ to
obtain a rational function whose numerator is a linear function and
whose denominator is $(1+r)^3$. This linear function is
estimated in an elementary way, using its
derivative. On the second interval, since the derivative
of $(1/100 +3r/25)e^{-13r/5}$ is negative, the function
$80(1/100 +3r/25){e^{-13r/5}}/({1+r})+(7/4){ e^{-r}}/({1+r})$ is
decreasing. Its upper bound on the
interval $[r_2,r_4]$ is  3  and  on interval
$[r_4,r_7]$ it is 1.



\section{\bf \Large Appendix: Polynomial quasi-solutions for the soliton}
 \small
 \begin{table}[h]\label{table:1}
 \begin{tabular}{|@{}c|@{}c@{}c@{}c@{}c@{}c@{}c@{}c@{}c@{}c@{}c@{}c@{}c@{}c@{}c@{} |} \hline \multicolumn{15}{|@{}c|}{Table \ref{table:1}:  $q_{1}(r)$ and  $q_{2}(r)$  }\\[.5ex]
	\hline
$\ell$&$a_{0}^{\ell}$&$a_{1}^{\ell}$&$a_{2}^{\ell}$&$a_{3}^{\ell}$&$a_{4}^{\ell}$&$a_{5}^{\ell}$&$a_{6}^{\ell}$&$a_{7}^{\ell}$&$a_{8}^{\ell}$&$a_{9}^{\ell}$&$a_{10}^{\ell}$&$a_{11}^{\ell}$&&\\[.5ex]
	\hline\hline
$1$&$\frac{21539}{93423}$&$0$&$\frac{127023}{185578}$&$-\f{1}{8885055}$&$\f{54169}{401949}$&$-\f{3}{44981}$&$\f{202}{73305}$&$-\f{113}{80657}$&$\f{293}{59051}$&$-\f{604}{151861}$&$\f{127}{76892}$&$-\f{28}{94431}$&&\\
$$&$$&&&&&&&&&&&&&\\
$2$&$\f{18176}{78783}$&$-\f{21}{15850}$&$\f{295367}{428350}$&$-\f{1415}{123249}$&$\f{31027}{204823}$&$-\f{5162}{329873}$&$\f{3025}{287391}$&$-\f{17}{36388}$&$\f{2}{74523}$&$\f{5}{86563}$
&$-\f{1}{120831}$&$\f{1}{1183575}$&&\\\hline
	\multicolumn{15}{|@{}c|}{ $q_{3}(r)$  }  \\
\hline
$r\in$&$b_{0}^{3}$&$b_{1}^{3}$&$b_{2}^{3}$&$b_{3}^{3}$&$b_{4}^{3}$&$b_{5}^{3}$&$b_{6}^{3}$&$b_{7}^{3}$&$b_{8}^{3}$&$b_{9}^{3}$&$b_{10}^{3}$&$b_{11}^{3}$&$b_{12}^{3}$&$b_{13}^{3}$\\
	\hline\hline
$J_{1}$&$-\frac{7}{44}$&$-\frac{17}{58}$&$\frac{341}{696}$&$\frac{330}{1097}$&$-\frac{19}{62}$&$-\frac{9}{59}$&$\frac{19}{138}$&
$\frac{8}{107}$&$-{\frac{9}{122}}$&$-{\frac{15}{1688}}$&${\frac{48}{3227}}$&${\frac{23}{3181}}$&$-{\frac{65}{8481}}$
&$\frac{2}{1337}$\\
&$$&&&&&&&&&&&&&\\
$J_{2}$&${\frac{73}{344}}$&${\frac{394}{459}}$&$-{\frac{234}{131}}$&$-{
\frac{5}{53}}$&${\frac{505}{94}}$&$-{\frac{7957}{816}}$&${\frac{2219}{429}}$&${\frac{2577}{209}}$&$-{
\frac{600}{19}}$&${\frac{7783}{269}}$&$-{\frac{190}{
173}}$&$-{\frac{1836}{115}}$&${\frac{309}{74}}$&${\frac{1439}{334}}$\\
&$$&&&&&&&&&&&&&\\
$J_{3}$&${\frac{12}{89}}$&$-{\frac{9}{68}}$&${\frac{241}{3975}}$&$-{
\frac{12}{917}}$&$-{\frac{15}{3659}}$&${\frac{35}{5002
}}$&$-{\frac{2}{385}}$&${\frac{23}{6263}}$&$-{
\frac{23}{6019}}$&$-{\frac{1}{861}}$&${\frac{13}{7234}
}$&${\frac{19}{3199}}$&${\frac{22}{6021}}$&${\frac{5}{7713}}$\\
	\hline
\multicolumn{15}{|@{}c|}{$q_{4}(r)$  }  \\
\hline
$r\in$&$b_{0}^{4}$&$b_{1}^{4}$&$b_{2}^{4}$&$b_{3}^{4}$&$b_{4}^{4}$&$b_{5}^{4}$&$b_{6}^{4}$&$b_{7}^{4}$&$b_{8}^{4}$&$b_{9}^{4}$&$b_{10}^{4}$&$b_{11}^{4}$&$b_{12}^{4}$&$b_{13}^{4}$\\
	\hline\hline
$J_{1}$&$\frac{136}{271}$&$-\frac{226}{193}$&$-\frac{286}{185}$&
$\frac{743}{618}$&${\frac{527}{545}}$&
$-{\frac{83}{135}}$&$-{\frac{229}{527}}$&${\frac{2170}{9217}}$&${\frac{13}{57}}$&$-{\frac{252}{1787}}$&$-{\frac{377}{4372}}$
&${\frac{283}{2552}}$&$-{\frac{97}{2281}}$&${\frac{4}{669}}$\\
&$$&&&&&&&&&&&&&\\
$J_{2}$&$-{\frac{407}{604}}$&${\frac{2470}{1243}}$&${\frac{1106}{195}}$&$-{\frac{3554}{275}}$&${\frac{268}{31}}$&${\frac{
5785}{323}}$&$-{\frac{1239}{23}}$&${\frac{5816}{99}}$&${\frac{719}{85}}$&$-{\frac{4122}{31}}$&${\frac{
8351}{39}}$&$-{\frac{21281}{136}}$&${\frac{5030}{173}
}$&${\frac{861}{50}}$\\
&$$&&&&&&&&&&&&&\\
$J_{3}$&${\frac{1640}{447}}$&${\frac{1485}{389}}$&${\frac{33}{20}}$&${
\frac{127}{168}}$&${\frac{65}{867}}$&${\frac{68}{1053
}}$&$-{\frac{7}{1313}}$&
${\frac{2}{859}}$&$-{
\frac{3}{1670}}$&$-{\frac{1}{149}}$&${\frac{23}{4038}}
$&${\frac{11}{817}}$&${\frac{92}{7641}}$&${
\frac{20}{7719}}$\\
	\hline
	\end{tabular}
\end{table}
\eject
\section{\bf \Large Appendix: Polynomial quasi-solutions for $L_+$ and $L_-$}
\hskip -2.5cm \begin{tabular}{|@{}c@{}c@{}c@{}c@{}c@{}c|@{}c @{}c @{}c @{}c @{}c @{}c |@{}c@{}c@{}c@{}c@{}c@{}c@{}c |}\hline
\multicolumn{6}{|c}{$c_{kl;j}$ for $r\in[0,r_2),z=\f{25r}{17}$}& \multicolumn{6}{|c}{$c_{kl;j}$ for $r\in[r_2,r_5), z=r-1$}&\multicolumn{7}{|c|}{$c_{kl;j}$ for $r\in[r_5,r_7], z=r-2$}\\
\hline$l:k$&$0$&$1$&$2$&$3$&$4$&$l:k$&$0$&$1$&$2$&$3$&$4$&                                                                                                                                                                  $l:k$  & $0  $&$1$    &$2$    &$3$    &$4$    &$5$    \\
\hline0&$1$&$0$&$0$&$0$&$0$&0&$\tfrac{-648}{2321}$&$\tfrac{13}{1467}$&$\tfrac{2}{2157}$&$\tfrac{-1}{21701}$&$\tfrac{1}{1118563}$&                                                                                          0&  $\tfrac{-746}{2291}$    &$\tfrac{416}{2757}$    &$\tfrac{-68}{4263}$    &$\tfrac{1}{3261}$    &$\tfrac{1}{24208}$    &$\tfrac{-1}{300133}$  \\
[.5ex]1&$0$&$0$&$0$&$0$&$0$&1&$\tfrac{-134}{1519}$&$\tfrac{157}{2061}$&$\tfrac{1}{5558}$&$\tfrac{-1}{5273}$&$\tfrac{1}{167616}$&                                                                                           1&  $\tfrac{-265}{2463}$    &$\tfrac{413}{1808}$    &$\tfrac{-117}{2305}$    &$\tfrac{12}{4069}$    &$\tfrac{1}{22487}$    &$\tfrac{-1}{82587}$  \\
[.5ex]2&$\tfrac{-1286}{301}$&$\tfrac{-83}{1077}$&$0$&$0$&$0$&2&$\tfrac{1187}{3641}$&$\tfrac{128}{2291}$&$\tfrac{-23}{4259}$&$\tfrac{-1}{4261}$&$\tfrac{1}{61302}$&                                                         2&  $\tfrac{-785}{8474}$    &$\tfrac{1495}{12838}$    &$\tfrac{-155}{2708}$    &$\tfrac{43}{6477}$    &$\tfrac{-1}{6371}$    &$\tfrac{-1}{62120}$  \\
[.5ex]3&$\tfrac{-1}{417}$&$\tfrac{-1}{19057}$&$0$&$0$&$\tfrac{1}{71024}$&3&$\tfrac{-5583}{8174}$&$\tfrac{-7}{1692}$&$\tfrac{-11}{1456}$&$\tfrac{1}{24417}$&$\tfrac{1}{44438}$&                                             3&  $\tfrac{-31}{2947}$    &$\tfrac{91}{2278}$    &$\tfrac{-151}{4774}$    &$\tfrac{17}{2446}$    &$\tfrac{-1}{2343}$    &$\tfrac{-1}{198555}$  \\
[.5ex]4&$\tfrac{13222}{1453}$&$\tfrac{334}{1683}$&$\tfrac{1}{559}$&$0$&$\tfrac{-1}{3668}$&4&$\tfrac{3158}{3977}$&$\tfrac{24}{1111}$&$\tfrac{-19}{5789}$&$\tfrac{1}{2997}$&$\tfrac{1}{71808}$&                              4&  $\tfrac{1}{3172}$    &$\tfrac{74}{6611}$    &$\tfrac{-29}{2777}$    &$\tfrac{9}{2204}$    &$\tfrac{-1}{2115}$    &$\tfrac{1}{84243}$  \\
[.5ex]5&$\tfrac{-380}{1021}$&$\tfrac{-20}{2461}$&$\tfrac{-1}{13736}$&$0$&$\tfrac{1}{346}$&5&$\tfrac{-2380}{3639}$&$\tfrac{-59}{4489}$&$\tfrac{-3}{4595}$&$\tfrac{1}{3475}$&$\tfrac{-1}{512265}$&                           5&  $\tfrac{-43}{6113}$    &$\tfrac{15}{7912}$    &$\tfrac{-11}{4122}$    &$\tfrac{2}{1343}$    &$\tfrac{-1}{3308}$    &$\tfrac{1}{53636}$  \\
[.5ex]6&$\tfrac{-26425}{1858}$&$\tfrac{-175}{566}$&$\tfrac{-1}{347}$&$\tfrac{-1}{56456}$&$\tfrac{-37}{1949}$&6&$\tfrac{1008}{3337}$&$\tfrac{9}{1369}$&$\tfrac{-1}{11075}$&$\tfrac{1}{10309}$&$\tfrac{-1}{112008}$&         6&  $\tfrac{19}{3527}$    &$\tfrac{1}{2072}$    &$\tfrac{-1}{1743}$    &$\tfrac{1}{2619}$    &$\tfrac{-1}{8075}$    &$\tfrac{1}{72145}$  \\
[.5ex]7&$\tfrac{-689}{122}$&$\tfrac{-101}{821}$&$\tfrac{-1}{913}$&$\tfrac{-1}{184536}$&$\tfrac{61}{739}$&7&$\tfrac{320}{2401}$&$\tfrac{7}{2421}$&$\tfrac{-1}{15816}$&$\tfrac{1}{80612}$&$\tfrac{-1}{176777}$&              7&  $\tfrac{-21}{5575}$    &$\tfrac{-1}{51636}$    &$\tfrac{-1}{11581}$    &$\tfrac{1}{12393}$    &$\tfrac{-1}{28593}$    &$\tfrac{1}{154234}$  \\
[.5ex]8&$\tfrac{20843}{669}$&$\tfrac{405}{598}$&$\tfrac{4}{659}$&$\tfrac{1}{31276}$&$\tfrac{-185}{749}$&8&$\tfrac{-1908}{3757}$&$\tfrac{-149}{13558}$&$\tfrac{-1}{14577}$&$\tfrac{1}{238641}$&$\tfrac{-1}{672528}$&        8&  $\tfrac{11}{4880}$    &$\tfrac{1}{19213}$    &$\tfrac{-1}{73329}$    &$\tfrac{1}{71256}$    &$\tfrac{-1}{132884}$    &$\tfrac{1}{480205}$  \\
[.5ex]9&$\tfrac{18587}{642}$&$\tfrac{632}{1001}$&$\tfrac{3}{526}$&$\tfrac{1}{33735}$&$\tfrac{93}{179}$&9&$\tfrac{1477}{2105}$&$\tfrac{160}{10507}$&$\tfrac{1}{8021}$&$\tfrac{1}{580882}$&$\tfrac{-1}{5058064}$&            9&  $\tfrac{-4}{3271}$    &$\tfrac{-1}{39754}$    &$\tfrac{-1}{458329}$    &$\tfrac{1}{525744}$    &$\tfrac{-1}{736052}$    &$\tfrac{1}{2039005}$  \\
[.5ex]10&$\tfrac{-165287}{969}$&$\tfrac{-1567}{422}$&$\tfrac{-25}{751}$&$\tfrac{-1}{5892}$&$\tfrac{-305}{394}$&10&$\tfrac{-1043}{1602}$&$\tfrac{-51}{3605}$&$\tfrac{-1}{7915}$&$\tfrac{-1}{1015424}$&$0$&                  10&  $\tfrac{1}{1662}$    &$\tfrac{1}{75918}$    &$\tfrac{1}{9779923}$    &$\tfrac{1}{3561333}$    &$\tfrac{-1}{4924768}$  &   $0$\\
[.5ex]11&$\tfrac{351973}{1352}$&$\tfrac{8834}{1559}$&$\tfrac{53}{1045}$&$\tfrac{1}{3893}$&$\tfrac{1375}{1693}$&11&$\tfrac{832}{2037}$&$\tfrac{18}{2027}$&$\tfrac{1}{12149}$&$\tfrac{1}{1662352}$&$0$&                      11&  $\tfrac{-1}{3754}$    &$\tfrac{-1}{170856}$    &$\tfrac{-1}{7716870}$  &   $0$&   $0$&   $0$\\
[.5ex]12&$\tfrac{-519457}{2400}$&$\tfrac{-16569}{3517}$&$\tfrac{-41}{973}$&$\tfrac{-1}{4702}$&$\tfrac{-295}{503}$&12&$\tfrac{-586}{3849}$&$\tfrac{-27}{8155}$&$\tfrac{-1}{31427}$&$\tfrac{-1}{4652814}$&$0$&               12&  $\tfrac{1}{9714}$    &$\tfrac{1}{439897}$  &   $0$&   $0$&   $0$&   $0$\\
[.5ex]13&$\tfrac{194127}{1816}$&$\tfrac{1822}{783}$&$\tfrac{16}{769}$&$\tfrac{1}{9539}$&$\tfrac{299}{1078}$&13&$\tfrac{91}{3737}$&$\tfrac{1}{1888}$&$\tfrac{1}{186011}$&$0$&$0$& &&&&&&\\
[.5ex]14&$\tfrac{-33551}{1130}$&$\tfrac{-223}{345}$&$\tfrac{-19}{3288}$&$\tfrac{-1}{34377}$&$\tfrac{-15}{194}$&&&&&&& &&&&&&\\
[.5ex]15&$\tfrac{2537}{705}$&$\tfrac{38}{485}$&$\tfrac{1}{1427}$&$\tfrac{1}{283750}$&$\tfrac{7}{727}$&&&&&&& &&&&&&\\
[.5ex]\hline\end{tabular}

\zn $ $\hskip -0.5cm \begin{tabular}{|@{}c@{}c@{}c@{}c@{}|c@{}c@{}c@{}c@{}c@{}c|@{}c@{}c@{}c@{}c@{}c@{}c@{}c|}
\hline \multicolumn{4}{|c}{$d^+_{kl;j}$ for $r\in[0,r_1), z=3r$} &\multicolumn{7}{|c}{$d^+_{kl;j}$ for $r\in[r_1,r_4), z=r-\f12$}& \multicolumn{6}{|c|}{$d^+_{kl;j}$ for $r\in[r_4,r_7], z=r-2$}\\
\hline$l:k$&$0$&$1$&$2$&$l:k$&$0$&$1$&$2$&$3$&$4$&$l:k$&$0$&$1$&$2$&$3$&$4$&$5$\\
\hline0&$0$&$0$&$0$&0&$\tfrac{1}{1489}$&$\tfrac{-16}{2463}$&$\tfrac{1}{8158}$&$0$&$0$&0&$\tfrac{-1492}{2291}$&$\tfrac{577}{1912}$&$\tfrac{-26}{815}$&$\tfrac{1}{1630}$&$\tfrac{1}{12104}$&$\tfrac{-1}{150067}$\\
[.5ex]1&$1$&$0$&$0$&1&$\tfrac{-661}{885}$&$\tfrac{-11}{956}$&$\tfrac{1}{1073}$&$\tfrac{-1}{92371}$&$0$&1&$\tfrac{-550}{1017}$&$\tfrac{894}{1471}$&$\tfrac{-39}{332}$&$\tfrac{6}{967}$&$\tfrac{1}{7677}$&$\tfrac{-1}{36299}$\\
[.5ex]2&$0$&$0$&$0$&2&$\tfrac{-11}{1948}$&$\tfrac{122}{2245}$&$\tfrac{1}{451}$&$\tfrac{-1}{18565}$&$0$&2&$\tfrac{-193}{659}$&$\tfrac{871}{1888}$&$\tfrac{-77}{466}$&$\tfrac{8}{493}$&$\tfrac{-1}{3711}$&$\tfrac{-1}{22571}$\\
[.5ex]3&$\tfrac{-4509}{488}$&$\tfrac{-1}{6}$&$0$&3&$\tfrac{898}{427}$&$\tfrac{47}{547}$&$\tfrac{1}{1561}$&$\tfrac{-1}{7433}$&$\tfrac{1}{594591}$&3&$\tfrac{-49}{431}$&$\tfrac{43}{219}$&$\tfrac{-87}{722}$&$\tfrac{16}{779}$&$\tfrac{-1}{989}$&$\tfrac{-1}{38211}$\\
[.5ex]4&$\tfrac{-27}{9226}$&$0$&$0$&4&$\tfrac{-4572}{1123}$&$\tfrac{-17}{205}$&$\tfrac{-3}{832}$&$\tfrac{-1}{6225}$&$\tfrac{1}{238553}$&4&$\tfrac{-6}{607}$&$\tfrac{47}{754}$&$\tfrac{-24}{457}$&$\tfrac{13}{860}$&$\tfrac{-1}{729}$&$\tfrac{1}{53463}$\\
[.5ex]5&$\tfrac{56619}{1333}$&$\tfrac{729}{787}$&$\tfrac{243}{29086}$&5&$\tfrac{1833}{898}$&$\tfrac{41}{879}$&$\tfrac{-4}{1877}$&$\tfrac{-1}{24306}$&$\tfrac{1}{156927}$&5&$\tfrac{-13}{956}$&$\tfrac{25}{1668}$&$\tfrac{-8}{507}$&$\tfrac{5}{708}$&$\tfrac{-1}{928}$&$\tfrac{1}{20342}$\\
[.5ex]6&$\tfrac{-2187}{1040}$&$\tfrac{-243}{5291}$&$0$&6&$\tfrac{4957}{830}$&$\tfrac{67}{513}$&$\tfrac{1}{1820}$&$\tfrac{1}{12187}$&$\tfrac{1}{186773}$&6&$\tfrac{3}{785}$&$\tfrac{3}{1048}$&$\tfrac{-1}{262}$&$\tfrac{1}{444}$&$\tfrac{-1}{1818}$&$\tfrac{1}{21568}$\\
[.5ex]7&$\tfrac{-118098}{851}$&$\tfrac{-2187}{725}$&$\tfrac{-2187}{77746}$&7&$\tfrac{-16502}{1069}$&$\tfrac{-433}{1292}$&$\tfrac{-3}{1121}$&$\tfrac{1}{16955}$&$\tfrac{1}{697800}$&7&$\tfrac{-4}{1315}$&$\tfrac{1}{2355}$&$\tfrac{-1}{1339}$&$\tfrac{1}{1841}$&$\tfrac{-1}{5160}$&$\tfrac{1}{37272}$\\
[.5ex]8&$\tfrac{-452709}{2441}$&$\tfrac{-6561}{1621}$&$\tfrac{-6561}{180628}$&8&$\tfrac{34798}{2423}$&$\tfrac{475}{1523}$&$\tfrac{1}{385}$&$\tfrac{1}{42868}$&$\tfrac{-1}{800114}$&8&$\tfrac{-1}{1953}$&$\tfrac{1}{17286}$&$\tfrac{-1}{8792}$&$\tfrac{1}{9195}$&$\tfrac{-1}{19986}$&$\tfrac{1}{93864}$\\
[.5ex]9&$\tfrac{642978}{439}$&$\tfrac{19683}{617}$&$\tfrac{6561}{22889}$&9&$\tfrac{4697}{807}$&$\tfrac{10}{79}$&$\tfrac{1}{952}$&$\tfrac{-1}{724921}$&$\tfrac{-1}{925029}$&9&$\tfrac{1}{945}$&$\tfrac{1}{34655}$&$\tfrac{-1}{56303}$&$\tfrac{1}{56043}$&$\tfrac{-1}{97604}$&$\tfrac{1}{326251}$\\
[.5ex]10&$\tfrac{-3956283}{1526}$&$\tfrac{-6561}{116}$&$\tfrac{-59049}{116578}$&10&$\tfrac{-29915}{954}$&$\tfrac{-381}{559}$&$\tfrac{-4}{675}$&$\tfrac{-1}{40723}$&$0$&10&$\tfrac{2}{563}$&$\tfrac{1}{12822}$&$\tfrac{-1}{639743}$&$\tfrac{1}{403473}$&$\tfrac{-1}{572379}$&$0$\\
[.5ex]11&$\tfrac{2302911}{1145}$&$\tfrac{177147}{4031}$&$\tfrac{19683}{49996}$&11&$\tfrac{26559}{790}$&$\tfrac{920}{1259}$&$\tfrac{5}{779}$&$\tfrac{1}{31728}$&$0$&11&$\tfrac{1}{436}$&$\tfrac{1}{20441}$&$0$&$0$&$0$&$0$\\
[.5ex]12&$\tfrac{-177147}{307}$&$\tfrac{-531441}{41978}$&$0$&12&$\tfrac{-39879}{3046}$&$\tfrac{-249}{875}$&$\tfrac{-1}{398}$&$\tfrac{-1}{77233}$&$0$&12&$\tfrac{1}{6262}$&$\tfrac{1}{318712}$&$0$&$0$&$0$&$0$\\
 [.5ex] \hline\end{tabular}
\zn  \vskip -2cm  \hskip -2cm  \raisebox{2.5cm}{\begin{tabular}{|@{}c @{}c @{}c @{}c @{}c @{}c| @{}c @{}c @{}c @{}c @{}c @{}c| @{}c @{}c @{}c @{}c @{}c @{}c  @{}c |}
\hline \multicolumn{6}{|c|}{$e^+_{kl;j}$ for $r\in[0,r_1), z=3r$} &\multicolumn{7}{c|}{$e^+_{kl;j}$ for $r\in[r_1,r_4), z=r-\f12$}& \multicolumn{6}{c|}{$e^+_{kl;j}$ for $r\in[r_4,r_7), z=r-2$}\\
 \hline $l:k$  & $0  $&$1$    &$2$    &$3$    &$4$    &                                                                                              $l:k$  & $0  $&$1$    &$2$    &$3$    &$4$    &                                                                        $l:k$  & $0  $&$1$    &$2$    &$3$    &$4$    &$5$    \\
 \hline 0&   $1$&   $0$&   $0$&   $0$&   $0$&                                                                                                        0&  $\tfrac{-507}{379}$    &$\tfrac{6}{643}$    &$\tfrac{1}{1403}$    &$\tfrac{-1}{102995}$  &   $0$&                          0&  $\tfrac{9296}{3379}$    &$\tfrac{424}{839}$    &$\tfrac{-279}{1040}$    &$\tfrac{35}{1234}$    &$\tfrac{-1}{759}$    &$\tfrac{1}{34562}$  \\
 [.5ex] 1&  $\tfrac{-3}{858242}$  &   $0$&   $0$&   $0$&   $0$&                                                                                      1&  $\tfrac{-821}{647}$    &$\tfrac{193}{724}$    &$\tfrac{2}{1129}$    &$\tfrac{-1}{11194}$  &   $0$&                         1&  $\tfrac{13848}{3625}$    &$\tfrac{-1971}{5221}$    &$\tfrac{-972}{1735}$    &$\tfrac{295}{2911}$    &$\tfrac{-35}{5158}$    &$\tfrac{1}{4588}$  \\
 [.5ex] 2&  $\tfrac{-9369}{338}$    &$\tfrac{-19305}{38611}$  &   $0$&   $0$&   $0$&                                                                 2&  $\tfrac{11429}{1015}$    &$\tfrac{317}{537}$    &$\tfrac{-13}{1219}$    &$\tfrac{-1}{3641}$    &$\tfrac{1}{226467}$  &     2&  $\tfrac{3859}{3119}$    &$\tfrac{-26003}{22645}$    &$\tfrac{-1328}{3557}$    &$\tfrac{177}{1205}$    &$\tfrac{-35}{2369}$    &$\tfrac{2}{2977}$  \\
 [.5ex] 3&  $\tfrac{-27}{1141}$    &$\tfrac{-27}{40190}$  &   $0$&   $0$&   $0$&                                                                     3&  $\tfrac{-24009}{2177}$    &$\tfrac{-589}{1355}$    &$\tfrac{-42}{1009}$    &$\tfrac{-1}{6648}$    &$\tfrac{1}{74939}$  &   3&  $\tfrac{1298}{1827}$    &$\tfrac{-2392}{3581}$    &$\tfrac{-179}{5186}$    &$\tfrac{292}{2655}$    &$\tfrac{-47}{2615}$    &$\tfrac{3}{2575}$  \\
 [.5ex] 4&  $\tfrac{38691}{247}$    &$\tfrac{2862}{617}$    &$\tfrac{81}{1933}$  &   $0$&   $0$&                                                     4&  $\tfrac{-8644}{767}$    &$\tfrac{-1839}{4681}$    &$\tfrac{-5}{163}$    &$\tfrac{1}{1149}$    &$\tfrac{1}{49139}$  &       4&  $\tfrac{394}{4905}$    &$\tfrac{-733}{3355}$    &$\tfrac{117}{1822}$    &$\tfrac{59}{1355}$    &$\tfrac{-23}{1708}$    &$\tfrac{3}{2333}$  \\
 [.5ex] 5&  $\tfrac{-14337}{1301}$    &$\tfrac{-243}{773}$    &$\tfrac{-243}{67030}$  &   $0$  &$\tfrac{243}{387208}$  &                             5&  $\tfrac{25973}{519}$    &$\tfrac{2033}{1357}$    &$\tfrac{16}{1129}$    &$\tfrac{1}{572}$    &$\tfrac{1}{143450}$  &       5&  $\tfrac{127}{3946}$    &$\tfrac{-185}{3047}$    &$\tfrac{67}{2065}$    &$\tfrac{11}{1522}$    &$\tfrac{-55}{8589}$    &$\tfrac{5}{5226}$  \\
 [.5ex] 6&  $\tfrac{-103275}{196}$    &$\tfrac{-26973}{1693}$    &$\tfrac{-81}{422}$    &$\tfrac{-729}{636424}$    &$\tfrac{-81}{13156}$  &          6&  $\tfrac{-184251}{2800}$    &$\tfrac{-2774}{1411}$    &$\tfrac{-13}{750}$    &$\tfrac{1}{1028}$    &$\tfrac{-1}{39776}$  &  6&  $\tfrac{114}{8405}$    &$\tfrac{-539}{44197}$    &$\tfrac{22}{2419}$    &$\tfrac{-2}{2797}$    &$\tfrac{-7}{3683}$    &$\tfrac{1}{2033}$  \\
 [.5ex] 7&  $\tfrac{-562059}{671}$    &$\tfrac{-18225}{758}$    &$\tfrac{-243}{877}$  &   $0$  &$\tfrac{2187}{56503}$  &                             7&  $\tfrac{9627}{1240}$    &$\tfrac{482}{2131}$    &$\tfrac{4}{1571}$    &$\tfrac{1}{6744}$    &$\tfrac{-1}{25900}$  &        7&  $\tfrac{-49}{4514}$    &$\tfrac{-5}{2344}$    &$\tfrac{21}{9209}$    &$\tfrac{-1}{1687}$    &$\tfrac{-1}{3088}$    &$\tfrac{1}{5705}$  \\
 [.5ex] 8&  $\tfrac{106191}{17}$    &$\tfrac{321489}{1760}$    &$\tfrac{2187}{1037}$    &$\tfrac{6561}{486394}$    &$\tfrac{-6561}{42389}$  &        8&  $\tfrac{33099}{265}$    &$\tfrac{5679}{1523}$    &$\tfrac{112}{2739}$    &$\tfrac{1}{9497}$    &$\tfrac{-1}{47430}$  &     8&  $\tfrac{47}{3968}$    &$\tfrac{-1}{67547}$    &$\tfrac{1}{2567}$    &$\tfrac{-1}{5711}$    &$\tfrac{-1}{47026}$    &$\tfrac{1}{23415}$  \\
 [.5ex] 9&  $\tfrac{-7735419}{692}$    &$\tfrac{-373977}{1156}$    &$\tfrac{-6561}{1762}$    &$\tfrac{-19683}{851923}$    &$\tfrac{2187}{5782}$  &   9&  $\tfrac{-102332}{451}$    &$\tfrac{-6050}{897}$    &$\tfrac{-29}{383}$    &$\tfrac{-1}{2360}$    &$\tfrac{-1}{265525}$  &  9&  $\tfrac{-8}{4399}$    &$\tfrac{-1}{11726}$    &$\tfrac{1}{19947}$    &$\tfrac{-1}{24581}$    &$\tfrac{1}{270039}$    &$\tfrac{1}{144354}$  \\
 [.5ex] 10&  $\tfrac{2007666}{227}$    &$\tfrac{236196}{949}$    &$\tfrac{19683}{6878}$  &   $0$  &$\tfrac{-59049}{115346}$  &                       10&  $\tfrac{92881}{600}$    &$\tfrac{5614}{1221}$    &$\tfrac{46}{879}$    &$\tfrac{1}{2940}$    &$\tfrac{1}{315365}$  &      10&  $\tfrac{-4}{3061}$    &$\tfrac{-1}{24496}$    &$\tfrac{1}{95846}$    &$\tfrac{-1}{167437}$    &$\tfrac{1}{576914}$    &$\tfrac{1}{1404425}$  \\
 [.5ex] 11&  $\tfrac{-1771470}{677}$    &$\tfrac{-177147}{2579}$    &$\tfrac{-177147}{224000}$  &   $0$  &$\tfrac{59049}{201628}$  &                 11&  $\tfrac{33374}{623}$    &$\tfrac{1042}{695}$    &$\tfrac{27}{1621}$    &$\tfrac{1}{14089}$  &   $0$&                      11&  $\tfrac{-11}{1218}$    &$\tfrac{-1}{3679}$    &$\tfrac{-1}{294819}$    &$\tfrac{-1}{1077142}$    &$\tfrac{1}{2623485}$  &   $0$\\
        12 &&&&&&                                                                                                                                    12&  $\tfrac{-137367}{890}$    &$\tfrac{-3853}{874}$    &$\tfrac{-32}{641}$    &$\tfrac{-1}{3474}$    &$\tfrac{-1}{920175}$  & 12&  $\tfrac{-9}{3994}$    &$\tfrac{-1}{15039}$    &$\tfrac{-1}{1354218}$    &$\tfrac{-1}{6141074}$  &   $0$&   $0$\\
        13 &&&&&&                                                                                                                                    13&  $\tfrac{59338}{817}$    &$\tfrac{233}{113}$    &$\tfrac{17}{726}$    &$\tfrac{1}{7200}$  &   $0$&                         13 &&&&&&\\
 \hline\end{tabular}}
$ $ \hskip -1.6cm \begin{tabular}{|@{}c @{}c @{}c @{}c @{}c |@{}c @{}c @{}c @{}c @{}c @{}c|@{}c @{}c @{}c @{}c @{}c @{}c @{}c @{}c | }
\hline \multicolumn{5}{|c|}{$d^-_{kl;j}$ for $r\in[0,r_1), z=3r$} & \multicolumn{6}{c|}{$d^-_{kl;j}$ for $r\in[r_1,r_4), z=r-\f12$} &  \multicolumn{8}{c|}{$d^-_{kl;j}$ for $r\in[r_4,r_7), z=r-2$}\\
\hline$l:k$&$0$&$1$&$2$&$3$&$l:k$&$0$&$1$&$2$&$3$&$4$&                                                                                                                                                $l:k$  & $0  $&$1$    &$2$    &$3$    &$4$    &$5$    &$6$    \\
\hline0&$1$&$0$&$0$&$0$&0&$\tfrac{484}{861}$&$\tfrac{-27}{892}$&$\tfrac{1}{2399}$&$\tfrac{-1}{382886}$&$0$&                                                                                           0&  $\tfrac{226}{5353}$    &$\tfrac{-263}{1306}$    &$\tfrac{104}{1763}$    &$\tfrac{-17}{2511}$    &$\tfrac{1}{2373}$    &$\tfrac{-1}{60373}$    &$\tfrac{1}{2224934}$  \\
[.5ex]1&$0$&$0$&$0$&$0$&1&$\tfrac{-3235}{3137}$&$\tfrac{-108}{1165}$&$\tfrac{1}{326}$&$\tfrac{-1}{33259}$&$0$&                                                                                        1&  $\tfrac{-183}{2903}$    &$\tfrac{-334}{2205}$    &$\tfrac{307}{2854}$    &$\tfrac{-63}{3259}$    &$\tfrac{4}{2449}$    &$\tfrac{-1}{12346}$    &$\tfrac{1}{377035}$  \\
[.5ex]2&$\tfrac{-6570}{2213}$&$\tfrac{-1}{6}$&$0$&$0$&2&$\tfrac{150}{223}$&$\tfrac{-34}{1633}$&$\tfrac{2}{251}$&$\tfrac{-1}{7051}$&$0$&                                                               2&  $\tfrac{266}{5123}$    &$\tfrac{-68}{1593}$    &$\tfrac{142}{1883}$    &$\tfrac{-66}{2857}$    &$\tfrac{4}{1443}$    &$\tfrac{-1}{5610}$    &$\tfrac{1}{139417}$  \\
[.5ex]3&$\tfrac{9}{315394}$&$0$&$0$&$0$&3&$\tfrac{565}{909}$&$\tfrac{39}{920}$&$\tfrac{11}{1392}$&$\tfrac{-1}{2881}$&$\tfrac{1}{263339}$&                                                             3&  $\tfrac{-25}{781}$    &$\tfrac{-31}{1981}$    &$\tfrac{230}{8399}$    &$\tfrac{-53}{3518}$    &$\tfrac{5}{1849}$    &$\tfrac{-1}{4293}$    &$\tfrac{1}{85007}$  \\
[.5ex]4&$\tfrac{9882}{1201}$&$\tfrac{648}{2183}$&$\tfrac{1}{120}$&$0$&4&$\tfrac{-1535}{729}$&$\tfrac{-8}{105}$&$\tfrac{1}{680}$&$\tfrac{-1}{2198}$&$\tfrac{1}{111414}$&                               4&  $\tfrac{97}{5791}$    &$\tfrac{-4}{1389}$    &$\tfrac{23}{3308}$    &$\tfrac{-22}{3683}$    &$\tfrac{6}{3577}$    &$\tfrac{-1}{4980}$    &$\tfrac{1}{76633}$  \\
[.5ex]5&$\tfrac{81}{2762}$&$\tfrac{243}{252596}$&$0$&$0$&5&$\tfrac{2313}{958}$&$\tfrac{45}{628}$&$\tfrac{1}{9833}$&$\tfrac{-1}{3472}$&$\tfrac{1}{75265}$&                                             5&  $\tfrac{-21}{2677}$    &$\tfrac{-1}{2455}$    &$\tfrac{9}{5621}$    &$\tfrac{-3}{1877}$    &$\tfrac{1}{1435}$    &$\tfrac{-1}{8318}$    &$\tfrac{1}{96698}$  \\
[.5ex]6&$\tfrac{-34992}{1513}$&$\tfrac{-81}{104}$&$\tfrac{-729}{67664}$&$\tfrac{-729}{3653291}$&6&$\tfrac{-541}{1122}$&$\tfrac{-5}{799}$&$\tfrac{1}{2333}$&$\tfrac{-1}{16948}$&$\tfrac{1}{84280}$&    6&  $\tfrac{82}{24519}$    &$\tfrac{-1}{9696}$    &$\tfrac{1}{3919}$    &$\tfrac{-1}{2955}$    &$\tfrac{1}{4932}$    &$\tfrac{-1}{19435}$    &$\tfrac{1}{165730}$  \\
[.5ex]7&$\tfrac{10935}{2777}$&$\tfrac{2187}{16928}$&$\tfrac{2187}{1369664}$&$0$&7&$\tfrac{-6207}{1946}$&$\tfrac{-22}{203}$&$\tfrac{-1}{767}$&$0$&$\tfrac{1}{188890}$&                                 7&  $\tfrac{-2}{1403}$    &$\tfrac{-1}{66500}$    &$\tfrac{1}{28487}$    &$\tfrac{-1}{16359}$    &$\tfrac{1}{22558}$    &$\tfrac{-1}{62158}$    &$\tfrac{1}{379198}$  \\
[.5ex]8&$\tfrac{78732}{2089}$&$\tfrac{6561}{5257}$&$\tfrac{6561}{408190}$&$0$&8&$\tfrac{2003}{349}$&$\tfrac{55}{298}$&$\tfrac{1}{486}$&$\tfrac{1}{308534}$&$\tfrac{1}{633769}$&                       8&  $\tfrac{1}{2216}$    &$\tfrac{1}{262906}$    &$\tfrac{1}{148304}$    &$\tfrac{-1}{116407}$    &$\tfrac{1}{125066}$    &$\tfrac{-1}{263845}$    &$\tfrac{1}{1144641}$  \\
[.5ex]9&$\tfrac{354294}{3235}$&$\tfrac{19683}{5482}$&$\tfrac{19683}{444148}$&$0$&9&$\tfrac{-2787}{803}$&$\tfrac{-394}{3677}$&$\tfrac{-1}{893}$&$\tfrac{-1}{182317}$&$\tfrac{1}{547500}$&              9&  $\tfrac{1}{3736}$    &$\tfrac{1}{107868}$    &$\tfrac{1}{8213080}$    &$\tfrac{-1}{873713}$    &$\tfrac{1}{815756}$    &$\tfrac{-1}{1409020}$    &$\tfrac{1}{4491233}$  \\
[.5ex]10&$\tfrac{-275562}{541}$&$\tfrac{-59049}{3541}$&$\tfrac{-59049}{287186}$&$0$&10&$\tfrac{-2403}{869}$&$\tfrac{-53}{561}$&$\tfrac{-1}{812}$&$\tfrac{-1}{138711}$&$\tfrac{-1}{126895}$&           10&  $\tfrac{1}{8315}$    &$\tfrac{1}{292450}$    &$\tfrac{1}{4486050}$    &$\tfrac{-1}{7201156}$    &$\tfrac{1}{6368801}$    &$\tfrac{-1}{8904667}$  &   $0$\\
[.5ex]11&$\tfrac{1062882}{1513}$&$\tfrac{177147}{7715}$&$\tfrac{19683}{69778}$&$0$&11&$\tfrac{47576}{8223}$&$\tfrac{157}{828}$&$\tfrac{3}{1303}$&$\tfrac{1}{73137}$&$\tfrac{1}{95558}$&               11&  $\tfrac{-1}{3146}$    &$\tfrac{-1}{105133}$    &$\tfrac{-1}{7531672}$  &   $0$&   $0$&   $0$&   $0$\\
[.5ex]12&$\tfrac{-531441}{1490}$&$\tfrac{-531441}{45622}$&$\tfrac{-531441}{3722458}$&$0$&12&$\tfrac{-1748}{635}$&$\tfrac{-77}{860}$&$\tfrac{-1}{930}$&$\tfrac{-1}{154886}$&$\tfrac{-1}{206912}$&      12&  $\tfrac{-1}{3177}$    &$\tfrac{-1}{108754}$    &$\tfrac{-1}{9654221}$  &   $0$&   $0$&   $0$&   $0$\\
 [.5ex] \hline\end{tabular}
$ $ \hskip -0.2cm \begin{tabular}{|@{}c @{}c @{}c @{}c |@{}c@{}c@{}c@{}c@{}c|@{}c @{}c @{}c @{}c @{}c @{}c @{}c |}\hline\multicolumn{4}{|@{}c}{$e^-_{kl;j}$ for $r\in[0,r_1), z=3r$}& \multicolumn{5}{|c}{$e^-_{kl;j}$for$r\in[r_1,r_4),z=r-\f12$} &\multicolumn{7}{|c|}{$e^-_{kl;j}$ for $r\in[r_4,r_7], z=r-2$}\\
\hline$l:k$&$0$&$1$&$2$&$l:k$&$0$&$1$&$2$&$3$&                                                                                                               $l:k$&$0$&$1$&$2$&$3$&$4$&$5$\\
\hline0&$1$&$0$&$0$&0&$\tfrac{-437}{1105}$&$\tfrac{-62}{955}$&$\tfrac{2}{1085}$&$\tfrac{-1}{58355}$&                                                            0&$\tfrac{-5044}{847}$&$\tfrac{1487}{645}$&$\tfrac{-71}{409}$&$\tfrac{-13}{1337}$&$\tfrac{1}{506}$&$\tfrac{-1}{8085}$\\
[.5ex]1&$0$&$0$&$0$&1&$\tfrac{-2087}{576}$&$\tfrac{-5}{51}$&$\tfrac{20}{1579}$&$\tfrac{-1}{5179}$&                                                              1&$\tfrac{-4393}{741}$&$\tfrac{6801}{1360}$&$\tfrac{-274}{345}$&$\tfrac{17}{1879}$&$\tfrac{7}{1182}$&$\tfrac{-1}{1819}$\\
[.5ex]2&$\tfrac{-7713}{866}$&$\tfrac{-1}{2}$&$0$&2&$\tfrac{4181}{4276}$&$\tfrac{484}{1351}$&$\tfrac{55}{1971}$&$\tfrac{-1}{1137}$&                              2&$\tfrac{-1341}{466}$&$\tfrac{5716}{1397}$&$\tfrac{-3021}{2443}$&$\tfrac{61}{743}$&$\tfrac{13}{2235}$&$\tfrac{-1}{954}$\\
[.5ex]3&$\tfrac{-9}{138082}$&$0$&$0$&3&$\tfrac{1877}{1212}$&$\tfrac{1333}{2971}$&$\tfrac{11}{873}$&$\tfrac{-2}{993}$&                                           3&$\tfrac{-1146}{1087}$&$\tfrac{8149}{4448}$&$\tfrac{-3563}{3695}$&$\tfrac{27}{202}$&$\tfrac{-1}{1919}$&$\tfrac{-1}{928}$\\
[.5ex]4&$\tfrac{6399}{284}$&$\tfrac{1863}{1255}$&$\tfrac{81}{1943}$&4&$\tfrac{-17991}{3070}$&$\tfrac{-269}{1211}$&$\tfrac{-23}{975}$&$\tfrac{-2}{883}$&         4&$\tfrac{-233}{1147}$&$\tfrac{1135}{1954}$&$\tfrac{-529}{1189}$&$\tfrac{49}{446}$&$\tfrac{-2}{315}$&$\tfrac{-1}{1748}$\\
[.5ex]5&$\tfrac{-81}{1906}$&$\tfrac{-243}{285199}$&$0$&5&$\tfrac{13329}{2014}$&$\tfrac{479}{1461}$&$\tfrac{-14}{897}$&$\tfrac{-1}{1296}$&                       5&$\tfrac{-123}{1744}$&$\tfrac{105}{737}$&$\tfrac{-187}{1339}$&$\tfrac{73}{1329}$&$\tfrac{-15}{2236}$&$\tfrac{-1}{35688}$\\
[.5ex]6&$\tfrac{-106434}{1721}$&$\tfrac{-3645}{1046}$&$\tfrac{-243}{3281}$&6&$\tfrac{-2131}{1574}$&$\tfrac{-12}{571}$&$\tfrac{-1}{456}$&$\tfrac{1}{1721}$&      6&$\tfrac{1}{1122}$&$\tfrac{15}{548}$&$\tfrac{-32}{937}$&$\tfrac{11}{595}$&$\tfrac{-11}{2839}$&$\tfrac{1}{5193}$\\
[.5ex]7&$\tfrac{-4374}{973}$&$\tfrac{-2187}{42902}$&$0$&7&$\tfrac{-9309}{1106}$&$\tfrac{-2251}{4727}$&$\tfrac{-10}{1213}$&$\tfrac{1}{2268}$&                    7&$\tfrac{-7}{1443}$&$\tfrac{6}{1339}$&$\tfrac{-1}{148}$&$\tfrac{13}{2807}$&$\tfrac{-1}{681}$&$\tfrac{1}{6286}$\\
[.5ex]8&$\tfrac{164025}{842}$&$\tfrac{6561}{697}$&$\tfrac{6561}{38831}$&8&$\tfrac{13085}{854}$&$\tfrac{191}{233}$&$\tfrac{28}{2249}$&$\tfrac{1}{8134}$&         8&$\tfrac{1}{599}$&$\tfrac{1}{1378}$&$\tfrac{-1}{924}$&$\tfrac{1}{1062}$&$\tfrac{-1}{2504}$&$\tfrac{1}{13765}$\\
[.5ex]9&$\tfrac{-59049}{1054}$&$\tfrac{6561}{2810}$&$\tfrac{6561}{163655}$&9&$\tfrac{-4249}{360}$&$\tfrac{-419}{750}$&$\tfrac{-13}{1541}$&$\tfrac{-1}{20568}$&  9&$\tfrac{1}{1378}$&$\tfrac{1}{9217}$&$\tfrac{-1}{6252}$&$\tfrac{1}{6329}$&$\tfrac{-1}{11794}$&$\tfrac{1}{44200}$\\
[.5ex]10&$\tfrac{-472392}{913}$&$\tfrac{-19683}{458}$&$\tfrac{-59049}{79175}$&10&$\tfrac{336}{517}$&$\tfrac{-136}{915}$&$\tfrac{-2}{639}$&$\tfrac{-1}{69993}$&  10&$\tfrac{-1}{1496}$&$\tfrac{-1}{48504}$&$\tfrac{-1}{51815}$&$\tfrac{1}{44446}$&$\tfrac{-1}{67140}$&$\tfrac{1}{190642}$\\
[.5ex]11&$\tfrac{236196}{347}$&$\tfrac{177147}{2930}$&$\tfrac{177147}{170006}$&11&$\tfrac{6443}{1141}$&$\tfrac{621}{1240}$&$\tfrac{9}{1033}$&$\tfrac{1}{15389}$&11&$\tfrac{-1}{547}$&$\tfrac{-1}{10720}$&$\tfrac{-1}{264802}$&$\tfrac{1}{356662}$&$\tfrac{-1}{473695}$&$0$\\
[.5ex]12&$\tfrac{-531441}{2569}$&$\tfrac{-531441}{19307}$&$0$&12&$\tfrac{-3106}{1039}$&$\tfrac{-429}{1775}$&$\tfrac{-5}{1209}$&$\tfrac{-1}{31521}$&             12&$\tfrac{-1}{873}$&$\tfrac{-1}{17714}$&$\tfrac{-1}{834241}$&$0$&$0$&$0$\\
[.5ex]\hline\end{tabular}


\begin{thebibliography}{10}

\bibitem{BJ} Bates, P.\ W., Jones, C.\ K.\ R.\ T.\  {\em Invariant manifolds for semilinear partial differential equations. } Dynamics reported, Vol.~2,  1--38,
Dynam.\ Report.\ Ser.\ Dynam.\ Systems Appl., 2, Wiley, Chichester, 1989.

\bibitem{BLZ} Bates, P.,  Lu, K.,  Zeng, C. {\em  Existence and persistence of invariant manifolds for semiflows in Banach space.} Mem.\ Amer.\ Math.\ Soc.\ 135 (1998), no.~645;
{\em Persistence of overflowing manifolds for semiflow.} Comm.\ Pure Appl.\ Math.\ 52 (1999), no.~8, 983--1046; {\em Approximately invariant manifolds and global dynamics of spike states.} Invent.\ Math.\ 174 (2008), no.~2, 355--433.

\bibitem{Bec2} Beceanu, M.  {\em  A Critical Centre-Stable Manifold for the Schroedinger Equation in Three Dimensions}, preprint 2009,  to appear
in Comm.\ Pure and Applied Math.

\bibitem{BerCaz} Berestycki, H.,  Cazenave, T. {\em
Instabilit\'e des \'etats stationnaires dans les \'equations de Schr\"odinger et de Klein--Gordon non lin\'eaires.}
C.\ R.\ Acad.\ Sci.\ Paris S\'er.~I Math.~293 (1981), no.~9, 489--492.

\bibitem{BerLions} Berestycki, H., Lions, P.-L. {\em Nonlinear scalar field equations. I. Existence of a ground state. } Arch.\ Rational Mech.\ Anal.~82  (1983), no.~4, 313--345;
{\em Nonlinear scalar field equations. II. Existence of infinitely many solutions. } Arch.\ Rational Mech.\ Anal.~82  (1983), no.~4, 347--375.

\bibitem{BP1} Buslaev, V.\ S., Perelman, G.\ S. {\em Scattering for the
nonlinear Schr\"odinger equation: states that are close to a soliton.} (Russian)
Algebra i Analiz  4  (1992),  no.~6, 63--102;  translation in  St.\ Petersburg Math.\ J.~4
(1993),  no.~6, 1111--1142.

\bibitem{BP2} Buslaev, V.\ S., Perelman, G.\ S. {\em On the stability of solitary waves
for nonlinear Schr\"odinger equations.  Nonlinear evolution equations,}  75--98,
Amer.\ Math.\ Soc.\ Transl.\ Ser.~2, 164, Amer.\ Math.\ Soc., Providence, RI, 1995.

\bibitem{Caz}
Cazenave, T.
\textit{Semilinear Schr\"odinger equations}.
Courant Lecture Notes in Mathematics, 10. New York University,
Courant Institute of Mathematical Sciences, New York;
American Mathematical Society, Providence, RI, 2003.

\bibitem{CazLions}
Cazenave, T.,  Lions, P.-L. {\em Orbital stability of standing waves for some nonlinear Schr\"odinger equations.}  Comm.\ Math.\ Phys.\  85  (1982), no.~4, 549--561.


\bibitem{Coff} Coffman, C. {\em Uniqueness of the ground state solution for $\Delta u-u+u^{3}=0$ and a variational characterization of other solutions.}
  Arch. Rational Mech. Anal.  46  (1972), 81--95.

\bibitem{Cucc} Cuccagna, S. {\em Stabilization of solutions to nonlinear Schr\"odinger equations.}  Comm.\ Pure Appl.\ Math.\  54  (2001),  no.~9, 1110--1145.


\bibitem{DS} Demanet, L., Schlag, W.  {\em
Numerical verification of a gap condition for a linearized nonlinear Schr\"odinger equation.}
Nonlinearity {\bf 19} (2006), no. 4, 829--852.

\bibitem{GJLS}
Gesztesy, F., Jones, C.\ K.\ R.\ T., Latushkin, Y., Stanislavova, M. {\em
A spectral mapping theorem and invariant manifolds for nonlinear Schr\"odinger equations.}
Indiana Univ.\ Math.\ J.~49 (2000), no.~1, 221--243.

\bibitem{GNN} Gidas, B., Ni, Wei Ming, Nirenberg, L. {\em Symmetry and related properties via the maximum principle.}  Comm.\ Math.\ Phys.~68  (1979), no.~3, 209--243.


\bibitem{Glassey} Glassey, R.\ T.\  \textit{On the blowing up of solutions
to the Cauchy problem for nonlinear Schr\"odinger equation},
J.\ Math.\ Phys., 18, 1977, 9, pp.\ 1794--1797.

\bibitem{GSS} Grillakis, M.,  Shatah, J.,  Strauss, W. {\em Stability theory of solitary waves in the presence of symmetry. I.}  J.\ Funct.\ Anal. {\bf 74}  (1987),  no.~1, 160--197; {\em Stability theory of solitary waves in the presence of symmetry. II.}  J.\ Funct.\ Anal.\ {\bf 94}  (1990),  no. 2, 308--348.

\bibitem{KrS1}  Krieger, J., Schlag, W. {\em Stable manifolds for all monic supercritical focusing nonlinear Schr\"odinger equations in one dimension. } J.\ Amer.\ Math.\ Soc. {\bf 19}  (2006),  no.~4, 815--920.

\bibitem{McLeod} McLeod, K. {\em Uniqueness of positive radial solutions of $\Delta u+f(u)=0$  in ${\bf R}^n$}. II. Trans.\ Amer.\ Math.\ Soc.\ 339 (1993), no.~2, 495--505.

\bibitem{NakS}  Nakanishi, K., Schlag, W.  {\em Invariant manifolds and dispersive Hamiltonian evolution equations.} To appear in ``Z\"urich Lectures in Advanced Mathematics'', Publishing House of the European Math.\ Society, 2011.

\bibitem{PS} Payne, L.\ E., Sattinger, D.\ H.\ {\em
Saddle points and instability of nonlinear hyperbolic equations.}
Israel J.\ Math.\ {\bf 22} (1975), no. 3-4, 273--303.

\bibitem{P} Perelman, G.   {\em On the formation of singularities in solutions of the critical nonlinear Schr\"odinger equation.}  Ann.\ Henri Poincar\'e  2  (2001),
no.~4, 605--673.


\bibitem{S} Schlag, W. {\em Stable manifolds for an orbitally unstable nonlinear Schr\"odinger equation.}  Ann.\ of Math.\ (2) {\bf 169}  (2009),  no.~1, 139--227.

\bibitem{Sha85} Shatah, J. {\em Unstable ground state of nonlinear Klein--Gordon equations.}  Trans.\ Amer.\ Math.\ Soc.\  290  (1985),  no.\ 2, 701--710.



\bibitem{SofWei1} Soffer, A., Weinstein, M.
{\em Multichannel nonlinear scattering for nonintegrable equations.}
Comm.\ Math.\ Phys.~133 (1990), 119--146;
{\em Multichannel nonlinear scattering, II. The case of anisotropic
potentials and data.} J.\ Diff.\ Eq.~98 (1992), 376--390.


\bibitem{StefStan} Stanislavova, M.,  Stefanov, A. {\em On precise center stable manifold theorems for certain reaction-diffusion and Klein-Gordon equations.} Phys.\ D 238 (2009), no.~23-24, 2298--2307.


\bibitem{Strauss77} Strauss, W.\ A. {\em Existence of solitary waves in higher dimensions.}  Comm.\ Math.\ Phys.\  55  (1977), no.~2, 149--162.

\bibitem{Strauss} Strauss, W.\ A. {\em Nonlinear wave equations.} CBMS Regional Conference Series in Mathematics, 73. Published for the Conference Board of the Mathematical Sciences, Washington, DC; by the American Mathematical Society, Providence, RI,  1989.


\bibitem{SS99} Sulem, C., Sulem, P-L.
\textit{The nonlinear Schr\"odinger equation. Self-focusing and wave collapse},
Applied Mathematical Sciences, 139. Springer-Verlag, New York, 1999.

\bibitem{Tao} Tao, T. {\em Nonlinear dispersive equations.}
Local and global analysis. CBMS Regional Conference Series in Mathematics, 106. AMS Providence, RI, 2006.

\bibitem{Wein} Weinstein, M.~I. {\em Modulational stability of ground states
of nonlinear Schr\"odinger equations,} SIAM J.\ Math.\ Anal.\ {\bf 16} (1985), no.~3, 472--491.

\bibitem{Wein2} Weinstein, M.  {\em  Lyapunov stability of ground states of nonlinear dispersive evolution equations.}  Comm.\ Pure Appl.\ Math.\  39  (1986),  no.~1, 51--67.


\end{thebibliography}
\end{document}